\documentclass[reqno]{amsart}
\usepackage[utf8x]{inputenc}
\usepackage{todonotes}

\usepackage{amsthm,amsmath,amssymb,amstext,amsfonts,color,tabu,stmaryrd}%latexsym}
\usepackage{fullpage}

\allowdisplaybreaks

\usepackage{mathtools,tabularx}

\usepackage{enumitem}

\usepackage{graphicx,float}

\newcommand{\field}[1]{\mathbb{#1}}

\newcommand{\N}{\field{N}}
\newcommand{\Z}{\field{Z}}

\numberwithin{equation}{section}
\newtheorem{theorem}{Theorem}[section]
\newtheorem{lemma}[theorem]{Lemma}
\newtheorem{corollary}[theorem]{Corollary}
\newtheorem{proposition}[theorem]{Proposition}

\theoremstyle{definition}
\newtheorem{property}[theorem]{Property}
\newtheorem{defprop}[theorem]{Definition/Proposition}
\newtheorem{definition}[theorem]{Definition}
\newtheorem{example}[theorem]{Example}

\theoremstyle{remark}
\newtheorem*{remark}{Remark}

\renewenvironment{proof}[1][Proof]{\begin{trivlist}
\item[\hskip \labelsep {\bfseries #1:}]}{\qed\end{trivlist}}

\newcommand{\B}{\mathcal{B}}

\title[Generalisations of Capparelli's and Primc's identities: I]{Generalisations of Capparelli's and Primc's identities, I:
\\coloured Frobenius partitions and combinatorial proofs}

\author{Jehanne Dousse}
\address{Univ Lyon, CNRS, Université Claude Bernard Lyon 1, UMR5208, Institut Camille Jordan, F-69622 Villeurbanne, France}
\email{dousse@math.cnrs.fr}

\author{Isaac Konan}
\address{IRIF, Université de Paris, Bâtiment Sophie Germain, Case courrier 7014, 8 Place Aurélie Nemours, 75205 Paris Cedex 13, France}
\email{konan@irif.fr}

\begin{document}

\begin{abstract}
The partition identities of Capparelli and Primc were originally discovered via representation theoretic techniques, and have since then been studied and refined combinatorially, but the question of giving a very broad generalisation remained open.
In these two companion papers, we give infinite families of partition identities which generalise Primc's and Capparelli's identities, and study their consequences on the theory of crystal bases of the affine Lie algebra $A_{n-1}^{(1)}.$ 

\smallskip
In this first paper, we focus on combinatorial aspects. We give a  $n^2$-coloured generalisation of Primc's identity by constructing a $n^2 \times n^2$ matrix of difference conditions, Primc's original identities corresponding to $n=2$ and $n=3$. While most coloured partition identities in the literature connect partitions with difference conditions to partitions with congruence conditions, in our case, the natural way to generalise these identities is to relate partitions with difference conditions to coloured Frobenius partitions. This gives a very simple expression for the generating function. With a particular specialisation of the colour variables, our generalisation also yields a partition identity with congruence conditions.

Then, using a bijection from our new generalisation of Primc's identity, we deduce a large family of identities on $(n^2-1)$-coloured partitions which generalise Capparelli's identity, also in terms of coloured Frobenius partitions.
The particular case $n=2$ is Capparelli's identity and one of the cases where $n=3$ recovers an identity of Meurman and Primc.

\smallskip
In the second paper, we will focus on crystal theoretic aspects. We will show that the difference conditions we defined in our $n^2$-coloured generalisation of Primc's identity are actually energy functions for certain $A_{n-1}^{(1)}$ crystals. We will then use this result to retrieve the Kac-Peterson character formula and derive a new character formula as a sum of infinite products for all the irreducible highest weight $A_{n-1}^{(1)}$-modules of level~$1.$ 
\end{abstract}

\maketitle

\section{Introduction and statement of results}

\subsection{Partition identities from representation theory}

\subsubsection{The Rogers-Ramanujan identities}
A \emph{partition} $\lambda$ of a positive integer $n$ is a non-increasing sequence of natural numbers $(\lambda_1,\dots,\lambda_s)$ whose sum is $n$. We write it as the sum $\lambda_1+\dots+\lambda_s$. The numbers $\lambda_1,\dots,\lambda_s$ are called the \emph{parts} of $\lambda$, the number $\ell(\lambda)=s$ is the \emph{length } of $\lambda$, and $|\lambda|=n$ is the \emph{weight} of $\lambda$. For example, the partitions of $4$ are $4, 3+1, 2+2, 2+1+1,$ and $ 1+1+1+1.$

The most famous partition identities are probably the Rogers--Ramanujan identities \cite{RogersRamanujan}. Using the standard $q$-series notation for $n \in \N \cup \{\infty\},$
$$(a;q)_n := (1-a)(1-aq)\cdots(1-aq^{n-1}),$$ they can be stated as follows.
\begin{theorem}[Rogers 1894, Ramanujan 1913]
\label{th:RR}
Let $i=0$ or $1$. Then
\begin{equation} \label{eq:rr}
\sum_{n \geq 0} \frac{q^{n^2+ (1-i)n}}{(q;q)_n} = \frac{1}{(q^{2-i};q^5)_{\infty}(q^{3+i};q^5)_{\infty}}.
\end{equation} 
\end{theorem}
By interpreting both sides of \eqref{eq:rr} as generating functions for partitions, MacMahon \cite{MacMahon} gave the following combinatorial version of the identities.
\begin{theorem}[Rogers--Ramanujan identities, partition version]
\label{th:RRcomb}
Let $i=0$ or $1$. For every natural number $n$, the number of partitions of $n$ such that the difference between two consecutive parts is at least $2$ and the part $1$ appears at most $i$ times is equal to the number of partitions of $n$ into parts congruent to $\pm 2-i \mod 5.$
\end{theorem}

More generally, a partition identity of the Rogers-Ramanujan type is a theorem stating that for all $n$, the number of partitions of $n$ satisfying some difference conditions equals the number of partitions of $n$ satisfying some congruence conditions.
Dozens of proofs of these identities have been given, using different techniques, see for example~\cite{Abook,Bressoud83,GarsiaMilne,Watson29}. But the starting point of our discussion is a representation theoretic proof due to Lepowsky and Wilson \cite{Lepowsky,Lepowsky2}. 

First, Lepowsky and Milne \cite{Le-Mi1,Le-Mi2} noticed that the product side of the Rogers-Ramanujan identities \eqref{eq:rr} multiplied by the ``fudge factor'' $1/(q;q^2)_{\infty}$ is equal to the principal specialisation of the Weyl-Kac character formula for level $3$ standard modules of the affine Lie algebra $A_1^{(1)}.$

Then, Lepowsky and Wilson \cite{Lepowsky,Lepowsky2} gave an interpretation of the sum side by constructing a basis of these standard modules using vertex operators. 
Very roughly, they proceed as follows. They start with a spanning set of such a module, indexed by monomials of the form $Z_1^{f_1} \dots Z_s^{f_s}$ for $s, f_1, \dots , f_s \in \N$. Then by the theory of vertex operators, there are some relations between these monomials. This allows them to reduce the spanning set by removing the monomials containing $Z_{j}^2$ and $Z_{j} Z_{j+1}$. The last step is then to prove that this reduced family of monomials is actually free, and therefore a basis of the representation. The connection to Theorem \ref{th:RR} is then done by noting that monomials $Z_1^{f_1} \dots Z_s^{f_s}$ which do not contain $Z_{j}^2$ or $Z_{j} Z_{j+1}$ for any $j$ are in bijection with partitions which do not contain twice the part $j$ or both the part $j$ and $j+1$ for any $j$, i.e. partitions with difference at least $2$ between consecutive parts.

The theory of vertex operator algebras developed by Lepowsky and Wilson turned out to be very influential: for example, it was used by Frenkel, Lepowsky, and Meurman to construct a natural representation of the Monster finite simple group \cite{FLM}, and was key in the work of Borcherds on vertex algebras and his resolution of the Conway-Norton monstrous moonshine conjecture \cite{Borcherds}.

\subsubsection{Capparelli's identity}
Following Lepowsky and Wilson's discovery, several other representation theorists studied other Lie algebras or representations at other levels, and discovered new interesting and intricate partition identities, that were previously unknown to combinatorialists, see for example \cite{Capparelli,Meurman,Meurman2,Meurman3,Nandi,Primc1,PrimcSikic,Siladic}, 

After Lepowsky and Wilson's work,  Capparelli \cite{Capparelli} was the first to conjecture a new identity, by studying the level $3$ standard modules of the twisted affine Lie algebra $A_2^{(2)}$. It was first proved combinatorially by Andrews in \cite{Andrewscap}, then refined by Alladi, Andrews and Gordon in \cite{AllAndGor} using the method of weighted words, and finally proved by Capparelli \cite{Capparelli2} and Tamba and Xie \cite{Xie} via representation theoretic techniques. Later, Meurman and Primc  \cite{Meurman2} showed that Capparelli's identity can also be obtained by studying the $(1,2)$-specialisation of the character formula for the level $1$ modules of $A_1^{(1)}$.
Capparelli's original identity can be stated as follows.

\begin{theorem}[Capparelli's identity (Andrews 1992)]
\label{th:capa}
Let $C(n)$ denote the number of partitions of $n$ into parts $>1$ such that parts differ by at least $2$, and at least $4$ unless consecutive parts add up to a multiple of $3$.  Let $D(n)$ denote the number of partitions of $n$ into distinct parts not congruent to $\pm 1 \pmod{6}$. Then for every positive integer $n$, $C(n) = D(n)$.   
\end{theorem}

In this paper, we will mostly be interested in the weighted words version of Theorem \ref{th:capa}, which was obtained by Alladi, Andrews and Gordon in \cite{AllAndGor}.
The principle of the method of weighted words, introduced by Alladi and Gordon to refine Schur's identity\cite{Alladi}, is to give an identity on coloured partitions, which under certain transformations on the coloured parts, becomes the original identity. We now describe Alladi, Andrews, and Gordon's refinement of Capparelli's identity (slightly reformulated by the first author in \cite{DousseCapaPrimc}).

Consider partitions into natural numbers in three colours, $a$, $c$, and $d$ (the absence of the colour $b$ will be made clear shortly, when we will mention the connection with Primc's identity),
with the order
\begin{equation}
\label{order}
1_a < 1_c < 1_d < 2_a < 2_c < 2_d < \cdots ,
\end{equation}  
satisfying the difference conditions in the matrix
\begin{equation} \label{Capmatrix}
C_2=\bordermatrix{\text{} & a & c & d \cr a & 2 & 2 & 2 \cr c & 1 & 1 & 2 \cr d & 0 & 1 & 2},
\end{equation} 
where the entry $(x,y)$ gives the minimal difference between consecutive parts of colours $x$ and $y$. 

The non-dilated version of Capparelli's identity can be stated as follows.
\begin{theorem}[Alladi--Andrews--Gordon 1995]
Let $C_2(n;i,j)$ denote the number of partitions of $n$ into coloured integers satisfying the difference conditions in matrix $C_2$, having $i$ parts coloured $a$ and $j$ parts coloured $d$.
\label{th:capa-dou}
We have
\begin{equation*}
\sum_{n,i,j \geq 0} C(n;i,j)a^i d^j q^n =  (-q)_{\infty}(-aq;q^2)_{\infty}(-dq;q^2)_{\infty}.
\end{equation*}
\end{theorem}
Performing the dilations 
$$q \rightarrow q^3, \quad a \rightarrow aq^{-1}, \quad d \rightarrow dq,$$
which correspond to the following transformations on the parts of the partitions:
$$k_a \rightarrow (3k-1)_a,  \quad k_c \rightarrow 3k, \quad k_d \rightarrow (3k+1)_d,$$
we obtain a refinement of Capparelli's original identity. Other dilations can lead to infinitely many other (but related) partition identities. Moreover, finding such refinements and non-dilated versions of partition identities can be helpful to find bijective proofs of them. For example, Siladi\'c's identity \cite{Siladic} was also discovered by using representation theory. Then, based on a non-dilated version of the theorem due to the first author \cite{Doussesil2}, the second author \cite{Konan1} was recently able to give a bijective proof and a broad generalisation of the identity.
For more on combinatorial refinements of partition identities, see for example \cite{Alladi,Alladi1,AllAndGor,AAB,Corteel,DousseFPSAC,DoussePrimc,Dousseunif,DousseCapaPrimc,Konan2} .

\subsubsection{Primc's identities}
Another way to obtain Rogers-Ramanujan type partition identities using representation theory is the theory of perfect crystals of affine Lie algebras. Much more detail on crystals is given in the second paper \cite{DK19-2} of this series, but the rough idea is the following. The generating function for partitions with congruence conditions, which is always an infinite product, is still obtained via a specialisation of the Weyl-Kac character formula. The equality with the generating function for partitions with difference conditions is established through the crystal base character formula of Kang, Kashiwara, Misra, Miwa, Nakashima, and Nakayashiki \cite{KMN2}. This formula expresses, under certain specialisations, the character as the generating function for partitions satisfying difference conditions given by energy matrices of perfect crystals.

The second identity which we study in this paper, due to Primc \cite{Primc}, was obtained in that way by studying crystal bases of $A_1^{(1)}.$
The energy matrix of the perfect crystal coming from the tensor product of the vector representation and its dual is given by
\begin{equation} \label{Primcmatrix4}
P_2=\bordermatrix{\text{} & a & b & c & d \cr a & 2&1&2&2 \cr b &1&0&1&1 \cr c &0&1&0&2 \cr d&0&1&0&2}.
\end{equation}
Let $P(n;i,j,k,\ell)$ denote the number of partitions of $n$ into four colours $a,b,c,d,$ with $i$ (resp. $j,k,\ell$) parts coloured $a$ (resp. $b,c,d$), satisfying the difference conditions of the matrix $P_2$. Then the crystal base character formula and the Weyl-Kac character formula imply that under the dilations
\begin{equation} 
\label{eq:Primcdilation}
k_{a} \rightarrow 2k-1, \quad k_{b} \rightarrow 2k, \quad k_{c} \rightarrow 2k, \quad k_{d} \rightarrow 2k+1,
\end{equation}
the generating function for these coloured partitions becomes $1/(q;q)_{\infty}.$
\begin{theorem}[Primc 1999]
\label{th:primc}
We have
$$
\sum_{n,i,j,k,\ell}P(n;i,j,k,\ell) q^{2n-i+\ell}= \frac{1}{(q;q)_{\infty}}, 
$$
\end{theorem}

By doing the same approach in the affine Lie algebra $A_2^{(1)},$ Primc also gave the following energy matrix (where we already use the colour names from our generalisation):
\begin{equation} \label{Primcmatrix9}
P_3=\bordermatrix{\text{} & a_2b_0 & a_2b_1 & a_1b_0 & a_0b_0 & a_2b_2 & a_1b_1 & a_0b_1 & a_1b_2 & a_0b_2 \cr a_2b_0 & 2&2&2&1&2&2&2&2&2 \cr a_2b_1 &1&2&1&1&2&1&2&2&2 \cr a_1b_0 &1&1&2&1&1&2&2&2&2 \cr a_0b_0 & 1&1&1&0&1&1&1&1&1 \cr a_2b_2 &0&0&1&1&0&1&1&2&2 \cr a_1b_1 &0&1&0&1&1&0&2&1&2 \cr a_0b_1 &0&1&0&1&1&0&2&1&2 \cr a_1b_2 &0&0&1&1&0&1&1&2&2 \cr a_0b_2 &0&0&0&1&0&0&1&1&2}.
\end{equation}

\begin{theorem}[Primc 1999]
\label{th:Primc9}
Under the dilations
\begin{equation} 
\label{eq:Primcdilation9}
\begin{array}{lll}
k_{a_2b_0} \rightarrow 3k-2, &\quad k_{a_2b_1} \rightarrow 3k-1, &\quad \quad k_{a_1b_0} \rightarrow 3k-1,\\
k_{a_0b_0} \rightarrow 3k, &\quad k_{a_1b_1} \rightarrow 3k, &\quad \quad k_{a_2b_2} \rightarrow 3k,\\
k_{a_0b_1} \rightarrow 3k+1, &\quad k_{a_1b_2} \rightarrow 3k+1, &\quad \quad k_{a_0b_2} \rightarrow 3k+2,\\
\end{array}
\end{equation}
the generating function for $9$-coloured partitions satisfying the difference conditions of \eqref{Primcmatrix9} becomes  $1/(q;q)_{\infty}.$
\end{theorem}

When seeing these two theorems of Primc, one might find it surprising that the generating function for partitions with such intricate difference conditions simply becomes $1/(q;q)_{\infty},$ the generating function for unrestricted partitions.
However recently, the first author and Lovejoy \cite{DoussePrimc} gave a weighted words version of Theorem \ref{th:primc}.
\begin{theorem}[Dousse-Lovejoy 2018, non-dilated version of Primc's identity]
\label{th:primcnondil}
Let $P(n;i,j,k,\ell)$ be defined as above. We have
$$\sum_{n,i,j,k,\ell} P(n;i,j,k,\ell) q^n a^i c^{k} d^{\ell} = \frac{(-aq;q^2)_{\infty}(-dq;q^2)_{\infty}}{(q;q)_{\infty}(cq;q^2)_{\infty}}.$$
\end{theorem}
Performing the dilations of \eqref{eq:Primcdilation} indeed transforms the infinite product above into $1/(q;q)_{\infty}.$ But this theorem shows that keeping track of all colours except $b$ leads to a more intricate infinite product as well, and that the extremely simple expression $1/(q;q)_{\infty}$ appears only because of the particular dilation that Primc considered. Later, the first author \cite{DousseCapaPrimc} even gave an expression for the generating function for $P(n;i,j,k,\ell)$ keeping track of all the colours, but it can be written as an infinite product only when one does not keep track of the colour $b$.

Thus it is interesting from a combinatorial point of view to see whether a similar phenomenon happens with Theorem \ref{th:Primc9} as well. To do so, we would like to compute the generating function for coloured partitions satisfying the difference conditions \eqref{Primcmatrix9}, at the non-dilated level, and keep track of as many colours as possible. In this paper, not only do we succeed to do this, but we embed both of Primc's theorems into an infinite family of identities about partitions satisfying difference conditions given by $n^2 \times n^2$ matrices.

\medskip

Apart from the fact that they can be obtained from the same Lie algebra $A_1^{(1)}$, Capparelli's and Primc's identities didn't seem related from the representation theoretic point of view, as they were obtained in completely different ways, and Capparelli's identity did not seem related to perfect crystals. However, recently, the first author \cite{DousseCapaPrimc} gave a bijection between coloured partitions satisfying the difference conditions \eqref{Primcmatrix4} and pairs of partitions $(\lambda,\mu)$, where $\lambda$ is a coloured partition satisfying the difference conditions \eqref{Capmatrix}, and $\mu$ is a partition coloured $b$. This bijection preserves the total weight, the number of parts, the sizes of the parts, and the number of parts coloured $a$ and $d$. Therefore, combinatorially, these two identities are very closely related.
In this paper, we generalise this bijection to our generalisation of Primc's identity and obtain a multi-parameter family of partition identities which generalise Capparelli's identity.

\newpage

\subsection{Statement of results}
\subsubsection{The difference conditions generalising Primc's identity}
In this paper, we give a family of partition identities with $n^2$ colours which generalises the two identities of Primc, and a multi-parameter family of partition identities with $n^2-1$ colours which generalise Capparelli's identity.

In a previous paper \cite{Konan1}, the second author gave a family of identities generalising Siladi\'c's identity using $n$ primary colours and $n^2$ secondary colours (products of two primary colours), giving $n^2+n$ colours in total. In \cite{Corteel}, Corteel and Lovejoy, gave a family of identities generalising Schur's theorem, later generalised by the first author to overpartitions \cite{Dousseunif}. These generalisations use $n$ primary colours, and products of at most $n$ different colours, giving $2^n-1$ colours in total.

Here, our generalisation uses only secondary colours, so we have $n^2$ colours in total.
Let us first define these colours and the corresponding difference conditions.
We start with two sequences of symbols $(a_{n})_{n\in\N}$ and $(b_{n})_{n\in\N}$, and use them to define two types of colours. 

\begin{definition}
The \emph{free colours} are the elements of the set $\{a_i b_i : i \in \N \}$, and the \emph{bound colours} are the elements of the set $\{a_i b_k : i \neq k, i,k \in \N \}$.
\end{definition} 

\begin{remark}
We choose these names because, to obtain our main theorems, we will set $b_i= a_i^{-1}$ for all $i$. In that case, the free colours will vanish, while the bound colours will have relations between them.
\end{remark}

In this paper, we consider partitions whose parts are coloured in free and bound colours, satisfying some difference conditions.
We now define these difference conditions, which generalise those of matrices \eqref{Primcmatrix4} and \eqref{Primcmatrix9} in the two identities of Primc.

\begin{definition}
For all $i,k,i',k' \in \N$, we define the minimal difference $\Delta$ between a part coloured $a_i b_k$ and a part coloured $a_{i'} b_{k'}$ in the following way:
\begin{equation}
\label{eq:Delta}
\Delta(a_i b_k, a_{i'} b_{k'}) = \chi(i \geq i') - \chi(i=k=i')+\chi(k \leq k') - \chi(k=i'=k'),
\end{equation}
where $\chi(prop)$ equals $1$ if the proposition $prop$ is true and $0$ otherwise.
\end{definition}

We start by observing some basic properties of $\Delta$ (the proofs, which are straightforward applications of the definition, are left to the reader).
\begin{property}
For all $i,k,i',k' \in \N$, $\Delta(a_i b_k, a_{i'} b_{k'})$ belongs to $\{0,1,2\}$.
\end{property}

\begin{property}
\label{property:aibiaibi}
For all $i\in \N$, we have $\Delta(a_i b_i, a_i b_i)=0$. In other words, free colours can repeat arbitrarily many times.
\end{property}

\begin{property}
\label{property:aibiakbk}
For all $i,k \in \N$ such that $i \neq k$, we have $\Delta(a_i b_i, a_k b_k)=1$.
\end{property}

\begin{property}[Triangle inequality]
\label{property:triangle}
Let $i,k,i',k'\in \N$. For all $i'',k'' \in \N$, we have
$$\Delta(a_i b_k, a_{i'} b_{k'}) \leq \Delta(a_i b_k, a_{i''} b_{k''}) + \Delta(a_{i''} b_{k''}, a_{i'} b_{k'}).$$
In other words, it is equivalent to say that $\Delta(a_i b_k, a_{i'} b_{k'})$ is the minimal difference between parts coloured $a_i b_k$ and $a_{i'} b_{k'}$, and that it is the minimal difference between \emph{consecutive} parts coloured $a_i b_k$ and $a_{i'} b_{k'}$.
\end{property}
By Properties \ref{property:aibiakbk} and \ref{property:triangle}, a part of a given size cannot appear in two different free colours.

\medskip

For every positive integer $n$, we define $\mathcal{P}_n$ to be the set of partitions $\lambda_1 + \cdots + \lambda_s$, where each part has a colour chosen from $\{a_i b_k : 0 \leq i,k \leq n-1\}$, satisfying the difference conditions for all $j \in \{1,\dots,s-1\}$:
$$\lambda_j - \lambda_{j+1} \geq \Delta(c(\lambda_j),c(\lambda_{j+1})),$$
where for all $j$, $c(\lambda_j)$ is the colour of part $\lambda_j$.
Such partitions are called \textit{generalised Primc partitions}.

To simplify some calculations throughout the paper, we adopt the following convention. If $c_1, \dots, c_s$ is the colour sequence of the partition $\lambda_1 + \cdots + \lambda_s$, we add free colours $c_0=c_{s+1}= a_{\infty}b_{\infty}$ to both ends of the colour sequence. The difference conditions are, for all $i,k \in \N$,
$$\Delta(a_{\infty}b_{\infty},a_i b_k)=\Delta(a_i b_k,a_{\infty}b_{\infty})=1,$$
which is coherent with the definition \eqref{eq:Delta} of $\Delta$. We also assume that $\lambda_{s+1}=0$.

The difference conditions defining $\mathcal{P}_n$ generalise Primc's difference conditions matrices $P_2$ and $P_3$ in $\eqref{Primcmatrix4}$ and $\eqref{Primcmatrix9}$, as we shall see in the next two examples.

\begin{example}
If we set $a = a_1 b_0, b= a_0 b_0, c= a_1 b_1, d= a_0 b_1$, then $\mathcal{P}_2$ is exactly the set of partitions with difference conditions \eqref{Primcmatrix4} of Primc's $4$-coloured identity.
%\begin{equation}
%\label{tab:abcd}
%\begin{array}{|c|cc|}
%\hline
%_{b_i}\setminus^{a_i}&0&1\\
%\hline
%0&b&a\\
%1&d&c\\
%\hline
%\end{array}
%\end{equation}
For example,
\begin{align*}
\Delta(a,b) &= \Delta(a_1 b_0, a_0 b_0)\\
&= \chi(1 \geq 0) - \chi(1=0=0) +\chi(0 \leq 0) - \chi(0=0=0)\\
&= 1-0+1-1\\
&=1.
\end{align*}
This  is exactly the entry in row $a$ and column $b$ in \eqref{Primcmatrix4}.
\end{example}

\begin{example}
The set $\mathcal{P}_3$ is exactly the set of partitions with difference conditions \eqref{Primcmatrix9} of Primc's $9$-coloured theorem.
For example,
\begin{align*}
\Delta(a_2 b_0, a_2 b_1)&= \chi(2 \geq 2) - \chi(2=0=2) +\chi(0 \leq 1) - \chi(0=2=1)\\
&= 1-0+1-0\\
&=2.
\end{align*}
This  is exactly the entry in row $a_2 b_0$ and column $a_2 b_1$ in \eqref{Primcmatrix9}.
\end{example}

It turns out that the matrix $(\Delta(a_kb_{\ell};a_{k'}b_{\ell'}))_{(k,\ell),(k',\ell') \in \{0, \dots, n-1\}^2}$ is an energy matrix for the crystal of the tensor product of the vector representation $\B$ of $A_{n-1}^{(1)}$ and its dual $\B^\vee$.
This is proved in our second paper \cite{DK19-2}. Using the formulas for the generating functions proved in this paper, it allows us to retrieve, for all $\ell\in\{0,\ldots,n-1\}$, the Kac-Peterson character formula, which expresses the characters $\mathrm{ch}(L(\Lambda_{\ell}))$ of the irreducible highest weight modules $L(\Lambda_{\ell})$ as a series in $\Z[[e^{-\delta},e^{\pm\alpha_1},\cdots,e^{\pm\alpha_{n-1}}]]$ with obviously positive coefficients, where the $\alpha_i$'s are the simple roots. Moreover, using Theorem \ref{th:main2}, we give the first expression for $\mathrm{ch}(L(\Lambda_{\ell}))$ as a sum of infinite products, also with obviously positive coefficients in $\Z[[e^{-\delta},e^{\pm\alpha_1},\cdots,e^{\pm\alpha_{n-1}}]]$.

%\begin{theorem}
%\label{thm:perfcrys}
% The crystal $\Bb=\B\ot\B^\vee$ is a perfect crystal of level 1 such that,
% for all $i\in \{0, \dots, n-1\}$, we have $b_{\Lambda_i}=b^{\Lambda_i} = v_i\ot \vv_i$, where $\Lambda_0, \dots \Lambda_{n-1}$ are the fundamental weights 
% Furthermore, the energy function
% on $\Bb\ot\Bb$ such that $H((v_0\ot\vv_0)\ot(v_0\ot\vv_0))=0$
% satisfies
% \begin{equation}\label{eq:equality}
%  H((v_{l'}\ot\vv_{k'})\ot(v_{l}\ot\vv_{k})) = \Delta(a_kb_l;a_{k'}b_{l'})\,,
% \end{equation}
%where $\Delta$ is the minimal difference \eqref{eq:primcdifference} defined for Primc generalized partitions. 
%\end{theorem}

\subsubsection{The difference conditions and forbidden patterns generalising Capparelli's identity}
In the previous section, we gave difference conditions which generalise those of Primc's identities \eqref{Primcmatrix4} and \eqref{Primcmatrix9}. In this section, we define a multi-parameter family of identities which generalises Capparelli's identity. These generalisations are expressed in terms of generalised Primc partitions avoiding some forbidden patterns.

\begin{definition}
Let $\pi=\pi_1+ \dots +\pi_r$ be a partition. We say that another partition $\lambda=\lambda_1+\cdots+\lambda_s$ \textit{contains the pattern} $\pi$ if there is some index $i$ such that
$$\lambda_i=\pi_1,\quad  \lambda_{i+1} = \pi_2, \quad \dots , \quad \lambda_{i+r-1}=\pi_r.$$
If $\lambda$ does not contain the pattern $\pi$, we say that $\lambda$ \textit{avoids} $\pi$.
\end{definition}

Let us start by defining some functions which will be parameters in our generalisations.

\begin{definition}
\label{def:cond1}
A function $\delta$ is said to satisfy Condition 1 if it is defined on the set of bound colours $\{a_k b_\ell : k \neq \ell\}$, has integer values, and for all $k,\ell$,
$$\min\{k,\ell\}<\delta(a_kb_\ell)\leq \max\{k,\ell\}.$$
\end{definition}

\begin{definition}
\label{def:cond2}
A function $\gamma$ is said to satisfy Condition 2 if it is defined on the set of pairs of bound colours $\{(a_{k_1} b_{\ell_1}, a_{k_2} b_{\ell_2}): k_1 \neq \ell_1, k_2 \neq \ell_2 \}$, has integer values, and if for all $k_1,k_2,\ell_1,\ell_2$, it satisfies the following:
\begin{itemize}
\item If $\max\{k_1,\ell_2\}<\min\{k_2,\ell_1\}$, we have 
$$\max\{k_1,\ell_2\} <\gamma(a_{k_1} b_{\ell_1}, a_{k_2} b_{\ell_2})\leq  \min\{k_2,\ell_1\}.$$
\item If $k_1>\ell_1$, $k_2>\ell_2$, and $\{\ell_2+1,\ldots,k_2\}\setminus \{\ell_1+1,\ldots,k_1\}\neq \emptyset$, we have
$$\gamma(a_{k_1} b_{\ell_1}, a_{k_2} b_{\ell_2})\in \{\ell_2+1,\ldots,k_2\}\setminus \{\ell_1+1,\ldots,k_1\}.$$
\item If $k_1<\ell_1$, $k_2<\ell_2$, and $\{k_1+1,\ldots,\ell_1\}\setminus\{k_2+1,\ldots,\ell_2\}\neq \emptyset$, we have 
$$\gamma(a_{k_1} b_{\ell_1}, a_{k_2} b_{\ell_2})\in \{k_1+1,\ldots,\ell_1\}\setminus\{k_2+1,\ldots,\ell_2\}.$$
\end{itemize}
\end{definition}

We now use these functions to define forbidden patterns and generalised Capparelli partitions.

\begin{definition}
\label{def:Capp_conditions}
Let $n$ be a positive integer, and let $\delta$ and $\gamma$ be functions satisfying Conditions 1 and 2, respectively. We define $\mathcal{C}_n(\delta,\gamma)$, the set of generalised Capparelli partitions related to $\delta$ and $\gamma$, to be the set of partitions $\lambda$ such that
\begin{itemize}
\item $\lambda \in \mathcal P_n$,
\item $\lambda$ has no part coloured $a_0b_0$,
\item $\lambda$ does not contain any of the following patterns, where $p$ is any positive integer:
\begin{itemize}
\item for any $i \in \{1, \dots , n-1\},$
$$p_{a_ib_i}+p_{a_ib_i},$$
(i.e. free colours cannot repeat)
\item for any $k_1,k_2,\ell_1,\ell_2$ such that $\max\{k_1,\ell_2\}<\min\{k_2,\ell_1\}$ and $i=\gamma(a_{k_1}b_{\ell_1},a_{k_2}b_{\ell_2})$,
$$p_{a_{k_1}b_{\ell_1}}+p_{a_ib_i}+p_{a_{k_2}b_{\ell_2}},$$
\item for any $k_2>\ell_2$,
\begin{itemize}
\item for any $2\leq u\leq \infty$, any $k_1, \ell_1$, and $i=\delta(a_{k_2}b_{\ell_2})$,
$$(p+u)_{a_{k_1}b_{\ell_1}}+p_{a_ib_i}+p_{a_{k_2}b_{\ell_2}},$$
(here we take the convention that $u=\infty$ if the pattern $p_{a_ib_i}+p_{a_{k_2}b_{\ell_2}}$ is at the beginning of the partition)
\item for any $k_1\leq \ell_1$, and $i=\delta(a_{k_2}b_{\ell_2})$,
$$(p+1)_{a_{k_1}b_{\ell_1}}+p_{a_ib_i}+p_{a_{k_2}b_{\ell_2}},$$
\item for any $k_1>\ell_1$ such that $\{\ell_2+1,\ldots,k_2\}\setminus \{\ell_1+1,\ldots,k_1\}\neq \emptyset$, and for $i=\gamma(a_{k_1}b_{\ell_1},a_{k_2}b_{\ell_2})$,
$$(p+1)_{a_{k_1}b_{\ell_1}}+p_{a_ib_i}+p_{a_{k_2}b_{\ell_2}},$$
\end{itemize} 
\item for any $k_1<\ell_1$,
\begin{itemize}
\item for any $2\leq u\leq \infty$, any $k_2, \ell_2$, and $i=\delta(a_{k_1}b_{\ell_1})$,
$$p_{a_{k_1}b_{\ell_1}}+p_{a_ib_i}+(p-u)_{a_{k_2}b_{\ell_2}},$$
(here we take the convention that $u=\infty$ if the pattern $p_{a_{k_1}b_{\ell_1}}+p_{a_ib_i}$ is at the end of the partition)
\item for any $k_2\geq \ell_2$, and $i=\delta(a_{k_1}b_{\ell_1})$,
$$(p+1)_{a_{k_1}b_{\ell_1}}+(p+1)_{a_ib_i}+p_{a_{k_2}b_{\ell_2}},$$
\item for any $k_2<l_2$ such that $\{k_1+1,\ldots,\ell_1\}\setminus\{k_2+1,\ldots,\ell_2\}\neq \emptyset$, and for $i=\gamma(a_{k_1}b_{\ell_1},a_{k_2}b_{\ell_2})$,
$$(p+1)_{a_{k_1}b_{\ell_1}}+(p+1)_{a_ib_i}+p_{a_{k_2}b_{\ell_2}}.$$
\end{itemize} 
\end{itemize}
\end{itemize}
\end{definition}

When $n=2$, there is only one possible choice for the functions $\delta$ and $\gamma$, and $\mathcal{C}_2:=\mathcal{C}_2(\delta,\gamma)$ is exactly the set of partitions with difference conditions from Capparelli's identity.

For general $n$, Definition \ref{def:Capp_conditions} is quite broad, but with some particular choices of functions $\delta$ and $\gamma$, the partitions in $\mathcal{C}_n(\delta,\gamma)$ can be described easily. We now describe two of these particular choices, one of which leads to a generalisation of an identity of Meurman--Primc \cite{Meurman3}.

\medskip

Let us start with our first example. For all $k\neq \ell$, we set
$$\delta_1(a_kb_\ell) = 1+\min\{k,\ell\}.$$
The function $\gamma_1$ is defined as follows:
\begin{itemize}
\item for $\max\{k_1,\ell_2\}<\min\{k_2,\ell_1\}$, we set 
$$
\gamma_1(a_{k_1} b_{\ell_1}, a_{k_2} b_{\ell_2})=1+\max\{k_1,\ell_2\},
$$
\item for $k_1>\ell_1$, $k_2>\ell_2$, such that $\{\ell_2+1,\ldots,k_2\}\setminus \{\ell_1+1,\ldots,k_1\}\neq \emptyset$, we set   
$$\gamma_1(a_{k_1} b_{\ell_1}, a_{k_2} b_{\ell_2})=
\begin{cases}
\ell_2+1 &\text{ if } \ell_2+1  \in \{\ell_2+1,\ldots,k_2\}\setminus \{\ell_1+1,\ldots,k_1\},\\
k_2 &\text{ otherwise,}
\end{cases}$$
\item for $k_1<\ell_1$, $k_2<\ell_2$, such that $\{k_1+1,\ldots,\ell_1\}\setminus\{k_2+1,\ldots,\ell_2\}\neq \emptyset$, we set   
$$\gamma_1(a_{k_1} b_{\ell_1}, a_{k_2} b_{\ell_2})=
\begin{cases}
k_1+1 &\text{ if } k_1+1  \in \{k_1+1,\ldots,\ell_1\}\setminus\{k_2+1,\ldots,\ell_2\},\\
\ell_1 &\text{ otherwise.}
\end{cases}$$
\end{itemize} 
The corresponding generalised Capparelli partitions are given by the following proposition (its verification is a simple application of the definition and is left to the interested reader).
\begin{proposition}
\label{prop:genMP}
The set $\mathcal{C}_n(\delta_1,\gamma_1)$ 
is the set of partitions $\lambda_1 + \cdots + \lambda_s$, where each part has a colour chosen from $\{a_i b_k : 0 \leq i,k \leq n-1, \ (i,k) \neq (0,0)\}$, satisfying for all $j \in \{1,\dots,s-1\}$ the difference conditions
$$\lambda_j - \lambda_{j+1} \geq \Delta_1(c(\lambda_j),c(\lambda_{j+1})),$$
where
\begin{equation}
\begin{cases}
\Delta_1(a_ib_i,a_ib_i)=1 &\text{ for all } i>0\\
\Delta_1(a_\ell b_\ell,a_k b_{\ell-1})=1 &\text{ for all } k\geq \ell>0\\
\Delta_1(a_{k-1}b_{\ell},a_kb_k)=1 &\text{ for all } \ell\geq k>0\\
\Delta_1(c_1,c_2)= \Delta(c_1,c_2) &\text{ otherwise,}
\end{cases}
\end{equation}
and which, for every positive integer $p$, avoid the forbidden patterns
$
(p+1)_{a_{k_1}b_{\ell_1}}+p_{a_{k_2}b_{k_2}}+p_{a_{k_2}b_{\ell_2}}
$
for all $k_2>k_1>\ell_2\geq \ell_1$,
and 
$
(p+1)_{a_{k_1}b_{\ell_1}}+(p+1)_{a_{\ell_1}b_{\ell_1}}+p_{a_{k_2}b_{\ell_2}}$
for all $\ell_1>\ell_2>k_1\geq k_2$.
\end{proposition}

As said above, when $n=2$, this reduces to the difference conditions of Capparelli's identity \eqref{Capmatrix}. But these conditions also generalise those of another partition identity mentioned in Primc's paper \cite{Primc}.

\begin{example}
Set $n=3$ in Proposition \ref{prop:genMP}.
The difference conditions defining $\mathcal{C}_3(\delta_1,\gamma_1)$ are shown in the following matrix, which appeared in Primc's paper \cite{Primc}.
\begin{equation} \label{Primcmatrix8}
C^1_3=\bordermatrix{\text{} & a_2b_0 & a_2b_1 & a_1b_0 & a_2b_2 & a_1b_1 & a_0b_1 & a_1b_2 & a_0b_2 \cr a_2b_0 & 2&2&2&2&2&2&2&2 \cr a_2b_1 &1&2&1&2&1&2&2&2 \cr a_1b_0 &1&1&2&1&2&2&2&2 \cr  a_2b_2  &0&1&1&1&1&1&2&2 \cr  a_1b_1 &1&1&1&1&1&2&1&2 \cr a_0b_1 &0&1&0&1&1&2&1&2 \cr a_1b_2 &0&0&1&1&1&1&2&2 \cr a_0b_2 &0&0&0&0&1&1&1&2}.
\end{equation}
Moreover, we have the forbidden patterns
$(p+1)_{a_1b_{0}}+p_{a_{2}b_{2}}+p_{a_2b_{0}}$ and $(p+1)_{a_0b_{2}}+(p+1)_{a_{2}b_{2}}+p_{a_0b_{1}}$ for all positive integers $p$, which again are exactly the forbidden patterns mentioned in \cite{Primc}.

It was proved by Meurman and Primc in \cite{Meurman3}, using basic $A_2^{(1)}$ modules, that after performing the dilations \eqref{eq:Primcdilation9}, the generating function for these partitions becomes $(q;q^3)_\infty^{-1}(q^2;q^3)_\infty^{-1}.$
\end{example}

\medskip
Let us now turn to the second choice of $\delta$ and $\gamma$ which gives rise to simple difference conditions and forbidden patterns.

For all $k\neq \ell$, we set
$$\delta_2(a_kb_\ell) = \max\{k,\ell\}.$$
The function $\gamma_2$ is defined as follows:
\begin{itemize}
\item for $\max\{k_1,\ell_2\}<\min\{k_2,\ell_1\}$, we set 
$$
\gamma_2(a_{k_1} b_{\ell_1}, a_{k_2} b_{\ell_2})=\min\{k_2,\ell_1\},
$$
\item for $k_1>\ell_1$, $k_2>\ell_2$, such that $\{\ell_2+1,\ldots,k_2\}\setminus \{\ell_1+1,\ldots,k_1\}\neq \emptyset$, we set   
$$\gamma_2(a_{k_1} b_{\ell_1}, a_{k_2} b_{\ell_2})=
\begin{cases}
k_2 &\text{ if } k_2 \in \{\ell_2+1,\ldots,k_2\}\setminus \{\ell_1+1,\ldots,k_1\},\\
\ell_2+1 &\text{ otherwise,}
\end{cases}$$
\item for $k_1<\ell_1$, $k_2<\ell_2$, such that $\{k_1+1,\ldots,\ell_1\}\setminus\{k_2+1,\ldots,\ell_2\}\neq \emptyset$, we set   
$$\gamma_2(a_{k_1} b_{\ell_1}, a_{k_2} b_{\ell_2})=
\begin{cases}
\ell_1 &\text{ if } \ell_1  \in \{k_1+1,\ldots,\ell_1\}\setminus\{k_2+1,\ldots,\ell_2\},\\
k_1+1 &\text{ otherwise.}
\end{cases}$$
\end{itemize} 
The corresponding generalised Capparelli partitions can be described as follows.
\begin{proposition}
\label{prop:Cn2}
The set $\mathcal{C}_n(\delta_2,\gamma_2)$ 
is the set of partitions $\lambda_1 + \cdots + \lambda_s$, where each part has a colour chosen from $\{a_i b_k : 0 \leq i,k \leq n-1, \ (i,k) \neq (0,0)\}$, satisfying for all $j \in \{1,\dots,s-1\}$ the difference conditions
$$\lambda_j - \lambda_{j+1} \geq \Delta_2(c(\lambda_j),c(\lambda_{j+1})),$$
where
\begin{equation}
\begin{cases}
\Delta_2(a_ib_i,a_ib_i)=1  &\text{ for all } i>0\\
\Delta_2(a_kb_k,a_kb_\ell)=1  &\text{ for all } k>\ell\geq 0\\
\Delta_2(a_kb_\ell,a_\ell b_\ell)=1  &\text{ for all } \ell>k\geq 0\\
\Delta_2(c_1,c_2)= \Delta(c_1,c_2)  &\text{ otherwise},
\end{cases}
\end{equation}
and which, for every positive integer $p$, avoid the forbidden patterns
$
(p+1)_{a_{k_1}b_{\ell_1}}+p_{a_{\ell_2+1}b_{\ell_2+1}}+p_{a_{k_2}b_{\ell_2}}
$
for all $k_1\geq k_2>\ell_1>\ell_2$,
and 
$
(p+1)_{a_{k_1}b_{\ell_1}}+(p+1)_{a_{k_1+1}b_{k_1+1}}+p_{a_{k_2}b_{\ell_2}}$
for all $\ell_2\geq \ell_1>k_2>k_1$.
\end{proposition}

Again, when setting $n=2$ in Proposition \ref{prop:Cn2}, we recover the difference conditions of Capparelli's identity. When setting $n=3$, we get new difference conditions similar to those of Meurman--Primc \cite{Meurman3}.

\begin{example}
The set $\mathcal{C}_3(\delta_2,\gamma_2)$  is the set of partitions satisfying the difference conditions given in the following matrix
\begin{equation} \label{Capgen8matrix}
C^2_3=\bordermatrix{\text{} & a_2b_0 & a_2b_1 & a_1b_0 & a_2b_2 & a_1b_1 & a_0b_1 & a_1b_2 & a_0b_2  \cr a_2b_0 & 2&2&2&2&2&2&2&2 \cr a_2b_1 &1&2&1&2&1&2&2&2 \cr a_1b_0  &1&1&2&1&2&2&2&2 \cr a_2b_2 &1&1&1&1&1&1&2&2 \cr a_1b_1 &0&1&1&1&1&2&1&2 \cr a_0b_1 &0&1&0&1&1&2&1&2 \cr a_1b_2 &0&0&1&1&1&1&2&2 \cr a_0b_2 &0&0&0&1&0&1&1&2},
\end{equation}
and avoiding the forbidden patterns
$(p+1)_{a_2b_{1}}+p_{a_{1}b_{1}}+p_{a_2b_{0}}$ and $(p+1)_{a_0b_{2}}+(p+1)_{a_{1}b_{1}}+p_{a_1b_{2}}$ for all positive integers $p$.
\end{example}

\medskip

Recently in \cite{DousseCapaPrimc}, the first author gave a bijection between Primc's partitions $\mathcal{P}_2$ and pairs $(\lambda, \mu)$ where $\lambda \in \mathcal{C}_2$ is a Capparelli partition and $\mu$ is a classical partition. This bijection only modifies some free colours, so it preserves the weight, the number of parts, the sizes of the parts, and the number of appearances of colours $a$ and $d$. In this way, she showed that Capparelli's identity is closely related to Primc's identity and can be deduced from it, even though until then, these two identities seemed unrelated from the representation theoretic point of view.

Here, we generalise this idea and show the following.

\begin{theorem}
\label{th:bij}
For all positive integers $n$ and all functions $\delta$ and $\gamma$  satisfying Conditions 1 and 2, respectively, there is a bijection $\Phi$ between the set $\mathcal{P}_n$ of generalised Primc partitions and the product set $\mathcal{C}_n(\delta,\gamma)\times \mathcal{P}^0$, where $\mathcal{C}_n(\delta,\gamma)$ is the set of generalised Capparelli partitions related to $\delta$ and $\gamma$, and $\mathcal{P}^0$ is the set of the classical partitions where all parts are coloured $a_0b_0$.

This bijection preserves the weight, the number of parts, the sizes of the parts, and the number of appearances of each bound colour.
\end{theorem}

Both Capparelli's identity and Meurman and Primc's identity with difference conditions \eqref{Primcmatrix8} did not have any apparent connection with the theory of perfect crystals. The bijection between $\mathcal{P}_2$ and $\mathcal{C}_2 \times \mathcal{P}^0$ in \cite{DousseCapaPrimc} gave an unexpected connection with Primc's identity and the theory of perfect crystals. The present theorem shows that Meurman and Primc's identity with difference conditions \eqref{Primcmatrix8} can actually be deduced from Primc's $9$-coloured Theorem \ref{th:Primc9}. More generally, through the bijection with the $\mathcal{P}_n$'s, we relate our family of generalisations of Capparelli's identity to the theory of perfect crystals.

The detailed bijections are given in Section \ref{sec:bij}.

\subsubsection{Coloured Frobenius partitions}
Since its discovery, Capparelli's identity has been one of the most studied partition identities in the literature, see for example \cite{Br-Ma1,Be-Un1,Be-Un2,DousseCapa,Fu-Ze1,Kanaderussellstaircase,Kursungoz,Si1} for articles from the combinatorial point of view.

While the other most important partition identities, such as the Rogers-Ramanujan identities \cite{RogersRamanujan} and Schur's theorem \cite{Schur}, have been successfully embedded in large families of identities, such as the Andrews-Gordon identities for Rogers-Ramanujan \cite{AndrewsGordon,Gordon65} and Andrews' theorems for Schur's theorem \cite{Generalisation1,Generalisation2}, finding such a broad generalisation of Capparelli's identity was still an open problem.

Here, we solve this problem by giving a multi-parameter family of identities which generalise Capparelli and a family of identities generalising Primc. Unlike most classical Rogers-Ramanujan type identities, we relate the partitions with difference conditions defined in the previous section to coloured Frobenius partitions. This allows us to find simple and elegant formulations for the generating functions.

Following Andrews \cite{AndrewsFrob}, a Frobenius partition is a two-rowed array
$$\begin{pmatrix}
\lambda_1 & \lambda_2 & \cdots & \lambda_s \\
\mu_1 & \mu_2 & \cdots & \mu_s
\end{pmatrix},$$
where $s$ is a non-negative integer and $\lambda :=\lambda_1 + \lambda_2 + \cdots + \lambda_s$ and $\mu:=\mu_1+\mu_2+\cdots+\mu_s$ are two partitions into $s$ distinct non-negative parts.
Frobenius partitions of length $s$ and weight $m=s+\sum_{i=1}^s \lambda_i+ \sum_{i=1}^s \mu_i$ are in bijection with partitions of $m$ whose Durfee square (the largest square fitting in the top-left corner of the Ferrers board of the partition) is of side $s$. An example can be seen on Figure \ref{fig:frob} in the case $s=4$ (where $\lambda_4=\mu_4=0$).
\begin{figure}[H]
\includegraphics[width=0.4\textwidth]{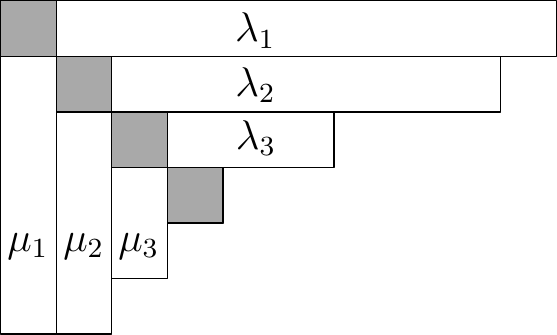}
\caption{A Frobenius partition of length $4$.}
\label{fig:frob}
\end{figure}
The generating function for the number $F(m)$ of Frobenius partitions of $m$ is given by
$$\sum_{m \geq 0} F(m) q^m = [x^0] (-xq;q)_{\infty}(-x^{-1};q)_{\infty} .$$
Indeed, the product $(-xq;q)_{\infty}$ generates the partition $\lambda$ together with the boxes on the diagonal where the power of $x$ counts the number of parts, $(-x^{-1};q)_{\infty}$ generates the partition $\mu$ where the power of $x^{-1}$ counts the number of parts, and taking the coefficient of $x^0$ in the above ensures that $\lambda$ and $\mu$ have the same number of parts. Using Jacobi's triple product identity (see, e.g., \cite{Abook}),
\begin{equation}
\label{eq:JTP}
(-xq;q)_{\infty}(-x^{-1};q)_{\infty} (q;q)_{\infty} = \sum_{k \in \Z} x^k q^{\frac{k(k+1)}{2}},
\end{equation}
we see that the generating function for Frobenius partitions equals $1/(q;q)_{\infty}$, the generating function for partitions.

\medskip
Let us now define the coloured Frobenius partitions which will be related to our coloured partitions with difference conditions.

In \cite[(4.8)]{AndrewsFrob}, Andrews defined a generalisation of Frobenius partitions where $\lambda$ and $\mu$ are partitions into distinct parts chosen from $\{k_j : k \in \N, 1 \leq j \leq n\},$ where $k_j = k'_{j'}$ if and only if $k=k'$ and $j=j'$. Their generating function $\mathrm{C}\Phi_k(q)$ has been widely studied from the point of view of modular forms and congruences, see for example \cite{ChanWangYang,Lov00,SellersFrob}.

Here we define a refinement of Andrews' partitions. Consider the same families of symbols $(a_{i})_{i\in\N}$ and $(b_{i})_{i\in\N}$ as in the previous section. We define a \textit{$n^2$-coloured Frobenius partition} to be a Frobenius partition
$$\begin{pmatrix}
\lambda_1 & \lambda_2 & \cdots & \lambda_s \\
\mu_1 & \mu_2 & \cdots & \mu_s 
\end{pmatrix},$$
where $\lambda =\lambda_1 + \lambda_2 + \cdots + \lambda_s$  is a partition into $s$ distinct non-negative parts, each coloured with some $a_i$, $i \in \{0, \dots, n-1\},$ with the following order
\begin{equation}
\label{eq:orderFroba}
0_{a_{n-1}} < 0_{a_{n-2}} < \cdots < 0_{a_0} < 1_{a_{n-1}} < 1_{a_{n-2}} < \cdots < 1_{a_0} < \cdots,
\end{equation}
and $\mu=\mu_1+\mu_2+\cdots+\mu_s$ is a partition into $s$ distinct non-negative parts, each coloured with some $b_i$, $i \in \{0, \dots, n-1\},$ with the order
\begin{equation}
\label{eq:orderFrobb}
0_{b_{0}} < 0_{b_1} < \cdots < 0_{b_{n-1}} <1_{b_{0}} < 1_{b_1} < \cdots < 1_{b_{n-1}} < \cdots .
\end{equation}
Let $\mathcal{F}_n$ denote the set of $n^2$-coloured Frobenius partitions.
Note that in $\lambda$ and $\mu$, a part of a given size can appear in different colours.
We define the \textit{colour sequence} of such a $n^2$-coloured Frobenius partition to be $(c(\lambda_1)c(\mu_1), \dots, c(\lambda_s)c(\mu_s))$.

\begin{example}
This array gives an example of $9$-coloured Frobenius partition with colour sequence

\noindent
$(a_1b_2,a_0 b_0, a_1 b_0, a_2b_1)$ and weight $18$:
$$\begin{pmatrix}
3_{a_1} & 2_{a_0} & 0_{a_1} & 0_{a_2} \\
4_{b_2} & 4_{b_0}& 1_{b_0} & 0_{b_1}
\end{pmatrix}.$$
\end{example}

Following the same reasoning as for classical Frobenius partitions, the generating function for the number $F_{n}(m;u_0,\dots,u_{n-1};v_0,\dots,v_{n-1})$ of $n^2$-coloured Frobenius partitions of $m$ where for $i \in \{0, \dots, n-1\},$ the symbol $a_i$ (resp. $b_i$) appears $u_i$ (resp. $v_i$) times, is
\begin{equation}
\label{eq:sgFrob}
\begin{aligned}
\sum_{m,u_0, \dots, u_{n-1}, v_0, \dots, v_{n-1} \geq 0} &F_{n}(m;u_0,\dots,u_{n-1};v_0,\dots,v_{n-1})q^m a_0^{u_0} \cdots a_{n-1}^{u_{n-1}} b_1^{v_1} \cdots b_{n-1}^{v_{n-1}} 
\\&= [x^0] \prod_{i=0}^{n-1}(-x a_i q;q)_{\infty}(-x^{-1} b_i;q)_{\infty}.
\end{aligned}
\end{equation}
This refines the following expression due to Andrews \cite[(5.14)]{AndrewsFrob}:
$$\mathrm{C}\Phi_k(q)= [x^0] (-xq;q)^n_{\infty}(-x^{-1};q)^n_{\infty},$$
where the colours were not taken into account in the generating function.

Note that the generating function \eqref{eq:sgFrob} does not depend on our orders \eqref{eq:orderFroba} and \eqref{eq:orderFrobb}, but only on the condition ``all parts are distinct'' in $\lambda$ and $\mu$. These particular orders will however be helpful to make the connection with the Primc generalised partitions $\mathcal{P}_n$ in the remainder of this paper.

\subsubsection{Generalisations of Capparelli and Primc's identities}
The $n^2$-coloured Frobenius partitions are very natural objects to consider when studying our generalisations of Primc and Capparelli's identities. 

\begin{theorem}[Connection between $\mathcal{P}_n$ and $\mathcal{F}_n$]
\label{th:PnFnbound}
Let $n$ be a positive integer.

\noindent
Let $\tilde{P}_{n}(m;u_0,\dots,u_{n-1};v_0,\dots,v_{n-1})$ be the number of $n^2$-coloured partitions of $m$ in colours $\{a_i b_k : 0 \leq i,k \leq n-1\}$, satisfying the difference conditions $\Delta$ (see \eqref{eq:Delta}), where for $i \in \{0, \dots, n-1\},$ the symbol $a_i$ (resp. $b_i$) appears $u_i$ (resp. $v_i$) times \textbf{in their bound colours}.

Let $\tilde{F}_{n}(m;u_0,\dots,u_{m-1};v_0,\dots,v_{m-1})$ be the number of $n^2$-coloured Frobenius partitions of $m$ where for $i \in \{0, \dots, n-1\},$ the symbol $a_i$ (resp. $b_i$) appears $u_i$ (resp. $v_i$) times \textbf{in their bound colours}.

Then
$$\tilde{P}_{n}(m;u_0,\dots,u_{n-1};v_0,\dots,v_{n-1})=\tilde{F}_{n}(m;u_0,\dots,u_{n-1};v_0,\dots,v_{n-1}).$$
\end{theorem} 

\begin{remark}
We actually prove a refinement of Theorem \ref{th:PnFnbound} according to the notion of \emph{reduced colour sequence} defined in Section \ref{sec:redcolseq}. This is given in Theorem \ref{th:Primcgenekernel}.
We do not state it in this introduction to avoid technicalities.
\end{remark}

Moreover, when we set for all $i$, $b_i = a_i^{-1}$, then all free colours vanish and we have the following elegant expression for our generating functions as the constant term in an infinite product.

\begin{theorem}[Generalisation of Primc's identity]
\label{th:Primcgene}
Let $n$ be a positive integer.

Let $P_{n}(m;u_0,\dots,u_{n-1};v_0,\dots,v_{n-1})$ be the number of partitions of $m$ in $\mathcal{P}_n$, such that for $i \in \{0, \dots, n-1\},$ the symbol $a_i$ (resp. $b_i$) appears $u_i$ (resp. $v_i$) times.

Let $F_{n}(m;u_0,\dots,u_{m-1};v_0,\dots,v_{m-1})$ be the number of $n^2$-coloured Frobenius partitions of $m$ where for $i \in \{0, \dots, n-1\},$ the symbol $a_i$ (resp. $b_i$) appears $u_i$ (resp. $v_i$) times.
We have
\begin{align*}
&\sum_{m,u_0, \dots, u_{n-1}, v_0, \dots, v_{n-1} \geq 0} P_{n}(m;u_0,\dots,u_{n-1};v_0,\dots,v_{n-1})q^m a_0^{u_0-v_0} \cdots a_{n-1}^{u_{n-1}-v_{n-1}} 
\\&=\sum_{m,u_0, \dots, u_{n-1}, v_0, \dots, v_{n-1} \geq 0} F_{n}(m;u_0,\dots,u_{n-1};v_0,\dots,v_{n-1})q^m a_0^{u_0-v_0} \cdots a_{n-1}^{u_{n-1}-v_{n-1}}
\\&= [x^0] \prod_{i=0}^{n-1}(-x a_i q;q)_{\infty}(-x^{-1} a_i^{-1};q)_{\infty}.
\end{align*}
\end{theorem}
Note that because the free colours vanish, the generating functions  for $P_{n}(m;u_0,\dots,u_{n-1};v_0,\dots,v_{n-1})$ and $\tilde{P}_{n}(m;u_0,\dots,u_{n-1};v_0,\dots,v_{n-1})$ (resp. $F_{n}(m;u_0,\dots,u_{n-1};v_0,\dots,v_{n-1})$ and $\tilde{F}_{n}(m;u_0,\dots,u_{n-1};v_0,\dots,v_{n-1})$) are the same.

From this theorem, it is easy to deduce a corollary, corresponding to the principal specialisation, which generalises both of Primc's original identities.
By performing the dilations $q \rightarrow q^n$, and for all $i \in \{0, \dots,n-1\}$, $a_i \rightarrow q^{-i},$ the generating function above becomes $[x^0] (-xq;q)_{\infty}(-x^{-1};q)_{\infty}$, which is also equal to $1/(q;q)_{\infty}.$

\begin{corollary}[Principal specialisation]
\label{cor:primcspec}
Let $n$ be a positive integer.
We have
\begin{align*}
&\sum_{m,u_0, \dots, u_{n-1}, v_0, \dots, v_{n-1} \geq 0} P_{n}(m;u_0,\dots,u_{n-1};v_0,\dots,v_{n-1})q^{nm -\sum_{i=0}^{n-1} i (u_i-v_i)}= \frac{1}{(q;q)_{\infty}}.
\end{align*}
\end{corollary}

Moreover, by using Jacobi's triple product repeatedly, we are able to give an expression of the generating function for coloured Frobenius partitions as a sum of infinite products, which gives yet another expression for the generating function for $\mathcal{P}_n$, or equivalently $\mathcal{F}_n$.
\begin{theorem}
\label{th:main2}
%Let $P_{n}(m;u_0,\dots,u_{n-1};v_0,\dots,v_{n-1})$ be the number of $n^2$-coloured partitions of $m$ in colours $\{a_i b_k : 0 \leq i,k \leq n-1\}$, satisfying the difference conditions $\Delta$ (see \eqref{eq:Delta}), where for $i \in \{0, \dots, n-1\},$ the symbol $a_i$ (resp. $b_i$) appears $u_i$ (resp. $v_i$) times in the colour sequence.
Let $n$ be a positive integer.
Then
\begin{align}
&\sum_{m,u_0, \dots, u_{n-1}, v_0, \dots, v_{n-1} \geq 0} P_{n}(m;u_0,\dots,u_{n-1};v_0,\dots,v_{n-1})q^m a_0^{u_0-v_0} \cdots a_{n-1}^{u_{n-1}-v_{n-1}} \nonumber
\\&=\sum_{m,u_0, \dots, u_{n-1}, v_0, \dots, v_{n-1} \geq 0} F_{n}(m;u_0,\dots,u_{n-1};v_0,\dots,v_{n-1})q^m a_0^{u_0-v_0} \cdots a_{n-1}^{u_{n-1}-v_{n-1}} \nonumber
\\&=\frac{1}{(q;q)_{\infty}^{n}}\sum_{\substack{s_1, \dots, s_{n-1}\in \Z\\s_n=0}} a_0^{-s_1}\prod_{i=1}^{n-1} a_i^{s_i-s_{i+1}}q^{s_i(s_i-s_{i+1})} \label{eq:jacob}
\\&=\frac{1}{(q;q)_{\infty}}\left(\prod_{i=1}^{n-1}\frac{\left(q^{i(i+1)};q^{i(i+1)}\right)_{\infty}}{(q;q)_{\infty}}\right) \sum_{\substack{r_1, \dots, r_{n-1}:\\ 0 \leq r_j \leq j-1\\r_n=0}} \prod_{i=1}^{n-1} a_i^{r_i-r_{i+1}}q^{r_i(r_i-r_{i+1})} \nonumber
\\& \qquad \qquad \qquad \times \left(- \left(\prod_{\ell=0}^{i-1} a_i a_{\ell}^{-1} \right) q^{\frac{i(i+1)}{2}+(i+1)r_i-ir_{i+1}};q^{i(i+1)}\right)_{\infty} \label{eq:formulefinale}
\\& \qquad \qquad \qquad \times \left(- \left(\prod_{\ell=0}^{i-1} a_{\ell} a_{i}^{-1} \right) q^{\frac{i(i+1)}{2}-(i+1)r_i+ir_{i+1}};q^{i(i+1)}\right)_{\infty}. \nonumber
\end{align}
\end{theorem}
The formula \eqref{eq:jacob} will allow us to retrieve the Kac-Peterson character formula \cite{KacPeterson} for level $1$ irreducible highest weight modules of $A_{n-1}^{(1)}$ in our second paper \cite{DK19-2}. On the other hand, \eqref{eq:formulefinale} will give a new expression for the character as a sum of infinite products.

Moreover, the formula \eqref{eq:formulefinale} gives an expression of Andrews' function $\mathrm{C}\Phi_k(q)$ as a sum of infinite products, which makes it very easy to express as a sum of modular forms. An expression for $\mathrm{C}\Phi_k(q)$ as a sum of infinite products was already given by Andrews \cite{AndrewsFrob} (without the colours) in the cases $k=1,2,3.$ This is the first times that the case of general $k$ is treated and that a refinement with colour variables is introduced.

\medskip

Finally, through our bijection from Theorem \ref{th:bij}, Theorem \ref{th:Primcgene} also gives us a very broad generalisation of Capparelli's identity in terms of coloured Frobenius partitions. %Due to the fact that the number of parts in each free colour is not constant in the bijection of Theorem \ref{th:bij}, we need to set $b_i = a_i^{-1}$ for all $i$, so that we do not keep track of the free colours in the generating function.

\begin{theorem}[Generalisations of Capparelli's identity]
\label{th:Capgene}
Let $n$ be a positive integer, and let $\delta$ and $\gamma$ be functions satisfying Conditions 1 and 2, respectively.

Let $C_{n}(\delta,\gamma;m;u_0,\dots,u_{n-1};v_0,\dots,v_{n-1})$ be the number of partitions of $m$ in $\mathcal{C}_n(\delta,\gamma)$ (see Definition \ref{def:Capp_conditions}), where for $i \in \{0, \dots, n-1\},$ the symbol $a_i$ (resp. $a_i^{-1}$) appears $u_i$ (resp. $v_i$) times in the colours.

Let $F_{n}(m;u_0,\dots,u_{n-1};v_0,\dots,v_{n-1})$ be the number of $n^2$-coloured Frobenius partitions of $m$ where for $i \in \{0, \dots, n-1\},$ the symbol $a_i$ (resp. $a_i^{-1}$) appears $u_i$ (resp. $v_i$) times in the colours.

Then
\begin{align*}
&\sum_{m,u_0, \dots, u_{n-1}, v_0, \dots, v_{n-1} \geq 0} C_{n}(\delta,\gamma;m;u_0,\dots,u_{n-1};v_0,\dots,v_{n-1})q^m a_0^{u_0-v_0} \cdots a_{n-1}^{u_{n-1}-v_{n-1}}
\\&=(q;q)_{\infty} \times \sum_{m,u_0, \dots, u_{n-1}, v_0, \dots, v_{n-1} \geq 0} F_{n}(m;u_0,\dots,u_{n-1};v_0,\dots,v_{n-1})q^m a_0^{u_0-v_0} \cdots a_{n-1}^{u_{n-1}-v_{n-1}}
\\&= (q;q)_{\infty} [x^0] \prod_{i=0}^{n-1}(-x a_i q;q)_{\infty}(-x^{-1} a_i^{-1};q)_{\infty}.
\end{align*}
\end{theorem}

\begin{remark}
Combining this with Propositions \ref{prop:genMP} and \ref{prop:Cn2} gives two simple generalisations of Capparelli's identity.
\end{remark}

Again, performing the dilations $q \rightarrow q^n$, and for all $i \in \{0, \dots,n-1\}$, $a_i \rightarrow q^{-i},$ gives us a very simple corollary corresponding to the principal specialisation.

\begin{corollary}[Principal specialisation]
\label{cor:capaspec}
Let $n$ be a positive integer, and let $\delta$ and $\gamma$ be functions satisfying Conditions 1 and 2, respectively.
We have
\begin{align*}
&\sum_{m,u_0, \dots, u_{n-1}, v_0, \dots, v_{n-1} \geq 0} C_{n}(\delta,\gamma;m;u_0,\dots,u_{n-1};v_0,\dots,v_{n-1})q^{nm -\sum_{i=0}^{n-1} i (u_i-v_i)}
\\&= \frac{(q^n;q^n)_{\infty}}{(q;q)_{\infty}}.
\end{align*}
\end{corollary}
In other words, after performing the principal specialisation, our generalised Capparelli partitions become equinumerous with $n$-regular partitions, i.e. partitions having no part divisible by $n$. In the representation theory of the symmetric group $S_m$, irreducible $n$-modular representations are labelled by $n$-regular partitions of $m$ when $n$ is prime \cite{JamesKerber}.
There is ample literature on $n$-regular partitions: they have been studied for their multiplicative properties \cite{Beckwith}, in connection with modular forms and congruences \cite{Carlson,GordonOno,Penniston}, and related to $K3$-surfaces \cite{LovejoyPenniston}.

\bigskip

The remainder of this paper is organised as follows.
In Section \ref{sec:redcolseq}, we define the notion of kernel and reduced colour sequence, which will be key in our proof of Theorem \ref{th:PnFnbound}, and compute the weight of the minimal partition with a given kernel. In Section \ref{sec:frob}, we study the combinatorics of coloured Frobenius partitions, and prove Theorems \ref{th:Primcgene} and \ref{th:main2}. In Section \ref{sec:bij}, we give the bijection between $\mathcal{P}_n$ and $\mathcal{C}_n(\delta,\gamma)\times \mathcal{P}^0$. Finally, in Section \ref{sec:prop1}, we give the proof of a key Proposition from Section \ref{sec:redcolseq}, which we postponed to the end as it is quite technical and not necessary to the understanding of the rest of this paper.

\section{Reduced colour sequences and minimal partitions}
\label{sec:redcolseq}

\subsection{Definition}
\label{subseq:definitions}
The original method of weighted words of Alladi and Gordon \cite{Alladi,AllAndGor} relies on the idea that any partition with $m$ parts and satisfying difference conditions can be obtained from the minimal partition satisfying difference conditions and adding a partition with at most $m$ parts to it. For example, all Rogers-Ramanujan partitions into $m$ parts, satisfying difference at least $2$ between consecutive parts, can be obtained by starting with the minimal partition $(2m-1)+(2m-3)+ \cdots + 3 +1$, and adding some partition into at most $m$ parts to it.

Here, to compute the generating function for coloured partitions with difference conditions of $\mathcal{P}_n$, we also use minimal partitions. But while Alladi, Andrews, and Gordon computed minimal partitions with a given number of parts, here we compute minimal partitions with a given \emph{kernel}. Let us start by defining this terminology.

\begin{definition}
Let $c_1, \dots , c_s$ be a sequence of colours taken from $\{a_i b_k : i,k \in \N \}$. The \emph{minimal partition} associated to $c_1, \dots , c_s$ according to the difference conditions $\Delta$ is the coloured partition $\lambda_1 + \cdots + \lambda_s$ with minimal weight such that for all $i \in \{1, \dots , s\}$, $c(\lambda_i)= c_i$. We denote this partition by $\mathrm{min}_{\Delta}(c_1,\dots,c_s).$
\end{definition}

\begin{proposition}
The weight of $\min_{\Delta}(c_1,\dots,c_s)$ is equal to
$$\left| \mathrm{min}_{\Delta}(c_1,\dots,c_s) \right| = \sum_{k=1}^{s} k \Delta(c_k,c_{k+1}).$$
Here, we used again the convention that $c_{s+1}= a_{\infty} b_{\infty}$ and $\Delta(c,a_{\infty} b_{\infty})=1$ for every colour $c$.
\end{proposition}

\begin{proof}
Let $c_1, \dots , c_s$ be a sequence of colours and let $\min_{\Delta}(c_1,\dots,c_s)=\lambda_1 + \cdots + \lambda_s$ be the corresponding minimal partition.
By definition, the smallest part $\lambda_s$ of the minimal partition is equal to $1$, which is also equal to $\Delta(c_s,c_{s+1})$. For all $i \in \{1, \dots, s-1\}$ we have $\lambda_i = \lambda_{i+1}+\Delta(c_i,c_{i+1})$. Thus by induction,
$$\lambda_i = \sum_{k=i}^{s} \Delta(c_k, c_{k+1}).$$
Summing over all $i \in \{1, \dots , s\},$ we get
\begin{align*}
\left|\mathrm{min}_{\Delta}(c_1,\dots,c_s) \right| &=\sum_{i=1}^{s} \sum_{k=i}^{s} \Delta(c_k, c_{k+1})
\\ &= \sum_{k=1}^{s} \sum_{i=1}^{k} \Delta(c_k, c_{k+1})
\\&= \sum_{k=1}^{s} k \Delta(c_k,c_{k+1}).
\end{align*}
\end{proof}

\begin{example}
Considering the difference conditions $\Delta$ from matrix $P_3$ in \eqref{Primcmatrix9}, the minimal partition with colour sequence $a_1b_0,a_0b_0,a_2b_2,a_1b_1,a_1b_1,a_0b_1,a_1b_2,a_0b_2$ is
$$ \mathrm{min}_{\Delta}(a_1b_0,a_0b_0,a_2b_2,a_1b_1,a_1b_1,a_0b_1,a_1b_2,a_0b_2)=
9_{a_1b_0}+8_{a_0b_0}+7_{a_2b_2}+6_{a_1b_1}+6_{a_1b_1}
+4_{a_0b_1}+3_{a_1b_2}+1_{a_0b_2}.$$
It has weight $44$.
\end{example}

Given a sequence $c_1, \dots , c_s$ of colours taken from $\{a_i b_k : i,k \in \N \}$, we define the following operations:
\begin{itemize}
\item if there is some $i$ such that $c_i = a_k b_{\ell}$ and $c_{i+1} = a_{\ell} b_{\ell}$, then remove $c_{i+1}$ from the colour sequence,
\item if there is some $i$ such that $c_i = a_k b_{k}$ and $c_{i+1} = a_{k} b_{\ell}$, then remove $c_{i}$ from the colour sequence.
\end{itemize}
Apply the operations above as long as it is possible.
The sequence obtained in the end is called the \emph{reduction} of $c_1, \dots , c_s$, denoted by $\mathrm{red}(c_1, \dots , c_s)$.
A colour sequence that is equal to its reduction is called a \emph{reduced colour sequence}.

\begin{remark}
The reduction operation only removes free colours.
\end{remark}

\begin{remark}
The order in which removals are done does not have any influence on the final result.
\end{remark}

%\begin{remark}
%For each pair of free colours $(a_kb_{k},a_{\ell} b_{\ell})$ with $k \neq \ell$, there is exactly one bound colour $a_{k} b_{\ell}$ such that $a_kb_{k}$ can be removed to its left and $a_{\ell} b_{\ell}$ can be removed to its right.
%\end{remark}

\begin{remark}
For each bound colour $a_{k} b_{\ell}$ ($k \neq \ell$), there is exactly one free colour $a_kb_{k}$ that can be removed to its left, and exactly one free colour $a_{\ell} b_{\ell}$ that can be removed to its right.
\end{remark}

\begin{example}
The reduction of 
$$a_1 b_1 , a_1 b_2 , a_2 b_2 ,a_3 b_3 , a_3 b_1  ,a_1 b_3 ,a_3 b_3  ,a_3 b_3  ,a_3 b_2, a_1b_1$$ 
is 
$$a_1 b_2, a_3 b_1  ,a_1 b_3 ,a_3 b_2,a_1b_1.$$
\end{example}

\begin{definition}
Let $\lambda = \lambda_1 + \cdots + \lambda_s$ be a partition such that $c(\lambda_1)=c_1, \dots , c(\lambda_s)=c_s$. The \emph{kernel} of $\lambda$, denoted by $\mathrm{ker}(\lambda)$, is the reduced colour sequence $\mathrm{red}(c_1, \dots , c_s)$. 
\end{definition}

\subsection{Combinatorial description of reduced colour sequences} 
We want to study the partitions of $\mathcal{P}_n$ having a given kernel. To do so, we need to understand combinatorially the set of colour sequences having a certain reduction.

\begin{proposition}
\label{prop:insert}
Let $S$ be a reduced colour sequence. Any colour sequence $C$ such that $\mathrm{red}(C)=S$ can be obtained by performing a certain number of insertions of the following types in $S$:
\begin{enumerate}
\item if there is a free colour $a_k b_k$ in $S$, insert the same colour $a_k b_k$ arbitrarily many times to its right,
\item if there is a bound colour $a_k b_\ell$ in $S$, insert the free colour $a_k b_k$ arbitrarily many times to its left,
\item if there is a bound colour $a_k b_\ell$ in $S$, insert the free colour $a_{\ell} b_{\ell}$ arbitrarily many times to its right.
\end{enumerate}
\end{proposition}
The proof follows immediately from the definition of reduced colour sequences in the previous section.

\begin{example}
\label{ex:CandS}
Let 
$$S=a_1b_2, a_3b_1, a_2 b_2, a_4 b_3, a_3 b_2.$$
The sequence
$$C=a_1 b_1, a_1 b_1, a_1b_2, a_2 b_2, a_3 b_3, a_3 b_3, a_3 b_3, a_3b_1,  a_2 b_2, a_2 b_2, a_4 b_3, a_3 b_2$$
is obtained from $S$ by inserting $a_1 b_1$ twice to the left of $a_1b_2$ (insertion $(2)$), $a_2 b_2$ once to the right of $a_1b_2$ (insertion $(3)$), $a_3 b_3$ three times to the left of $a_3 b_1$ (insertion $(2)$),  and $ a_2 b_2$ once to the right of $ a_2 b_2$ (insertion $(1)$).
\end{example}

\begin{remark}
The way one obtains $C$ from $S$ via the insertions above is not unique (even up to the order in which we perform the insertions). Indeed, it could be that in $S= c_1, \dots, c_s$, the colour that can be inserted to the right of some $c_j$ is the same as the one that can be inserted to the left of $c_{j+1}$.

For example $a_1 b_2, a_2 b_2, a_2 b_3$ can be obtained from $a_1 b_2, a_2 b_3$ either by inserting $a_2 b_2$ to the right of $a_1 b_2$ (insertion $(3)$) or to the left of $a_2 b_3$ (insertion $(2)$).
\end{remark}

To understand reduced colour sequences and insertions combinatorially, and make sure that we count our partitions in an unique way, we need some definitions. 
\begin{defprop}
A \emph{primary pair} is a pair $(c,c')$ of bound colours such that in the insertion rules of Proposition \ref{prop:insert}, the free colour that can be inserted to the right of $c$ is the same as the one that can be inserted to the left of $c'$.

These pairs are exactly those of the form $(a_i b_k, a_k b_{\ell})$, where $i\neq k$ and $k \neq \ell$.
\end{defprop}

We will be interested in maximal sequences of primary pairs in $S$.
\begin{definition}
Let $S = c_1, \dots, c_s$ be a reduced colour sequence. The \emph{maximal primary subsequences} of $S$ are subsequences $c_i, c_{i+1}, \dots, c_j$ of $S$ such that
\begin{itemize}
\item for all $k \in \{i, \dots, j-1\}$, $(c_k,c_{k+1})$ is a primary pair,
\item $(c_{i-1},c_{i})$ and $(c_{j},c_{j+1})$ are not primary pairs.
\end{itemize}
We denote by $t(S)$ the number of maximal primary subsequences of $S$, and by $S_1, \dots , S_{t(S)}$ these maximal primary subsequences.
\end{definition}

\begin{example}
Let
$$S= a_1 b_2, a_2 b_3, a_2 b_2, a_1 b_4, a_3 b_2, a_2 b_1, a_3 b_3, a_2 b_2.$$
Here $t(S)=3$ and the maximal primary subsequences of $S$ are, from left to right,
\begin{align*}
S_1 &:= a_1 b_2, a_2 b_3,\\
S_2 &:= a_1 b_4,\\
S_3 &:= a_3 b_2, a_2 b_1.
\end{align*}
\end{example}

Let us now define secondary pairs of colours, inside which two different colours can be inserted.
\begin{defprop}
\label{def:secondarypair}
A \emph{secondary pair} is a pair $(c,c')$ of colours satisfying one of the following assertions:
\begin{enumerate}
\item The colours $c$ and $c'$ are both bound, and the free colour that can be inserted to the right of $c$ is different from the one that can be inserted to the left of $c'$. These pairs are those of the form $(a_i b_j, a_k b_{\ell})$, where $i\neq j$, $j \neq k$,  and $k \neq \ell$.

\item The colour $c$ is free, $c'$ is bound, and the colour that can be inserted to the left of $c'$ is different from $c$.  These pairs are those of the form $(a_i b_i, a_k b_{\ell})$, where $i\neq k$,  and $k \neq \ell$.

\item The colour $c$ is bound,  $c'$ is free, and the colour which can be inserted to the right of $c$ is different from $c'$.  These pairs are those of the form $(a_i b_k, a_{\ell} b_{\ell})$, where $i\neq k$,  and $k \neq \ell$.
\end{enumerate}
\end{defprop}

\begin{remark} In the above, the colours $c$ or $c'$ can be equal to $a_{\infty}b_{\infty}$ (when they are free). This allows us to avoid treating the case of insertions at the ends of the colour sequence $C=c_1, \dots , c_s$ separately. Indeed, by our convention, inserting $a_ib_i$ to the left of $c_1= a_i b_k$ is the same as inserting $a_ib_i$ inside the pair $(c_0,c_1)=(a_{\infty}b_{\infty}, a_i b_k)$. This is included in Case (2).
Similarly, inserting $a_kb_k$ to the right of $c_s= a_i b_k$ is the same as inserting $a_kb_k$ inside the pair $(c_s,c_{s+1})=( a_i b_k,a_{\infty}b_{\infty})$, which is included in Case (3).
\end{remark}

\medskip
With the definitions and propositions above, we can now uniquely determine the places where insertions can occur in a reduced colour sequence.

Let $S= c_1, \dots , c_s$ be a reduced colour sequence of length $s$. Then $S$ can be written uniquely in the form
$$S=T_1S_1T_2S_2 \dots T_tS_tT_{t+1},$$
where $S_1, \dots, S_t$ are the maximal primary subsequences of $S$, and $T_1, \dots, T_{t+1}$ are (possibly empty) sequences of consecutively distinct free colours.

For all $u \in \{1, \dots , t\}$, let $i_{2u-1}$ (resp. $i_{2u}$) be the index of the first (resp. last) colour of $S_u$, i.e.
$$S_u= c_{i_{2u-1}}, \dots , c_{i_{2u}}.$$
We have $i_{2u-1} \leq i_{2u}$, with equality when $S_u$ is a singleton.
By the definition of maximal primary subsequences, for all $u$, the pairs $(c_{i_{2u-1}-1},c_{i_{2u-1}})$ and $(c_{i_{2u}},c_{i_{2u}+1})$ are secondary pairs.

We can now state the following.
\begin{proposition}
Using the notation above, the insertions of free colours in $S$ can occur exactly in the following $s+t$ places (possibly multiple times in the same place):
\begin{itemize}
\item to the right of $c_i$, for all $i \in \{1,\dots ,s\}$,
\item to the left of $c_{i_{2u-1}}$, for all  $u \in \{1, \dots , t\}$.
\end{itemize}
\end{proposition}

Let $f_1, \dots , f_{s+t}$ be the $s+t$ free colours that can be inserted in $S$ (in order).

Let $n_1, \dots , n_{s+t}$ be non-negative integers. We denote by $S(n_1,\dots , n_{s+t})$ the colour sequence obtained from $S$ by inserting $n_i$ times the colour $f_i$ in $S$, for all $i$.

Using this notation, we finally have unicity of the insertions.
\begin{proposition}
For each colour sequence $C$ such that $\mathrm{red}(C)=S$, there exist a unique $(s+t)$-tuple of non-negative integers $(n_1,\dots , n_{s+t})$ such that $C=S(n_1,\dots , n_{s+t})$.
\end{proposition}

\begin{example}
In Example \ref{ex:CandS}, we have $s=5$, $t=3$,
$$S_1=a_1b_2, \quad S_2= a_3b_1, \quad S_3=a_4 b_3, a_3 b_2$$
$$T_1= \emptyset, \quad  T_2= \emptyset, \quad T_3= a_2 b_2, \quad  T_4= \emptyset,$$
 and
$$C=S(2,1,3,0,1,0,0,0).$$
\end{example}

\subsection{Influence of the insertions on the minimal partition}
We now study how insertions inside a colour sequence affect the minimal differences between the parts of the corresponding minimal partition.
Let us start with a general lemma about the minimal differences $\Delta$.
\begin{lemma}
\label{lem:Delta}
For all $k, \ell \in \N$ with $k \neq \ell$, we have
\begin{align}
\Delta(a_kb_k, a_k b_{\ell})&= \chi(k < \ell), \label{eq1}\\
\Delta(a_kb_{\ell}, a_{\ell} b_{\ell})&= \chi(k > \ell), \label{eq2}\\
\Delta(a_kb_k, a_k b_{\ell}) + \Delta(a_kb_{\ell}, a_{\ell} b_{\ell})&= 1. \label{eq3}
\end{align}
\end{lemma}
\begin{proof}
We only give the details for \eqref{eq1}.
Remembering that $k \neq \ell$, we have
\begin{align*}
\Delta(a_kb_k, a_k b_{\ell})&= \chi(k \geq k) - \chi(k =k= k) + \chi(k \leq \ell) - \chi(k =k= \ell) \\
&= 1-1+\chi(k \leq \ell)-0.
\end{align*}
Equation \eqref{eq2} is proved in the same way, and \eqref{eq3} is obtained by adding \eqref{eq1} and \eqref{eq2} together.
\end{proof}

If $S$ is a reduced colour sequence, we want to see how the insertion of some free colour in $S$ affects the minimal partition, or equivalently the minimal differences between successive parts.

Let us start with an observation. Because for all $k$, $\Delta(a_k b_k,a_k b_k)=0$, inserting a free colour $a_k b_k$ once or multiple times inside a given pair has exactly the same effect on the rest of the minimal partition. Therefore we only need to study the case where we insert a single free colour inside a primary or secondary pair.

First, let us see what happens to the minimal differences if we insert a free colour inside a primary pair.
\begin{proposition}
\label{prop:insertprimary}
Let $(a_i b_k, a_k b_{\ell})$, with $i\neq k$ and $k \neq \ell$, be a primary pair. We have
$$\Delta(a_i b_k, a_k b_{k})+\Delta(a_k b_k, a_k b_{\ell})=\Delta(a_i b_k, a_k b_{\ell}).$$
\end{proposition}
\begin{proof}
By \eqref{eq1} and \eqref{eq2}, we have
$$\Delta(a_i b_k, a_k b_{k})+\Delta(a_k b_k, a_k b_{\ell}) = \chi(i > k) + \chi(k < \ell).$$
On the other hand, by the definition of $\Delta$, and using that $i\neq k$ and $k \neq \ell$, we have
\begin{align*}
\Delta(a_i b_k, a_k b_{\ell}) &= \chi(i \geq k) - \chi(i=k=k) + \chi(k \leq \ell)-\chi(k=k=\ell)
\\&=\chi(i > k) -0 + \chi(k < \ell) -0.
\end{align*}
This is the same expression as before.
\end{proof}

The above proposition shows that inserting a free colour inside a primary pair leaves the other parts of the minimal partition unchanged.
\begin{corollary}
\label{cor:insertprimary}
Let $C=c_1, \dots , c_s$ be a colour sequence, and let $\min_{\Delta}(C)= \lambda_1 + \cdots + \lambda_s$ be the corresponding minimal partition. Inserting a free colour $c'$ inside a primary pair $(c_i,c_{i+1})$ doesn't disrupt the minimal differences. The minimal partition after insertion will be $\lambda_1 + \cdots + \lambda_i + \lambda' + \lambda_{i+1}+ \cdots + \lambda_s$, with $\lambda' = \lambda_{i+1} + \Delta(c',c_{i+1})$.
\end{corollary}

We now turn to insertions inside secondary pairs. In certain cases, it will disrupt the minimal differences. We first study the case where we insert a free colour to the left of $c'$ in a secondary pair $(c,c')$.
\begin{proposition}[Left insertion]
\label{prop:leftinsert}
Let $(a_i b_j, a_k b_{\ell})$, with $j\neq k$ and $k \neq \ell$, be a secondary pair where $a_k b_{\ell}$ is a bound colour (Cases (1) and (2) in Definition \ref{def:secondarypair}). We have
$$\Delta(a_i b_j, a_k b_{k})+\Delta(a_k b_k, a_k b_{\ell})-\Delta(a_i b_j, a_k b_{\ell})= 0 \text{ or } 1.$$
\end{proposition}
\begin{proof}
Let $D$ denote the difference above.
By definition of $\Delta$ and the fact that $j\neq k$ and $k \neq \ell$, we have
$$\Delta(a_i b_j, a_k b_{\ell})= \chi(i \geq k)+\chi(j \leq \ell).$$
On the other hand, we have
$$\Delta(a_i b_j, a_k b_{k})= \chi(i \geq k)+\chi(j < k),$$
and
$$\Delta(a_k b_k, a_k b_{\ell})=\chi(k < \ell).$$
Thus the difference is equal to
$$D=  \chi(j < k) + \chi(k < \ell) - \chi(j \leq \ell).$$
This is always equal to $0$ or $1$. Indeed, when the first two terms are $1$, then we have $j < k < \ell$ and the third term is $1$ too. When the last term is $1$, then at least one of the first two is $1$ too. If it wasn't the case, we would have $j \geq k \geq \ell$ and $j \leq \ell$, i.e. $j=k=\ell$, which is impossible because $j \neq k.$
\end{proof}

\begin{definition}
When the difference in the Proposition \ref{prop:leftinsert} is $0$ (resp. $1$), we call $(a_i b_j, a_k b_{\ell})$  a \textit{type $0$ (resp. type $1$) left pair}, and the corresponding insertion a \textit{type $0$ (resp. type $1$) left insertion}.
\end{definition}

\begin{remark}
The type of the left pair $(a_i b_j, a_k b_{\ell})$ in Proposition \ref{prop:leftinsert} doesn't depend on $i$. In particular $(a_i b_j, a_k b_{\ell})$ with $i \neq j$ and $(a_j b_j, a_k b_{\ell})$ have the same type.
\end{remark}

Similarly, we study the case where we insert a free colour to the right of $c$ in a secondary pair $(c,c')$. This essentially works in the same way as left insertions.
\begin{proposition}[Right insertion]
\label{prop:rightinsert}
Let $(a_i b_j, a_k b_{\ell})$, with $i\neq j$ and $j \neq k$, be a secondary pair where $a_i b_j$ is a bound colour (Cases (1) and (3) in Definition \ref{def:secondarypair}). We have
$$\Delta(a_i b_j, a_j b_j)+\Delta(a_j b_j, a_k b_{\ell})-\Delta(a_i b_j, a_k b_{\ell})= 0 \text{ or } 1.$$
\end{proposition}
\begin{proof}
Following the same reasoning as in the proof of Proposition \ref{prop:leftinsert}, we show that the difference above is equal to
$$\chi(i > j) + \chi(j > k) - \chi(i \geq k),$$
which again is always equal to $0$ or $1$.
\end{proof}

As before, we define type $0$ and type $1$.
\begin{definition}
When the difference in Proposition \ref{prop:rightinsert} is $0$ (resp. $1$), we call $(a_i b_j, a_k b_{\ell})$  a \textit{type $0$ (resp. type $1$) right pair}, and the corresponding insertion a \textit{type $0$ (resp. type $1$) right insertion}.
\end{definition}

\begin{remark}
The type of the right pair $(a_i b_j, a_k b_{\ell})$ in Proposition \ref{prop:rightinsert} doesn't depend on $\ell$. In particular $(a_i b_j, a_k b_{\ell})$ with $k \neq \ell$ and $(a_i b_j, a_k b_{k})$ have the same type.
\end{remark}

From Propositions \ref{prop:leftinsert} and \ref{prop:rightinsert}, we now understand the effect that an insertion inside a secondary pair has on the minimal partition, depending on the type of this insertion.
\begin{corollary}[Type $0$ insertion]
\label{cor:type0}
Let $C=c_1, \dots , c_s$ be a colour sequence, and let $\min_{\Delta}(C)= \lambda_1 + \cdots + \lambda_s$ be the corresponding minimal partition. For any $i \in \{0, \dots, s\}$, the type $0$ insertion of a free colour $c'$ inside a secondary pair $(c_i,c_{i+1})$ doesn't disrupt the minimal differences. The minimal partition after insertion will be $\lambda_1 + \cdots + \lambda_i + \lambda' + \lambda_{i+1}+ \cdots + \lambda_s$, with $\lambda' = \lambda_{i+1} + \Delta(c',c_{i+1})$.
\end{corollary}

\begin{example}
\label{ex:min}
The minimal partition with colour sequence
$$C=a_2 b_2,a_1b_0,a_0b_2,a_1b_0,a_2b_1$$ is
$$\mathrm{min}_{\Delta}(C)=5_{a_2 b_2}+4_{a_1b_0}+2_{a_0b_2}+2_{a_1b_0} +1_{a_2b_1}.$$
We insert $a_1b_1$ inside $(a_0b_2,a_1b_0)$. The minimal partition with colour sequence
$$C'=a_2 b_2,a_1b_0,a_0b_2,a_1 b_1,a_1b_0,a_2b_1$$ is
$$\mathrm{min}_{\Delta}(C')=5_{a_2 b_2}+4_{a_1b_0}+2_{a_0b_2}+2_{a_1b_1}+2_{a_1b_0} +1_{a_2b_1}.$$
The part $2_{a_1b_1}$ was inserted, but all the other parts stay the same.
\end{example}

\begin{corollary}[Type $1$ insertion]
\label{cor:type1}
Let $C=c_1, \dots , c_s$ be a colour sequence, and let $\min_{\Delta}(C)= \lambda_1 + \cdots + \lambda_s$ be the corresponding minimal partition. For any $i \in \{0, \dots, s\}$, the type $1$ insertion of a free colour $c'$ inside a secondary pair $(c_i,c_{i+1})$ adds $1$ to the minimal difference between $c_i$ and $c_{i+1}$. This forces us to add $1$ to each part to the left of the newly inserted part in the minimal partition, which become $(\lambda_1+1) + \cdots + (\lambda_i+1) + \lambda' + \lambda_{i+1}+ \cdots + \lambda_s$, with $\lambda' = \lambda_{i+1} + \Delta(c',c_{i+1})$.
\end{corollary}

\begin{example}
In the colour sequence $C$ of Example \ref{ex:min}, we insert $a_2b_2$ inside $(a_0b_2,a_1b_0)$.
The minimal partition with colour sequence
$$C''=a_2 b_2,a_1b_0,a_0b_2,a_2 b_2,a_1b_0,a_2b_1$$ is
$$\mathrm{min}_{\Delta}(C'')=6_{a_2 b_2}+5_{a_1b_0}+3_{a_0b_2}+3_{a_2b_2}+2_{a_1b_0} +1_{a_2b_1}.$$
All the parts to the left of the newly inserted part are increased by one compared to $\mathrm{min}_{\Delta}(C)$.
\end{example}

So far we have only studied the case of a single insertion (either left or right) inside a secondary pair. We still need to understand what happens to the minimal differences if, inside a secondary pair $(a_i b_j, a_k b_{\ell})$, we insert both $a_j b_j$ to the right of $a_i b_j$ and $a_k b_k$ to the left of $a_k b_{\ell}$.
\begin{proposition}[Left and right insertion]
\label{prop:doubleinsert}
Let $(a_i b_j, a_k b_{\ell})$, with $j \neq k$, be a secondary pair. We have
\begin{align*}
\Delta(a_i &b_j, a_j b_j)+\Delta(a_j b_j, a_k b_k)+ \Delta(a_k b_k, a_k b_{\ell})-\Delta(a_i b_j, a_k b_{\ell})\\
&=
\begin{cases}
0 \text{ if both the right and left insertions inside $(a_i b_j, a_k b_{\ell})$ are of type $0$,}\\
1 \text{ if exactly one of the insertions inside $(a_i b_j, a_k b_{\ell})$ is of type $1$,}\\
2 \text{ if both the right and left insertions inside $(a_i b_j, a_k b_{\ell})$ are of type $1$.}\\
\end{cases}
\end{align*}
\end{proposition}
\begin{proof}
Let $D$ be the difference above.
We have
\begin{align*}
D &= \Delta(a_i b_j, a_j b_j)+\Delta(a_j b_j, a_k b_k) - \Delta(a_i b_j, a_k b_k)\\
&+ \Delta(a_i b_j, a_k b_k) + \Delta(a_k b_k, a_k b_{\ell})-\Delta(a_i b_j, a_k b_{\ell}).
\end{align*}
The first line is equal to the right type of $(a_i b_j, a_k b_k)$, which by the remark after Proposition \ref{prop:rightinsert}, is the same as the right type of $(a_i b_j, a_k b_{\ell})$. The second line is simply the left type of $(a_i b_j, a_k b_{\ell})$.
\end{proof}
Thus performing both a left and right insertion inside a secondary pair is the same as performing the two insertions separately.

We conclude this section by summarising the influence of all the possible insertions on the minimal partition.

\begin{proposition}[Summary of the different types of insertion]
\label{prop:summaryinsert}
Let $C=c_1, \dots , c_s$ be a colour sequence, and let $\min_{\Delta}(C)= \lambda_1 + \cdots + \lambda_s$ be the corresponding minimal partition. When we insert a free colour $c'$ inside a pair $(c_i,c_{i+1})$, the minimal partition transforms as follows:
\begin{itemize}
\item if $c_i$ is a free colour and $c'=c_i$, the minimal partition becomes $\lambda_1 + \cdots + \lambda_i + \lambda_i + \lambda_{i+1}+ \cdots + \lambda_s$ (i.e. the part $\lambda_i$ repeats, and the rest of the partition remains unchanged);
\item if $(c_i,c_{i+1})$ is a primary pair, the minimal partition becomes $\lambda_1 + \cdots + \lambda_i + \lambda' + \lambda_{i+1}+ \cdots + \lambda_s$, with $\lambda' = \lambda_{i+1} + \Delta(c',c_{i+1})$;
\item if $(c_i,c_{i+1})$ is a secondary pair and the insertion of $c'$ is of type $0$, the minimal partition becomes $\lambda_1 + \cdots + \lambda_i + \lambda' + \lambda_{i+1}+ \cdots + \lambda_s$, with $\lambda' = \lambda_{i+1} + \Delta(c',c_{i+1})$;
\item if $(c_i,c_{i+1})$ is a secondary pair and the insertion of $c'$ is of type $1$, the minimal partition becomes $(\lambda_1+1) + \cdots + (\lambda_i+1) + \lambda' + \lambda_{i+1}+ \cdots + \lambda_s$, with $\lambda' = \lambda_{i+1} + \Delta(c',c_{i+1})$ (i.e. we add $1$ to all the parts to the left of the newly inserted part $\lambda'$).
\end{itemize}
\end{proposition}
We call the first two types of insertions above \emph{neutral insertions}.

\subsection{Generating function for partitions with a given kernel}
\label{subsec:gfkernelDelta}
Our goal is to count partitions of $\mathcal{P}_n$ with a given kernel. The results from the previous section will help us do so.

Let $S= c_1, \dots , c_s$ be a reduced colour sequence of length $s$, having $t$ maximal primary subsequences.
Let $f_1 , \dots , f_{s+t}$ be the free colours that can be inserted in $S$. In the following, we denote by $\mathcal{N}$ (resp. $\mathcal{T}_0$, $\mathcal{T}_1$) the set of indices $i$ such that the insertion of $f_i$ is neutral (resp. of type $0$, of type $1$). We have $\mathcal{N} \sqcup \mathcal{T}_0 \sqcup \mathcal{T}_1 = \{1, \dots, s+t\}.$

Moreover, the secondary pairs in $S$ are exactly $(c_{i_{2u-1}-1},c_{i_{2u-1}})$ and $(c_{i_{2u}},c_{i_{2u}+1})$, for $u \in \{1, \dots , t\}$, where $S_u = c_{i_{2u-1}}, \dots , c_{i_{2u}}$. So we can write
$$\mathcal{T}_0 = \bigsqcup\limits_{u=1}^{t} \mathcal{T}_0^u, \qquad \mathcal{T}_1 = \bigsqcup\limits_{u=1}^{t} \mathcal{T}_1^u,$$
where $\mathcal{T}_0^u$ (resp. $\mathcal{T}_1^u$) is the set of indices $j$ such that $f_j$ can be inserted inside $(c_{i_{2u-1}-1},c_{i_{2u-1}})$ or $(c_{i_{2u}},c_{i_{2u}+1})$ and is of type $0$ (resp. $1$).
For all $u \in \{1, \dots , t\}$, we have $|\mathcal{T}^u_0|=2-|\mathcal{T}^u_1|.$
More precisely, $\mathcal{T}_0^u \bigsqcup \mathcal{T}_1^u = \{i_{2u-1}+u-1, i_{2u}+u\}.$

We want to study the minimal partition of the colour sequence $S(n_1, \dots, n_{s+t})$.
Denote by $\mathcal{S}^u_1$ (resp. $\mathcal{S}_1)$  the indices $j$ of $\mathcal{T}^u_1$ (resp. $\mathcal{T}_1$) such that $n_j >0$, i.e. such that the colour $f_j$ is actually inserted in the sequence.
We start with the following lemma.

\begin{lemma}
\label{lem:sizepart}
For all $j \in \{1, \dots , s+t\}$, if $n_j > 0$, i.e. the colour $f_j$ is actually inserted, then the corresponding part $\lambda(f_j)$ in the minimal partition of $S(n_1, \dots, n_{s+t})$ is equal to 
\begin{equation}
\label{eq:sizepart}
\lambda(f_j)= \#\left(\{j, \dots , s+t\} \cap (\mathcal{N}\sqcup\mathcal{T}_0 \sqcup \mathcal{S}_1) \right).
\end{equation}
\end{lemma}
\begin{proof}
We proceed via backward induction on $j$.
\begin{itemize}
\item If $j=s+t$, $\lambda(f_{s+t})$ is the last part of the minimal partition and therefore has size $1$. Equation \eqref{eq:sizepart} is correct, as $s+t \in \mathcal{N}\sqcup\mathcal{T}_0 \sqcup \mathcal{S}_1$.
\item Now assume that \eqref{eq:sizepart} holds for $f_{j+1}$, and prove it for $f_j$. Let $k$ and $\ell$ be such that $f_j=a_k b_k$ and $f_{j+1}=a_{\ell} b_{\ell}$. We always have $k \neq \ell.$
\begin{enumerate}
\item For now, let us assume that $n_{j+1}>0$, i.e. that $f_{j+1}$ was actually inserted in the colour sequence.
\begin{itemize}
\item If $j \in \mathcal{N}$ or $j$ is a left secondary insertion, then the subsequence between $f_j$ and $f_{j+1}$ in $S(n_1, \dots, n_{s+t})$ is $f_j, a_k b_{\ell}, f_{j+1}$ or $f_j, a_{\ell} b_{\ell}, f_{j+1}$. In the first case, we have
\begin{align*}
\lambda(f_j) &= \Delta(a_k b_k, a_k b_{\ell})+\Delta(a_k b_{\ell},a_{\ell} b_{\ell})+\lambda(f_{j+1})\\
&=1+ \lambda(f_{j+1}),
\end{align*}
where the second equality follows from Lemma \ref{lem:Delta}.

In the second case, we also have
\begin{align*}
\lambda(f_j) &= \Delta(a_k b_k, a_{\ell} b_{\ell})+\Delta(a_{\ell} b_{\ell},a_{\ell} b_{\ell})+\lambda(f_{j+1})\\
&=1+ \lambda(f_{j+1}),
\end{align*}

By the induction hypothesis, we have
\begin{align*}
\lambda(f_j) &= 1+  \#\left(\{j+1, \dots , s+t\} \cap (\mathcal{N}\sqcup\mathcal{T}_0 \sqcup \mathcal{S}_1) \right)\\
&=  \#\left(\{j, \dots , s+t\} \cap (\mathcal{N}\sqcup\mathcal{T}_0 \sqcup \mathcal{S}_1) \right),
\end{align*}
because $j \in \mathcal{N}\sqcup\mathcal{T}_0 \sqcup \mathcal{S}_1$.

\item If $j$ is a right secondary insertion, then $f_j$ appears directly before $f_{j+1}$ in $S(n_1, \dots, n_{s+t})$. Thus we have
\begin{align*}
\lambda(f_j) &= \Delta(f_j,f_{j+1})+\lambda(f_{j+1})\\
&=1+ \lambda(f_{j+1}),
\end{align*}
and we can deduce \eqref{eq:sizepart} in the exact same way as before.
\end{itemize}

\item Now we treat the case where $f_{j+1}$ was not inserted in the colour sequence. By Proposition \ref{prop:summaryinsert}, if $j+1 \in \mathcal{N}\sqcup\mathcal{T}_0$, it does not change anything to the other parts in the minimal partition , so $\lambda(f_j)$ stays the same as in case (1).

If $j+1 \in \mathcal{T}_1$ and $f_{j+1}$ was not inserted, then by Proposition \ref{prop:summaryinsert}, the part $\lambda(f_{j})$ decreases by one compared to the previous case. But in this case, $\#\left(\{j, \dots , s+t\} \cap (\mathcal{N}\sqcup\mathcal{T}_0 \sqcup \mathcal{S}_1) \right)$ also decreases by one compared to case (1), so Equation \eqref{eq:sizepart} is still correct.
\end{enumerate}
\end{itemize}
\end{proof}

We can now give a formula for the weight of the minimal partition with colour sequence $S(n_1, \dots, n_{s+t})$.
\begin{proposition}
\label{prop:minpartition}
With the notation above, the weight of the minimal partition with colour sequence $S(n_1, \dots, n_{s+t})$ is
\begin{equation}
\label{eq:minpartition}
\begin{aligned}
\left| \mathrm{min}_{\Delta}(S(n_1, \dots, n_{s+t})) \right| &= |\mathrm{min}_{\Delta}(S)|
\\& + \sum_{j \in \mathcal{S}_1} \left( P(j)+ n_j \times \#\left(\{j, \dots , s+t\} \cap (\mathcal{N}\sqcup\mathcal{T}_0 \sqcup \mathcal{S}_1) \right) \right)
\\& + \sum_{j \in \mathcal{N}\cup\mathcal{T}_0}  n_j \times \#\left(\{j, \dots , s+t\} \cap (\mathcal{N}\sqcup\mathcal{T}_0 \sqcup \mathcal{S}_1) \right),
\end{aligned}
\end{equation}
where $P(j)$ is the number of colours of $S$ that are to the left of $f_j$.
\end{proposition}
\begin{proof}
We start with the minimal partition $\mathrm{min}_{\Delta}(S)$ with colour sequence $S$. It has weight $|\mathrm{min}_{\Delta}(S)|$.

Then we insert the parts corresponding to colours of type $1$. Let $j \in \mathcal{S}_1$. By Proposition \ref{prop:summaryinsert}, inserting $f_j$ adds $1$ to all the parts of $\mathrm{min}_{\Delta}(S)$ which are to the left of $\lambda(f_j)$. So this adds $P(j)$ to the total weight. Moreover, by Lemma \ref{lem:sizepart}, the part $\lambda(f_j)$ is of size $\#\left(\{j, \dots , s+t\} \cap (\mathcal{N}\cup\mathcal{T}_0 \cup \mathcal{S}_1) \right)$, and we insert it $n_j$ times. Summing over all $j \in \mathcal{S}_1$ gives the first sum.

Finally, the insertion of parts corresponding to colours $f_j$ with $j \in \mathcal{N}\cup\mathcal{T}_0$ yields the last sum.
\end{proof}

\medskip
Starting from Proposition \ref{prop:minpartition}, we will show a key proposition, which will be very useful to establish the connection with coloured Frobenius partitions.

Recall that the \emph{$q$-binomial coefficient} is defined as follows:
$${n \brack k}_q:=\frac{(q;q)_n}{(q;q)_k(q;q)_{n-k}},$$
and we assume that ${n \brack k}_q=0$ if $k<0$ or $k>n$.

\begin{proposition}
\label{prop:main'}
Let $n$ be a positive integer and $m$ be a non-negative integer. Let $S= c_1, \dots , c_s$ be a reduced colour sequence of length $s$, having $t$ maximal primary subsequences. The generating function for minimal partitions in $\mathcal{P}_n$ with kernel $S$, having $s+m$ parts, is the following:
\begin{equation}
\label{eq:prop1'}
\sum_{\substack{C \text{colour sequence of length } s+m\\ \text{such that } \mathrm{red}(C)=S}} q^{|\min_{\Delta}(C)|} = q^{|\min_{\Delta}(S)|+m} \sum_{u=0}^{t} q^{u(s-t)}g_{u,t}(q;|\mathcal{T}_0^1|,\ldots,|\mathcal{T}_0^{t}|) {s+m-1 \brack m-u}_q,
\end{equation}
where $g_{0,0}=1$, and for $u \leq v$,
$$g_{u,v}(q;x_1, \dots, x_v) = \sum_{\substack{\epsilon_1,\dots , \epsilon_v \in \{0,1\}:\\ \epsilon_1 + \cdots + \epsilon_v = u}} q^{uv+ {u \choose 2}} \prod_{k=1}^v q^{(x_k -1) \sum_{i=1}^{k-1} \epsilon_i}.$$
\end{proposition}

By observing that all partitions of $\mathcal{P}_n$ with a given colour sequence $C$ of length $s+m$ can be obtained in a unique way by adding a partition with at most $s+m$ parts to the minimal partition $\min_{\Delta}(C)$, Proposition \ref{prop:main'} is actually equivalent to the following generating function for all partitions of $\mathcal{P}_n$ with a given kernel.
\begin{proposition}
\label{prop:main}
Let $n$ be a positive integer and $m$ be a non-negative integer. Let $S= c_1, \dots , c_s$ be a reduced colour sequence of length $s$, having $t$ maximal primary subsequences. The generating function for partitions in $\mathcal{P}_n$ with kernel $S$, having $s+m$ parts, is the following:
\begin{equation}
\label{eq:prop1}
\sum_{\substack{\lambda \in \mathcal{P}_n:\\ \ell(\lambda)=s+m\\\mathrm{ker}(\lambda)=S}} q^{|\lambda|} = \frac{q^{|\min_{\Delta}(S)|+m}}{(q;q)_{s+m}} \sum_{u=0}^{t} q^{u(s-t)}g_{u,t}(q;|\mathcal{T}_0^1|,\ldots,|\mathcal{T}_0^{t}|) {s+m-1 \brack m-u}_q.
\end{equation}
\end{proposition}

The proof of Proposition \ref{prop:main'} from Proposition \ref{prop:minpartition}, quite technical, is postponed to Section \ref{sec:prop1}. Its reading is not necessary to understand the connection between the generalised Primc partitions $\mathcal{P}_n$ and the $n^2$-coloured Frobenius partitions $\mathcal{F}_n$, which we  study in the next section, nor the bijection with the generalisation of Capparelli's identity, which we give in Section \ref{sec:bij}.

\section{Coloured Frobenius partitions}
\label{sec:frob}
In this section, we compute the generating function for $n^2$-coloured Frobenius partitions with a given kernel and show that it is the same as the generating function \eqref{eq:prop1} for generalised Primc partitions with the same kernel.

\subsection{The difference conditions corresponding to minimal $n^2$-coloured Frobenius partitions}
We start by showing that minimal $n^2$-coloured Frobenius partitions are in bijection with minimal coloured partitions satisfying some new difference conditions $\Delta'.$

Let $\begin{pmatrix}
\lambda_1 & \lambda_2 & \cdots & \lambda_s \\
\mu_1 & \mu_2 & \cdots & \mu_s 
\end{pmatrix}$ be a $n^2$-coloured Frobenius partition. Recall from the introduction that $\lambda =\lambda_1 + \lambda_2 + \cdots + \lambda_s$ is a partition into $s$ distinct non-negative parts, each coloured with some $a_i$, $i \in \{0, \dots, n-1\},$ with the order \eqref{eq:orderFroba}. Similarly, $\mu=\mu_1+\mu_2+\cdots+\mu_s$ is a partition into $s$ distinct non-negative parts, each coloured with some $b_i$, $i \in \{0, \dots, n-1\},$ with the order \eqref{eq:orderFrobb}. The colour sequence of this $n^2$-coloured Frobenius partition is $(c(\lambda_1)c(\mu_1), \dots, c(\lambda_s)c(\mu_s))$, and its kernel can  be defined in the same way as for coloured partitions.

Given a colour sequence $c_1, \dots , c_s$ taken from $\{a_i b_k : i,k \in \{0, \dots, n-1\} \}$, the minimal $n^2$-coloured Frobenius partition associated to $c_1, \dots , c_s$ is the $n^2$-coloured Frobenius partition $\begin{pmatrix}
\lambda_1 & \lambda_2 & \cdots & \lambda_s \\
\mu_1 & \mu_2 & \cdots & \mu_s 
\end{pmatrix}$
with minimal weight such that for all $i \in \{1, \dots , s\}$, $c(\lambda_i)c(\mu_i)= c_i$. We denote it by $\mathrm{min^F}(c_1,\dots,c_s).$

\begin{proposition}
\label{prop:bijDelta'}
Let $c_1, \dots , c_s$ be a colour sequence taken from $\{a_i b_k : i,k \in \{0, \dots, m-1\} \}$. There is a weight-preserving bijection between the minimal $n^2$-coloured Frobenius partition $\mathrm{min^F}(c_1,\dots,c_s)$ and the minimal coloured partition $\mathrm{min}_{\Delta'}(c_1,\dots,c_s),$ where for all $i,k,i',k' \in \N$,
\begin{equation}
\label{eq:Delta'}
\Delta'(a_i b_k, a_{i'} b_{k'}) = \chi(i \geq i')+\chi(k \leq k').
\end{equation}
\end{proposition}

\begin{proof}
Start with $\mathrm{min^F}(c_1,\dots,c_s)=\begin{pmatrix}
\lambda_1 & \lambda_2 & \cdots & \lambda_s \\
\mu_1 & \mu_2 & \cdots & \mu_s 
\end{pmatrix}$, and transform it into the coloured partition $\nu = \nu_1 + \cdots + \nu_s$, where for all $j \in \{1, \dots, s\}$,
\begin{align*}
\nu_j &= \lambda_j + \mu_j+1,\\
c(\nu_j) &= c(\lambda_j)c(\mu_j).
\end{align*}
Clearly $\mathrm{min^F}(c_1,\dots,c_s)$ and $\nu$ have the same weight and colour sequence.

Moreover, by definition of the order \eqref{eq:orderFroba}, and using the minimality of $\mathrm{min^F}(c_1,\dots,c_s)$, the difference between $\lambda_j$ of colour $a_i$ and $\lambda_{j+1}$ of colour $a_{i'}$ is exactly $\chi(i \geq i')$, for all $j \in \{1, \dots, s \}$.
Similarly, the difference between $\mu_j$ of colour $b_{k}$ and $\mu_{j+1}$ of colour $b_{k'}$ is exactly $\chi(k \leq k').$

Thus for all $j \in \{1, \dots, s \}$, the difference between $\nu_j$ and $\mu_{j+1}$ is exactly $\chi(i \geq i')+\chi(k \leq k')$ and $\nu= \mathrm{min}_{\Delta'}(c_1,\dots,c_s)$.

By unicity of the minimal partition (resp. minimal Frobenius partition), this is indeed a bijection.
\end{proof}

We denote by $\mathcal{P}'_n$ the set of $n^2$-coloured partitions satisfying the minimal difference conditions $\Delta'$.

\begin{remark}
When we don't have the minimality condition, the $n^2$-coloured Frobenius partitions with colour sequence $c_1, \dots , c_s$ are \textbf{not} in bijection with coloured partitions with colour sequence $c_1, \dots , c_s$ and minimal differences $\Delta'$.
For example, take the case of one colour $a_1 b_1$. The $n^2$-coloured Frobenius partitions with colour sequence $a_1 b_1$ are generated by $q/(1-q)^2$, as we can choose any value for both $\lambda_1$ and $\mu_1$. On the other hand, coloured partitions with colour sequence $a_1 b_1$ and difference $\Delta'$ are generated by $q/(1-q)$, as we can only choose the value of one part $\nu_1$.
\end{remark}

However, for our purpose in this paper, we only need the generating function for minimal partitions. Moreover, we will be able to relate $\Delta'$ with the difference conditions $\Delta$ of Primc's identity, which will allow us to reuse a lot of the work done in Section \ref{sec:redcolseq}

Let us start with the following property, which follows from the definition of $\Delta$ \eqref{eq:Delta} and $\Delta'$ \eqref{eq:Delta'}.
\begin{property}
\label{prop:DeltaDelta'}
The minimal differences $\Delta(c,c')$ and $\Delta'(c,c')$ are equal, except in the following cases:
\begin{enumerate}
\item $c=c'=a_ib_i$, in which case $\Delta(a_ib_i,a_ib_i)=0$ and $\Delta'(a_ib_i,a_ib_i)=2$,
\item $c=a_ib_i$ and $c'= a_ib_{\ell}$, in which case $\Delta'(a_ib_i,a_ib_{\ell})=\Delta(a_ib_i,a_ib_{\ell})+1$,
\item $c=a_ib_{\ell}$ and $c'= a_{\ell}b_{\ell}$, in which case $\Delta'(a_ib_{\ell},a_{\ell}b_{\ell})=\Delta(a_ib_{\ell},a_{\ell}b_{\ell})+1.$
\end{enumerate}
These particular cases correspond to the insertions of type $(1)$, $(2)$, and $(3)$, respectively, in Proposition \ref{prop:insert}.
\end{property}

Using the fact that reduced colour sequences do not contain any pair $(c,c')$ of the types mentioned in Property \ref{prop:DeltaDelta'}, we have the following corollary.
\begin{corollary}
\label{cor:minDeltaDelta'}
Let $S$ be a reduced colour sequence. Then
$$\mathrm{min}_{\Delta}(S)=\mathrm{min}_{\Delta'}(S).$$
\end{corollary}

But when $C$ is a coloured sequence which is not reduced, we do not have in general $\mathrm{min}_{\Delta}(C)=\mathrm{min}_{\Delta'}(C)$. So to compute the generating function for $n^2$-coloured Frobenius partitions with a given kernel, we define one last difference condition 
$$\Delta'':=2-\Delta',$$ which shares many properties with $\Delta$.

\begin{proposition}
\label{prop:Delta''}
The difference conditions $\Delta''$ satisfy the following properties on free colours.
\begin{enumerate}
\item \emph{Difference between two free colours:} For all $i,k$, $\Delta''(a_ib_i,a_kb_k)=\chi(i \neq k)=\Delta(a_ib_i,a_kb_k).$
\item \emph{Insertion inside a primary pair :} Let $(a_i b_k, a_k b_{\ell})$, with $i\neq k$ and $k \neq \ell$, be a primary pair. We have
$$\Delta''(a_i b_k, a_k b_{k})+\Delta''(a_k b_k, a_k b_{\ell})=\Delta''(a_i b_k, a_k b_{\ell}).$$
\item \emph{Left insertion inside a secondary pair :} Let $(a_i b_j, a_k b_{\ell})$, with $j\neq k$ and $k \neq \ell$, be a secondary pair. We have
$$\Delta''(a_i b_j, a_k b_{k})+\Delta''(a_k b_k, a_k b_{\ell})-\Delta''(a_i b_j, a_k b_{\ell})= 0 \text{ or } 1.$$
Moreover such an insertion is of $\Delta''$-type $0$ (resp. $1$) if and only if it is of $\Delta$-type $1$ (resp. $0$).
\item \emph{Right insertion inside a secondary pair :} Let $(a_i b_j, a_k b_{\ell})$, with $i\neq j$ and $j \neq k$, be a secondary pair. We have
$$\Delta''(a_i b_j, a_j b_j)+\Delta''(a_j b_j, a_k b_{\ell})-\Delta''(a_i b_j, a_k b_{\ell})= 0 \text{ or } 1.$$
Moreover such an insertion is of $\Delta''$-type $0$ (resp. $1$) if and only if it is of $\Delta$-type $1$ (resp. $0$).
\end{enumerate}
\end{proposition}
\begin{proof}
Property $(1)$ follows clearly from the definition of $\Delta'$.

Let us now prove $(2)$. We have:
$$\begin{array}{llll}
\Delta''(a_i b_k, a_k b_{k})+\Delta''(a_k b_k, a_k b_{\ell})
&=& 4 - \Delta'(a_i b_k, a_k b_{k})-\Delta'(a_k b_k, a_k b_{\ell}) & \text{by definition of $\Delta''$}
\\&=& 2 - \Delta(a_i b_k, a_k b_{k})-\Delta(a_k b_k, a_k b_{\ell}) & \text{by Property \ref{prop:DeltaDelta'}}
\\&=& 2 - \Delta(a_i b_k,a_k b_{\ell}) & \text{by Proposition \ref{prop:insertprimary}}
\\&=& 2 - \Delta'(a_i b_k,a_k b_{\ell}) &\text{by Property \ref{prop:DeltaDelta'}}
\\&=& \Delta''(a_i b_k,a_k b_{\ell}) &\text{by definition of $\Delta''$}.
\end{array}
$$

Let us finally turn to $(3)$. Property $(4)$ is proved in a similar way. We have
$$\begin{array}{llll}
& &\Delta''(a_i b_j, a_k b_{k})+\Delta''(a_k b_k, a_k b_{\ell})-\Delta''(a_i b_j, a_k b_{\ell})&\\
&=& 2 - \left(\Delta'(a_i b_j, a_k b_{k})+\Delta'(a_k b_k, a_k b_{\ell})-\Delta'(a_i b_j, a_k b_{\ell}) \right) & \text{by definition of $\Delta''$}
\\&=& 2 - \left(\Delta(a_i b_j, a_k b_{k})+\Delta(a_k b_k, a_k b_{\ell})+1 -\Delta(a_i b_j, a_k b_{\ell}) \right) & \text{by Property \ref{prop:DeltaDelta'}}
\\&=& 1 - \left(\Delta(a_i b_j, a_k b_{k})+\Delta(a_k b_k, a_k b_{\ell})-\Delta(a_i b_j, a_k b_{\ell}) \right). &
\end{array}
$$
But by Proposition \ref{prop:leftinsert},
$$\Delta(a_i b_j, a_k b_{k})+\Delta(a_k b_k, a_k b_{\ell})-\Delta(a_i b_j, a_k b_{\ell})= 0 \text{ or } 1,$$
and the value $0$ or $1$ is the $\Delta$-type of the insertion. This completes the proof of $(3)$.
\end{proof}

Proposition \ref{prop:Delta''} shows that $\Delta''$ behaves exactly like $\Delta$ with respect to the insertion of free colours, except that the types of all insertions inside secondary pairs are reversed. In other words, using the notation at the beginning of Section \ref{subsec:gfkernelDelta}, given a reduced colour sequence $S= c_1, \dots , c_s$ and $f_1 , \dots , f_{s+t}$ the free colours that can be inserted in $S$, the set $\mathcal{N}$ (resp. $\mathcal{T}_0$, $\mathcal{T}_1$) is exactly the set of indices $i$ such that the insertion of $f_i$ is neutral (resp. of type $1$, of type $0$) for the difference conditions $\Delta''$.

\subsection{The generating function for $n^2$-coloured Frobenius partitions with a given kernel}
Now that we understand the difference conditions $\Delta'$ and $\Delta''$, we will use them to compute the generating function for $n^2$-coloured Frobenius partitions with a given kernel.

Before doing this, we need a technical lemma about the function $g_{u,v}$ defined in Proposition \ref{prop:main'}, which will appear again in this section.
\begin{lemma}
\label{lem:g_uv}
Let $g_{u,v}$ be the function defined in Proposition \ref{prop:main'}. We have
$$g_{u,v}(q^{-1};2-x_1, \dots , 2-x_v) = q^{-u(2v+u-1)}g_{u,v}(q;x_1,\dots , x_v).$$
\end{lemma}
\begin{proof}
When $u=v=0$, this is trivially true.
Otherwise, we have by definition:
\begin{align*}
g_{u,v}(q^{-1};2-x_1, \dots , 2-x_v) &= \sum_{\substack{\epsilon_1,\dots , \epsilon_v \in \{0,1\}:\\ \epsilon_1 + \cdots + \epsilon_v = u}} q^{-(uv+ {u \choose 2})} \prod_{k=1}^v q^{-(2-x_k -1) \sum_{i=1}^{k-1} \epsilon_i}\\
&=q^{-u(2v+u-1)} \sum_{\substack{\epsilon_1,\dots , \epsilon_v \in \{0,1\}:\\ \epsilon_1 + \cdots + \epsilon_v = u}} q^{(uv+ {u \choose 2})} \prod_{k=1}^v q^{(x_k -1) \sum_{i=1}^{k-1} \epsilon_i}\\
&=q^{-u(2v+u-1)}g_{u,v}(q;x_1,\dots , x_v).
\end{align*}
\end{proof}

We now give the generating function for minimal coloured partitions with difference conditions $\Delta'$ and a given kernel.
\begin{proposition}
\label{prop:prop2'}
Let $n$ be a positive integer and $m$ be a non-negative integer. Let $S= c_1, \dots , c_s$ be a reduced colour sequence of length $s$, having $t$ maximal primary subsequences. Using the notation of Section \ref{subsec:gfkernelDelta}, the generating function for minimal partitions in $\mathcal{P}'_n$ with kernel $S$, having $s+m$ parts, is the following:
\begin{equation}
\label{eq:prop2'}
\sum_{\substack{C \text{colour sequence} \\ \text{of length } s+m\\ \text{such that } \mathrm{red}(C)=S}} q^{|\min_{\Delta'}(C)|} = q^{|\min_{\Delta}(S)|+m(s+m+1)} \sum_{u=0}^{t} q^{-u(t+m)}g_{u,t}(q;|\mathcal{T}_0^1|,\ldots,|\mathcal{T}_0^{t}|) {s+m-1 \brack m-u}_{q}.
\end{equation}
\end{proposition}
\begin{proof}
Let $C=c_1, \dots, c_{s+m}$ be a colour sequence whose reduction is $S$. The weight of the corresponding minimal partition in $\mathcal{P}'_n$ is
\begin{equation}
\label{eq:6}
|\mathrm{min}_{\Delta'}(C)| = \sum_{i=1}^{s+m} i \Delta'(c_i,c_{i+1})= (s+m)(s+m+1) - |\mathrm{min}_{\Delta''}(C)|,
\end{equation}
where the second equality follows from the definition of $\Delta''.$

On the other hand, by Corollary \ref{cor:minDeltaDelta'} and \eqref{eq:6}, we have
\begin{equation}
\label{eq:minDeltaS}
|\mathrm{min}_{\Delta}(S)|=|\mathrm{min}_{\Delta'}(S)|= s(s+1) - |\mathrm{min}_{\Delta''}(S)|.
\end{equation}

Given that, by Proposition \ref{prop:Delta''}, $\Delta$ and $\Delta''$ have exactly the same insertion properties up to exchanging the type $0$ and $1$ insertions, Proposition \ref{prop:main'} immediately yields
\begin{equation*}
\sum_{\substack{C \text{colour sequence of length } s+m\\ \text{such that } \mathrm{red}(C)=S}} q^{|\min_{\Delta''}(C)|} = q^{|\min_{\Delta''}(S)|+m} \sum_{u=0}^{t} q^{u(s-t)}g_{u,t}(q;|\mathcal{T}_1^1|,\ldots,|\mathcal{T}_1^{t}|) {s+m-1 \brack m-u}_q.
\end{equation*}

Combining this with \eqref{eq:6}, we get that the generating function for minimal partitions in $\mathcal{P}'_n$ is
\begin{align*}
G:&=\sum_{\substack{C \text{colour sequence} \\ \text{of length } s+m\\ \text{such that } \mathrm{red}(C)=S}} q^{|\min_{\Delta'}(C)|} 
\\&= q^{(s+m)(s+m+1)-|\min_{\Delta''}(S)|-m} \sum_{u=0}^{t} q^{-u(s-t)}g_{u,t}(q^{-1};|\mathcal{T}_1^1|,\ldots,|\mathcal{T}_1^{t}|) {s+m-1 \brack m-u}_{q^{-1}}.
\end{align*}
By Lemma \ref{lem:g_uv} and the fact that for all $k \in \{1, \dots t\},$ $|\mathcal{T}_1^k|=2-|\mathcal{T}_0^k|$, the above becomes
$$G=q^{(s+m)(s+m+1)-|\min_{\Delta''}(S)|-m} \sum_{u=0}^{t} q^{-u(s+t+u-1)}g_{u,t}(q;|\mathcal{T}_0^1|,\ldots,|\mathcal{T}_0^{t}|) {s+m-1 \brack m-u}_{q^{-1}}.$$
Now using the fact that 
$${s+m-1 \brack m-u}_{q^{-1}}= q^{-(s+u-1)(m-u)}{s+m-1 \brack m-u}_{q},$$
we obtain
\begin{align*}
G&=q^{(s+m)(s+m+1)-|\min_{\Delta''}(S)|-ms} \sum_{u=0}^{t} q^{-u(t+m)}g_{u,t}(q;|\mathcal{T}_0^1|,\ldots,|\mathcal{T}_0^{t}|) {s+m-1 \brack m-u}_{q}\\
&=q^{|\min_{\Delta}(S)|+m(s+m+1)} \sum_{u=0}^{t} q^{-u(t+m)}g_{u,t}(q;|\mathcal{T}_0^1|,\ldots,|\mathcal{T}_0^{t}|) {s+m-1 \brack m-u}_{q},
\end{align*}
where we used \eqref{eq:minDeltaS} in the last equality. This completes the proof.
\end{proof}

By Proposition \ref{prop:bijDelta'}, the generating function in \eqref{eq:prop2'} is also the generating function for minimal $n^2$-coloured Frobenius partitions with kernel $S$. Finally, using the fact that any $n^2$-coloured Frobenius partition with colour sequence $C$ of length $s+m$ can be obtained in a unique way by adding a partition into at most $s+m$ parts to $\lambda$ and another partition into at most $s+m$ parts to $\mu$ in the minimal $n^2$-coloured Frobenius partition $\mathrm{min^F}(C)=\begin{pmatrix}
\lambda_1 & \lambda_2 & \cdots & \lambda_{s+m} \\
\mu_1 & \mu_2 & \cdots & \mu_{s+m} 
\end{pmatrix},$
we obtain the following key expression for the generating function of 
$n^2$-coloured Frobenius partitions with a given kernel $S$.
\begin{proposition}
\label{prop:prop2}
Let $n$ be a positive integer and $m$ be a non-negative integer. Let $S= c_1, \dots , c_s$ be a reduced colour sequence of length $s$, having $t$ maximal primary subsequences. Using the notation of Section \ref{subsec:gfkernelDelta}, the generating function for $n^2$-coloured Frobenius partitions with kernel $S$, having length $s+m$, is the following:
\begin{equation}
\label{eq:prop2}
\sum_{\substack{F \in \mathcal{F}_n:\\ \ell(F)=s+m\\\mathrm{ker}(F)=S}} q^{|F|}= \frac{q^{|\min_{\Delta}(S)|+m(s+m+1)}}{(q;q)_{s+m}^2} \sum_{u=0}^{t} q^{-u(t+m)}g_{u,t}(q;|\mathcal{T}_0^1|,\ldots,|\mathcal{T}_0^{t}|) {s+m-1 \brack m-u}_{q}.
\end{equation}
\end{proposition}

\subsection{Equality of generating functions for $\mathcal{F}_n$ and $\mathcal{P}_n$}
Proposition \ref{prop:main} gives the generating function for coloured partitions of $\mathcal{P}_n$ with kernel $S$, and Proposition \ref{prop:prop2} gives the generating function for coloured Frobenius partitions of $\mathcal{F}_n$ with the same kernel $S$. In this section, we show that these two generating functions are actually equal, which will complete the proof of our generalisation of Primc's identity (Theorem \ref{th:Primcgene}).

But before doing so, we need a lemma about $q$-binomial coefficients.
\begin{lemma}
\label{lem:lemma1qbin}
Let $s$ be a positive integer and $m,u$ be two non-negative integer. Then
$$\frac{1}{(q;q)_{s+m}} = \sum_{m' \geq 0} \frac{q^{(m'-u)(s+m')}}{(q;q)_{s+m'}} {m-u \brack m'-u}_{q}.$$
\end{lemma}
\begin{proof}
Let us consider a partition into parts at most $s+m$, generated by $\frac{1}{(q;q)_{s+m}}$.
\begin{figure}[H]
\includegraphics[width=0.5\textwidth]{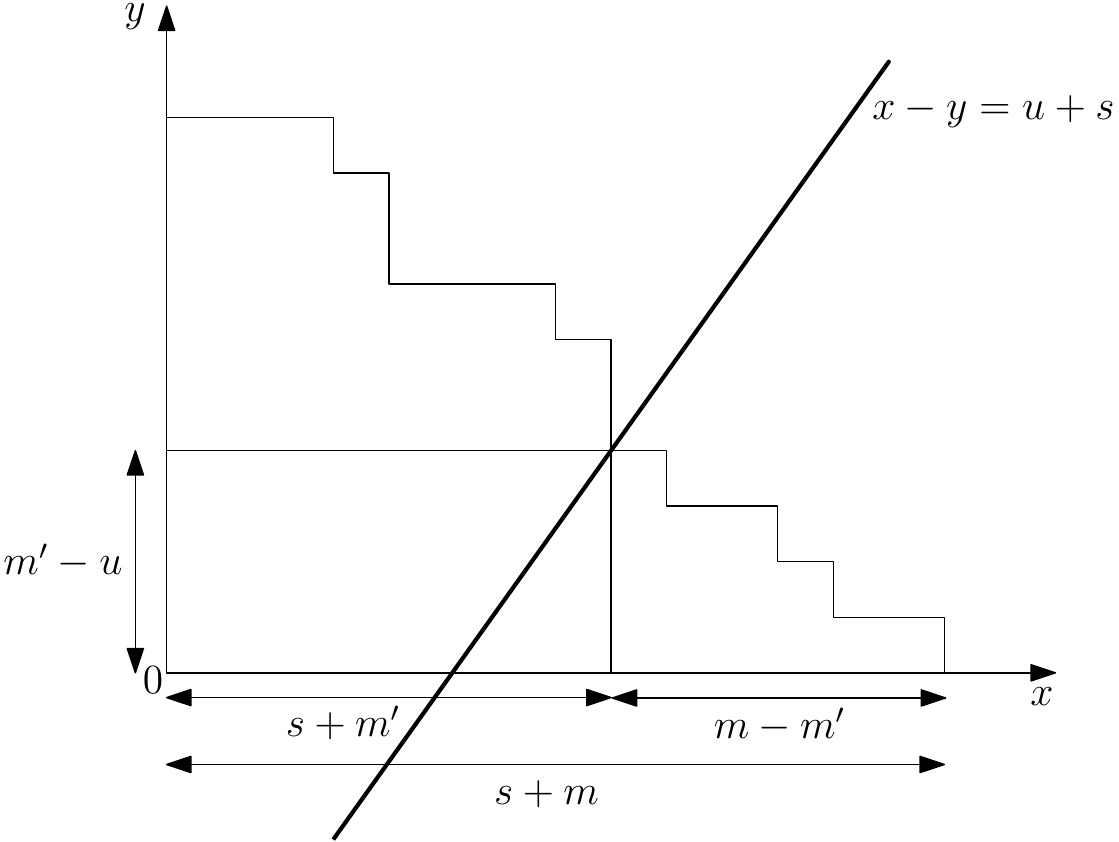}
\caption{Decomposition of the Ferrers board.}
\label{fig:lem1}
\end{figure}
Draw its Ferrers diagram on the plane as shown in Figure \ref{fig:lem1}, and draw the line defined by the equation $x-y=u+s$. This line intersects the boundary of the Ferrers board in a point with coordinates $(s+m',m'-u)$ for some integer $m' \in \{u, \dots , m \}$ (we take the convention that the $x$-axis always belongs to the  boundary of the Ferrers board). It defines three areas in the Ferrers diagram:
\begin{itemize}
\item a rectangle of size $(m'-u)\times(s+m')$ on the bottom-left of the intersection, generated by $q^{(m'-u)(s+m')}$,
\item a partition into parts at most $s+m'$ on top of the rectangle, generated by $\frac{1}{(q;q)_{s+m'}},$
\item a partition with at most $m'-u$ parts, each of size at most $m-m'$, generated by ${m-u \brack m'-u}_{q}$, to the right of the rectangle.
\end{itemize}
Summing over all possible values of $m'$ gives the desired result.
\end{proof}

We are now ready to prove the following theorem, which implies Theorem \ref{th:Primcgene}.
\begin{theorem}
\label{th:Primcgenekernel}
Let $n$ be a positive integer and $m$ be a non-negative integer. Let $S= c_1, \dots , c_s$ be a reduced colour sequence of length $s$, having $t$ maximal primary subsequences. Then
\begin{equation}
\label{eq:thker}
\sum_{\substack{\lambda \in \mathcal{P}_n:\\\mathrm{ker}(\lambda)=S}} q^{|\lambda|}=\sum_{\substack{F \in \mathcal{F}_n:\\\mathrm{ker}(F)=S}} q^{|F|}.
\end{equation}
\end{theorem}
\begin{proof}
By Proposition \ref{prop:main}, we have
\begin{align*}
\sum_{\substack{\lambda \in \mathcal{P}_n:\\\mathrm{ker}(\lambda)=S}} q^{|\lambda|} &= \sum_{m \geq 0} \frac{q^{|\min_{\Delta}(S)|+m}}{(q;q)_{s+m}} \sum_{u=0}^{t} q^{u(s-t)}g_{u,t}(q;|\mathcal{T}_0^1|,\ldots,|\mathcal{T}_0^{t}|) {s+m-1 \brack m-u}_q
\\&= \sum_{u=0}^{t}  q^{|\min_{\Delta}(S)|+u(s-t)}g_{u,t}(q;|\mathcal{T}_0^1|,\ldots,|\mathcal{T}_0^{t}|) \sum_{m \geq 0} \frac{q^{m}}{(q;q)_{s+m}}  {s+m-1 \brack m-u}_q,
\end{align*}
and by Proposition \ref{prop:prop2},
\begin{align*}
\sum_{\substack{F \in \mathcal{F}_n:\\\mathrm{ker}(F)=S}} q^{|F|}&=\sum_{m \geq 0}  \frac{q^{|\min_{\Delta}(S)|+m(s+m+1)}}{(q;q)_{s+m}^2} \sum_{u=0}^{t} q^{-u(t+m)}g_{u,t}(q;|\mathcal{T}_0^1|,\ldots,|\mathcal{T}_0^{t}|) {s+m-1 \brack m-u}_{q}\\
&= \sum_{u=0}^{t} q^{|\min_{\Delta}(S)|+u(s-t)} g_{u,t}(q;|\mathcal{T}_0^1|,\ldots,|\mathcal{T}_0^{t}|) \sum_{m \geq 0} \frac{q^{(m-u)(s+m)+m}}{(q;q)_{s+m}^2} {s+m-1 \brack m-u}_{q}
\end{align*}
Thus, to prove the theorem, it is sufficient to show that for $u \in \{0, \dots, t\},$
\begin{equation}
\label{eq:goal}
 \sum_{m \geq 0} \frac{q^{m}}{(q;q)_{s+m}}  {s+m-1 \brack m-u}_q = \sum_{m \geq 0} \frac{q^{(m-u)(s+m)+m}}{(q;q)_{s+m}^2} {s+m-1 \brack m-u}_{q}.
\end{equation}
Lemma \ref{lem:lemma1qbin} and the definition of $q$-binomial coefficients yields
\begin{align*}
\frac{1}{(q;q)_{s+m}} {s+m-1 \brack m-u}_q &= \sum_{m' \geq 0} \frac{q^{(m'-u)(s+m')}}{(q;q)_{s+m'}} {m-u \brack m'-u}_{q} {s+m-1 \brack m-u}_q\\
&= \sum_{m' \geq 0} \frac{q^{(m'-u)(s+m')}}{(q;q)_{s+m'}} {s+m'-1 \brack m'-u}_{q} {s+m-1 \brack s+m'-1}_q.
\end{align*}
Therefore
\begin{align*}
 \sum_{m \geq 0} \frac{q^{m}}{(q;q)_{s+m}}  {s+m-1 \brack m-u}_q &= \sum_{m \geq 0} \sum_{m' \geq 0} \frac{q^{(m'-u)(s+m')+m}}{(q;q)_{s+m'}} {s+m'-1 \brack m'-u}_{q} {s+m-1 \brack s+m'-1}_q
 \\&= \sum_{m' \geq 0} \frac{q^{(m'-u)(s+m')+m'}}{(q;q)_{s+m'}} {s+m'-1 \brack m'-u}_{q} \sum_{m \geq 0} q^{m-m'} {s+m-1 \brack s+m'-1}_q.
\end{align*}
The last thing to show is that
$$\sum_{m \geq 0} q^{m-m'} {s+m-1 \brack s+m'-1}_q = \frac{1}{(q;q)_{s+m'}},$$
which is true by separating the partitions into at most $s+m'$ parts, generated by $\frac{1}{(q;q)_{s+m'}}$, according to the length $m-m'$ of their largest part.

Thus \eqref{eq:goal} is true and the theorem is proved.
\end{proof}

\subsection{Proof of Theorem \ref{th:main2}}
In the last section, we proved our main theorem (Theorem \ref{th:Primcgene}) relating the generating function for generalised Primc partitions and the one for coloured Frobenius partitions. In this section, we study the particular case where we set $b_i=a_i^{-1}$ for all $i \in \{0, \dots , n-1\}$. All the free colours vanish in the generating function, which can now be written as a sum of infinite products, as stated in Theorem \ref{th:main2}.

Let $n$ be a positive integer.
By Theorem \ref{th:Primcgene} in which we set $b_i=a_i^{-1}$ for all $i$, we have
\begin{align*}
P_n&:=\sum_{m,u_0, \dots, u_{n-1}, v_0, \dots, v_{n-1} \geq 0} P_{n}(m;u_0,\dots,u_{n-1};v_0,\dots,v_{n-1})q^m a_0^{u_0-v_0} \cdots a_{n-1}^{u_{n-1}-v_{n-1}}
\\&= [x^0] \prod_{i=0}^{n-1}(-x a_i q;q)_{\infty}(-x^{-1} a_i^{-1};q)_{\infty}.
\end{align*}
Using the Jacobi triple product \eqref{eq:JTP} in each term of this product, we obtain
\begin{align*}
P_n &= \frac{1}{(q;q)_{\infty}^n} [x^0] \prod_{i=0}^{n-1} \left( \sum_{m_i \in \Z} x^{m_i} a_i^{m_i} q^{\frac{m_i(m_i+1)}{2}}\right)
\\&= \frac{1}{(q;q)_{\infty}^n} \sum_{\substack{m_0, \dots, m_{n-1} \in \Z\\ m_0 + \cdots + m_{n-1} =0}} \left( \prod_{i=0}^{n-1} a_i^{m_i}\right) q^{\sum_{i=0}^{n-1} \frac{m_i(m_i+1)}{2}}.
\end{align*}
Now replacing $m_0$ by $-m_1- \cdots - m_{n-1}$ and using that
$$\frac{m_0(m_0+1)}{2}= \frac{\sum_{i=1}^{n-1} m_i^2 -\sum_{i=1}^{n-1} m_i}{2} + \sum_{1 \leq i < j \leq n-1} m_im_j,$$ we get
\begin{equation}
\label{eq:cor1}
P_n= \frac{1}{(q;q)_{\infty}^n} \sum_{m_1, \dots, m_{n-1} \in \Z} \left( \prod_{i=1}^{n-1} (a_ia_0^{-1})^{m_i}\right) q^{\sum_{i=1}^{n-1} m_i^2 + \sum_{1 \leq i < j \leq n-1} m_im_j}.
\end{equation}

We want to apply the Jacobi triple product again inside the sum, in order to obtain a sum of infinite products. To do so, we perform some changes of variables. We first need the following lemma.

\begin{lemma}
\label{lem:23'}
Let $$M(n):= \sum_{i=1}^{n-1} m_i^2 + \sum_{1 \leq i < j \leq n-1} m_im_j.$$
Let $s_n=0$ and for all $i \in \{1, \dots, n-1 \},$ 
$$s_i:= \sum_{j=i}^{n-1} m_j.$$
Then we have
$$M(n) = \sum_{i=1}^{n-1} s_i(s_i-s_{i+1}) =  \sum_{i=1}^{n-1} \frac{((i+1)s_i-is_{i+1})^2}{2i(i+1)}.$$
\end{lemma}
\begin{proof}
The first equality follows directly from the definition of the $s_i$'s.

Let us now prove the second equality. We have
\begin{align*}
\sum_{i=1}^{n-1} \frac{((i+1)s_i-is_{i+1})^2}{2i(i+1)} &= \sum_{i=1}^{n-1} \left(\frac{i+1}{2i}s_i^2 -s_is_{i+1} + \frac{i}{2(i+1)}s_{i+1}^2 \right)
\\&= -\sum_{i=1}^{n-1} s_is_{i+1} +s_1^2 +  \sum_{i=2}^{n-1} \left(\frac{i+1}{2i}s_i^2 + \frac{i-1}{2i}s_{i}^2 \right)
\\&=\sum_{i=1}^{n-1} s_i(s_i-s_{i+1}),
\end{align*}
where the second equality follows from the change of variable $i\rightarrow i-1$ in the last sum.
\end{proof}
By Lemma \ref{lem:23'} and \eqref{eq:cor1}, we obtain
\begin{align*}
P_n &= \frac{1}{(q;q)_{\infty}^n} \sum_{\substack{s_1, \dots, s_{n-1}\in \Z\\s_n=0}} \left( \prod_{i=1}^{n-1} (a_ia_0^{-1})^{s_i-s_{i+1}}\right) q^{\sum_{i=1}^{n-1} s_i(s_i-s_{i+1})}\\
&=\frac{1}{(q;q)_{\infty}^{n}}\sum_{\substack{s_1, \dots, s_{n-1}\in \Z\\s_n=0}} a_0^{-s_1}\prod_{i=1}^{n-1} a_i^{s_i-s_{i+1}}q^{s_i(s_i-s_{i+1})}.
\end{align*}
This is \eqref{eq:jacob}. Let us perform a few more changes of variables to obtain \eqref{eq:formulefinale}.

\medskip
For all $i \in \{1, \dots, n-1 \},$ let us write $s_i=i \times d_i + r_i$, with $r_i \in \{0, \dots, i-1 \}.$ This is euclidean division by $i$, so this expression is unique, and for $r_1, \dots , r_{n-1}$ fixed, there is a bijection between $\{(s_1, \dots , s_{n-1}) \in \Z^{n-1} : s_i \equiv r_i \mod i\}$ and $\{(d_1, \dots , d_{n-1}) \in \Z^{n-1}\}.$
Moreover our choice $s_n=0$ corresponds to $d_n=r_n=0.$ We obtain
$$M(n)=\sum_{i=1}^{n-1} \left(\frac{i(i+1)}{2} (d_i-d_{i+1})^2 + \frac{((i+1)r_i-ir_{i+1})^2}{2i(i+1)} + (d_i-d_{i+1})((i+1)r_i-ir_{i+1}) \right).$$

By a last change of variable $p_i = d_i-d_{i+1}$, equivalent to $d_i= \sum_{j=i}^{n-1} p_j$, $\{(d_1, \dots , d_{n-1}) \in \Z^{n-1}\}$ is in bijection with $\{(p_1, \dots , p_{n-1}) \in \Z^{n-1}\}$. %Our previous conventions correspond to $d_n=0$.
This yields
\begin{align*}
M(n) &=\sum_{i=1}^{n-1} \left(\frac{i(i+1)}{2} p_i^2 + \frac{((i+1)r_i-ir_{i+1})^2}{2i(i+1)} + p_i((i+1)r_i-ir_{i+1}) \right)
\\&= \sum_{i=1}^{n-1} r_i(r_i-r_{i+1}) + \sum_{i=1}^{n-1} \left(\frac{i(i+1)}{2} p_i^2  + p_i((i+1)r_i-ir_{i+1}) \right)
\end{align*}

Backtracking all these changes of variables, we have for all $i \in \{1, \dots, n-1 \},$
$$
\begin{array}{llll}
m_i &=& s_i-s_{i+1} &(\text{with } s_n=0)\\
&=& id_i+r_i -(i+1)d_{i+1}-r_{i+1} & (\text{with } d_n=r_n=0)\\
&=& i \sum_{j=i}^{n-1} p_j +r_i - (i+1) \sum_{j=i+1}^{n-1} p_j - r_{i+1} & \\
&=& i p_i - \sum_{j=i+1}^{n-1} p_j +r_i - r_{i+1}.
\end{array}
$$
Thus, by the above and Lemma \ref{lem:23'}, the generating function in \eqref{eq:cor1} becomes
\begin{equation}
\label{eq:cor1bis}
\begin{aligned}
P_n = \frac{1}{(q;q)_{\infty}^n} \sum_{\substack{r_1, \dots, r_{n-1}\\0 \leq r_j \leq j-1}} \sum_{p_1, \dots, p_{n-1} \in \Z} & \left( \prod_{i=1}^{n-1} (a_ia_0^{-1})^{i p_i - \sum_{j=i+1}^{n-1} p_j +r_i - r_{i+1}}\right)
\\& \times q^{\sum_{i=1}^{n-1} r_i(r_i-r_{i+1}) + \sum_{i=1}^{n-1} \left(\frac{i(i+1)}{2} p_i^2  + p_i((i+1)r_i-ir_{i+1})\right)}.
\end{aligned}
\end{equation}
It can be shown by induction on $n$ that
$$\prod_{i=1}^{n-1} (a_ia_0^{-1})^{i p_i - \sum_{j=i+1}^{n-1} p_j} = \prod_{i=1}^{n-1} \left(\prod_{\ell=0}^{i-1} a_i a_{\ell}^{-1} \right)^{p_i}.
$$
Therefore reorganising \eqref{eq:cor1bis} leads to
\begin{align*}
P_n &=  \frac{1}{(q;q)_{\infty}^n} \sum_{\substack{r_1, \dots, r_{n-1}\\0 \leq r_j \leq j-1}} \left( \prod_{i=1}^{n-1} (a_ia_0^{-1})^{r_i - r_{i+1}} q^{r_i(r_i-r_{i+1})} \right) 
\\ &\qquad \qquad \quad \times \sum_{p_1, \dots, p_{n-1} \in \Z} \prod_{i=1}^{n-1} \left( \left(\prod_{\ell=0}^{i-1} a_i a_{\ell}^{-1} \right) q^{(i+1)r_i-ir_{i+1}}\right)^{p_i} q^{\frac{i(i+1)}{2} p_i^2}\\
&=\frac{1}{(q;q)_{\infty}^n} \sum_{\substack{r_1, \dots, r_{n-1}\\0 \leq r_j \leq j-1}} \left( \prod_{i=1}^{n-1} a_i^{r_i - r_{i+1}} q^{r_i(r_i-r_{i+1})} \right) 
\\ &\qquad \qquad \quad  \times \prod_{i=1}^{n-1} \sum_{p_1, \dots, p_{n-1} \in \Z} \left( \left(\prod_{\ell=0}^{i-1} a_i a_{\ell}^{-1} \right) q^{-\frac{i(i+1)}{2}+(i+1)r_i-ir_{i+1}}\right)^{p_i} q^{i(i+1)\frac{p_i(p_i+1)}{2}}
\\&=\frac{1}{(q;q)_{\infty}^n} \sum_{\substack{r_1, \dots, r_{n-1}\\ 0 \leq r_j \leq j-1}} \prod_{i=1}^{n-1} a_i^{r_i-r_{i+1}}q^{r_i(r_i-r_{i+1})}
\\& \qquad \qquad \quad \times  \left(q^{i(i+1)};q^{i(i+1)}\right)_{\infty}
\left(- \left(\prod_{\ell=0}^{i-1} a_i a_{\ell}^{-1} \right) q^{\frac{i(i+1)}{2}+(i+1)r_i-ir_{i+1}};q^{i(i+1)}\right)_{\infty}
\\& \qquad \qquad \quad \times \left(- \left(\prod_{\ell=0}^{i-1} a_{\ell} a_{i}^{-1}\right) q^{\frac{i(i+1)}{2}-(i+1)r_i+ir_{i+1}};q^{i(i+1)}\right)_{\infty},
\end{align*}
where in the last equality, we used Jacobi's triple product identity in each of the sums over $p_i$ for $i \in \{1,\dots,n-1\}$. Theorem \ref{th:main2} is proved.

\begin{remark}
Andrews \cite{AndrewsFrob} gave the particular cases $n=1,2,3$ of this formula, but without keeping  track of the colours. Our result is more general, as it is both valid for all $n$ and keeps track of the colours.
\end{remark}

\section{Bijection between $\mathcal{P}_n$ and $\mathcal{C}_n(\delta,\gamma)\times \mathcal{P}^0$}
\label{sec:bij}
Now that we have established the connection between the generalised Primc partitions $\mathcal{P}_n$ and the $n^2$-coloured Frobenius partitions $\mathcal{F}_n$, this section is dedicated to the proof of Theorem \ref{th:bij}, which connects generalised Primc partitions with generalised Capparelli partitions. This connection is key in proving our generalisation of Capparelli's identity (Theorem \ref{th:Capgene}).

The bijection in this section generalises the first author's bijection between $\mathcal{P}_2$ and $\mathcal{C}_2 \times \mathcal{P}^0$ in \cite{DousseCapaPrimc}. However, the partitions in $\mathcal{P}_n$ and $\mathcal{C}_n(\delta,\gamma)$ have much more intricate combinatorial descriptions for general $n$, so that it is better to reformulate and simplify the bijection between $\mathcal{P}_2$ and $\mathcal{C}_2 \times \mathcal{P}^0$ before generalising it.

\subsection{Reformulation of Dousse's bijection between $\mathcal{P}_2$ and $\mathcal{C}_2 \times \mathcal{P}^0$}
\label{subsec:bijP2CC2}
We first give a variant of the bijection of \cite{DousseCapaPrimc}.
The one-to-one correspondence is the same, but the intermediate steps are different.

Let $(\lambda,\mu) \in \mathcal{C}_2 \times \mathcal{P}^0$ be a partition pair of total weight $m$, where $\lambda \in \mathcal{C}_2$ and $\mu$ is an unrestricted partition coloured $b$. The idea from \cite{DousseCapaPrimc} is to insert the parts of $\mu$ inside $\lambda$ and to modify the colour of certain parts in order to obtain a partition in $\mathcal{P}_2$, all in a bijective way. Here we keep the same idea but perform the insertions in a different order, making the resulting partitions easier to describe at each step.

To make the comparison with \cite{DousseCapaPrimc} clear, we illustrate our variant of the bijection on the same example
\begin{align*}
\lambda &= 8_d+8_a+6_c+5_c+3_d+1_a,
\\ \mu &= 8_b +8_b + 7_b+ 5_b+3_b+2_b+2_b+1_b+1_b.
\end{align*}

\medskip

First of all, recall that $\lambda\in \mathcal{C}_2$ satisfies the difference conditions from 
$$C_2=\bordermatrix{\text{} & a & c & d \cr a & 2 & 2 & 2 \cr c & 1 & 1 & 2 \cr d & 0 & 1 & 2}.$$
Note also that the column and row $b$ in matrix $P_2$ from \eqref{Primcmatrix4} mean that if there is a part $k_b$ in the partition, then it can repeat but the number $k$ cannot appear in any other colour.

\textbf{Step 1:} For all $j$, if there are some parts of size $j$ in $\mu$ but none in $\lambda$, then move these parts from $\mu$ to $\lambda$. Call $\lambda_1$ and  $\mu_1$ the resulting partitions.

In our example, we obtain
\begin{align*}
\lambda_1 &= 8_d+8_a+ 7_b+6_c+5_c+3_d+2_b+2_b+1_a,
\\ \mu_1 &= 8_b +8_b+5_b+3_b+1_b+1_b.
\end{align*}

The pair $(\lambda_1,\mu_1)$ is such that $\lambda_1$ satisfies the difference conditions in the matrix
\begin{equation} \label{N1matrix}
C_2^1=\bordermatrix{\text{} & a & b & c & d \cr a & 2&1&2&2 \cr b &1&0&1&1 \cr c &1&1&1&2 \cr d&0&1&1&2},
\end{equation}
and $\mu_1$ is a partition coloured $b$ containing only parts of sizes that also appear in $\lambda_1$ but in a colour different from $b$. Indeed, in $\lambda_1$, there can now be some parts coloured $b$  which can repeat and are distinct from all the other parts, and the minimal differences between parts coloured $a,c,d$ are the same as before.

This process is reversible, as one can simply move the $b$-parts of $\lambda_1$ back to $\mu_1$.

\medskip

\textbf{Step 2:}
For all $j$, if there are some parts $j_b$ in $\mu_1$, and $j_c$ appears in $\lambda_1$ (by \eqref{N1matrix}, it cannot repeat nor appear in any other colour), then transform those $j_b$'s into $j_c$'s and move them from $\mu_1$ to $\lambda_1$. Call $\lambda_2$ and  $\mu_2$ the resulting partitions.

In our example, we obtain
\begin{align*}
\lambda_2 &= 8_d+8_a+ 7_b+6_c+5_c+5_c+3_d+2_b+2_b+1_a,
\\ \mu_2 &= 8_b +8_b+3_b+1_b+1_b.
\end{align*}

Now the parts coloured $c$ can repeat, and the rest of the partition was not affected at all.
Thus the pair $(\lambda_2,\mu_2)$ is such that $\lambda_2$ satisfies the difference conditions in the matrix
\begin{equation} \label{N2matrix}
C_2^2=\bordermatrix{\text{} & a & b & c & d \cr a & 2&1&2&2 \cr b &1&0&1&1 \cr c &1&1&0&2 \cr d&0&1&1&2},
\end{equation}
and $\mu_2$ is a partition coloured $b$ containing only parts of sizes that also appear in $\lambda_2$ with colour $a$ or $d$.

This process is also reversible. If in $\lambda_2$, there is a $c$-coloured part $j_c$ that repeats, then transform all but one of these $j_c$'s into $j_b$'s and move them to $\mu_2$.

\medskip

\textbf{Step 3:}
For all $j$, if there are some parts $j_b$ in $\mu_2$, then $j$ appears in $\lambda_2$ in colour $a$ or $d$, but not $c$. Transform those $j_b$'s into $j_c$'s and insert them inside $\lambda_2$, with the colour order $a < c < d.$ Call $\lambda_3$ the resulting partition.

In our example, we obtain
$$\lambda_3 = 8_d+8_c +8_c+8_a+ 7_b+6_c+5_c+5_c+3_d+3_c+2_b+2_b+1_c+1_c+1_a.$$

Now the minimal difference between parts of colour $c$ and $a$ (resp. $d$ and $c$) is $0$, and the other difference conditions were not affected at all.
Thus the partition $\lambda_3$ satisfies exactly the difference conditions of Primc's matrix $P_2$ from \eqref{Primcmatrix4}.

This final step is also reversible. If in $\lambda_3$, there are some parts $j_c$ such that $j_a$ or $j_d$ also appears, then transform those $j_c$'s into $j_b$'s, remove them from $\lambda_3$,  and put them in a separate partition $\mu_2$.

\medskip
We obtain the same final partition as in \cite{DousseCapaPrimc}, only the intermediate steps are different.

All the steps in this bijection preserve the weight, the number of parts, the sizes of the parts, and the number of $a$-parts and $d$-parts. This proves Theorem \ref{th:bij} in the case $n=2$. 

In the remainder of this section, we generalise this bijection for all $n$.

\subsection{Preliminary observations}
\label{subsec:prelim}
Before giving the bijection which proves the generalisation of Capparelli's identity, we start with a few observations to better understand the combinatorial structure of the difference conditions $\Delta$.

Let us start by rewriting $\Delta$ in a more explicit form depending on whether it is applied to bound or free colours.
\begin{proposition}
Let $n$ be a positive integer. For all $i,j,k,\ell\in \{0,\ldots,n-1\}$, we have the following expressions for $\Delta(a_i b_j, a_{k} b_{\ell})$.
\label{prop:reformulation_delta}
\begin{itemize}
\item If $\Delta$ is applied to two free colours,
\begin{equation}
\label{eq:ff}
 \Delta(a_i b_i, a_{k} b_{k})=\chi(i\neq k).
\end{equation}
\item If $\Delta$ is applied to a free and a bound colour such that  $i<j$, then 
\begin{align}
\Delta(a_i b_j, a_{k} b_{k}) &= 1-\chi(i< k \leq j)\label{eq:upf}\\
\Delta(a_k b_k, a_{i} b_{j}) &= 1+\chi(i< k \leq j)\label{eq:fup}.
\end{align}
\item If $\Delta$ is applied to a free and a bound colour such that $i>j$, then 
\begin{align}
\Delta(a_i b_j, a_{k} b_{k}) &= 1+\chi(i \geq  k > j)\label{eq:lowf}\\
\Delta(a_k b_k, a_{i} b_{j}) &= 1-\chi(i \geq  k > j)\label{eq:flow}.
\end{align}
\item If $\Delta$ is applied to two free colours, i.e. $i\neq j$ and $k\neq \ell$, then
\begin{equation}
\label{eq:bb}
\Delta(a_i b_j, a_{k} b_{\ell}) = \chi(i\geq k)+\chi(j\leq \ell). 
\end{equation}
In particular, if $i<j$ and $k>\ell$, we have 
\begin{equation}
\label{eq:uplow}
\Delta(a_i b_j, a_{k} b_{\ell}) = 1-\chi(i<k)\chi(j>\ell).
\end{equation}
\end{itemize}
\end{proposition}

\begin{proof}
These formulas are straightforward reformulations of the definition \eqref{eq:Delta} of $\Delta$ in some particular cases.

We only give details for \eqref{eq:lowf} as an example.
We have $i>j$, so by definition $\chi(i=j=k)=0$. Thus by \eqref{eq:Delta}, we have
\begin{align*}
\Delta(a_i b_j, a_{k} b_{k}) &= \chi(i \geq k)+\chi(j \leq k) - \chi(j=k)\\
&= \chi(i \geq k)+\chi(j < k).
\end{align*}
Thus $\Delta(a_i b_j, a_{k} b_{k})$ is equal to $1$ when $k>i$ or $k \leq j$, and to $2$ when $i \geq  k > j$, which yields \eqref{eq:lowf}.
\end{proof}

We can deduce a simple remark about the colour $a_0b_0$, which plays a particular role in our reasoning as it is forbidden in generalised Capparelli partitions.
\begin{remark}
We have $\Delta(a_0b_0,a_0b_0)=0$, and for all $c \neq a_0b_0$,
$$\Delta(c,a_0b_0)= \Delta(a_0b_0,c)=1.$$
This means that, in generalised Primc partitions, the colour $a_0b_0$ can repeat, but if there is an integer $k$ of colour $a_0b_0,$ then $k$ cannot appear in any other colour. This is the only restriction involving $a_0b_0$.
\end{remark}

Our bijection will rely on the insertion of parts with free colours inside sequences of parts of the same size, so we need to understand the combinatorics of these sequences. The first step towards this is understanding pairs of colours $(c,c')$ such that $\Delta(c,c')=0$. This is done in the following proposition, which is a straightforward consequence of Proposition \ref{prop:reformulation_delta}. 

\begin{proposition}
\label{prop:Delta=0}
A pair of colours $(c,c')$ satisfies $\Delta(c,c')=0$ if and only if it satisfies one of the following four conditions:
\begin{enumerate}
\item $c=c'$ and $c$ is a free colour,
\item $c=a_ib_j$ is a bound colour (i.e. $i \neq j$), $c'=a_kb_k$ is a free colour , and $i < k \leq j$,
\item $c=a_kb_k$ is a free colour, $c'=a_ib_{j}$ is a bound colour (i.e. $i \neq j$), and $j < k \leq i$,
\item $c=a_ib_j$ and $c'=a_kb_{\ell}$ are both bound colours (i.e. $i \neq j$ and $k \neq \ell$), and $i < k$ and $j > \ell$.
\end{enumerate}
\end{proposition}

%\begin{proof}
%(1) This follows easily from Properties \ref{property:aibiaibi} and \ref{property:aibiakbk}.
%
%\noindent
%(2) By the definition \eqref{eq:Delta} of $\Delta$, we have
%$$\Delta(a_ib_i,a_kb_{\ell}) = \chi(i \geq k)-\chi(i=k)+\chi(i\leq \ell).$$
%If $i=k$, then $\Delta(a_ib_i,a_kb_{\ell})=0$ if and only if $i >\ell$. If $i \neq k$, then $\Delta(a_ib_i,a_kb_{\ell})=0$ if and only if $\ell < i <k$. Both cases can be summed up as $\ell < i \leq k$.
%
%\noindent
%(3) Again by the definition of $\Delta$, we have
%$$\Delta(a_ib_j,a_kb_k) = \chi(i \geq k)+\chi(j\leq k)-\chi(j=k).$$
%If $j=k$, then $\Delta(a_ib_j,a_kb_k)=0$ if and only if $i <k$. If $j \neq k$, then $\Delta(a_ib_j,a_kb_k)=0$ if and only if $i <k <j$. Both cases can be summed up as $i <k \leq j$.
%
%\noindent
%(4) Finally,
%$$\Delta(a_ib_j,a_kb_{\ell}) = \chi(i \geq k)+\chi(j\leq \ell).$$
%This is zero if and only if $i < k$ and $j > \ell$.
%\end{proof}

The principle of our bijection is to insert/remove parts with free colours in/from generalised Primc partitions. Therefore we want to know where we can insert these parts without violating the difference conditions $\Delta$. 
When we want to insert a part of size $k$ in a generalised Primc partition which does not have any part of size $k$, we can simply insert this part with colour $a_0b_0$ without disrupting any other difference, and this is what we will do in our bijection. So it remains to study how one can insert a part of size $k$ in a generalised Primc partition which already contains parts of size $k$, without disrupting the differences.
To do so exhaustively, we have to consider how such an insertion interacts with parts of the same size $k$ and with parts with size $k+1$ and $k-1$.

We start by studying interactions with parts of the same size.
Proposition \ref{prop:Delta=0} allows us to understand exactly the shape of the colour sequences of sub-partitions where all the parts have the same value.

\begin{proposition}
\label{prop:seq0000}
Let $C=c_1 \cdots c_s$ be a sequence of colours such that for all $i \in \{1 , \dots, s-1\},$ $\Delta(c_i,c_{i+1})=0.$ Then, writing for all $i$, $c_i=a_{k_i}b_{\ell_i}$, the sequence $C$ satisfies one of the following:

\smallskip
\noindent
\textbf{Case 1} : There is exactly one free colour $c_i$ in $C$ (which may repeat an arbitrary number $j$ of times). In this case, the inequalities between then $k_i$'s and $\ell_i$'s can be summarised as follows, where the numbers below indicate which case of Proposition \ref{prop:Delta=0} each pair of inequalities correspond to.
\begin{equation}
\label{tab:onefreecolour}
\arraycolsep=2pt\def\arraystretch{1.2}
\begin{array}{c|ccccccccccccccccccc}
&c_1&&c_2& &\cdots &&c_{i-1}&&c_i&&\dots&&c_{i+j-1}&&c_{i+j}&&\cdots&&c_s\\
\hline
index(a)&k_1&<&k_2&<&\cdots&<&k_{i-1}&<&k_i&=&\cdots&=&k_i&\leq&k_{i+j}&<&\cdots&<&k_s\\
&&&&&&&\wedge&&&&&&&&\vee\\
index(b)&\ell_1&>&\ell_2&>&\cdots&>&\ell_{i-1}&\geq&k_i&=&\cdots&=&k_i&>&\ell_{i+j}&>&\cdots&>&\ell_s\\
\text{Case in Prop. \ref{prop:Delta=0}}&&(4)&&(4)&\cdots&(4)&&(2)&&(1)&\cdots&(1)&&(3)&&(4)&\cdots&(4)
\end{array}
\end{equation}
There are three possible sub-cases:

\noindent
Case 1a: the free colour is on the left end ($i=1$). 

\noindent
Case 1b: the free colour is on the right end ($i+j-1=s$).

\noindent
Case 1c: there are bound colours on both sides of the free colour ($i \neq 1, s+1-j$).

\smallskip
\noindent
\textbf{Case 2} : There is no free colour in $C$. In this case, the inequalities between then $k_i$'s and $\ell_i$'s can be summarised as follows, where all the inequalities come from Case (4) of Proposition \ref{prop:Delta=0}.
\begin{equation}
\label{tab:nofreecolour}
\arraycolsep=2pt\def\arraystretch{1.2}
\begin{array}{c|ccccccccccccccccccc}
&c_1&&c_2& &\cdots &&c_i&&c_{i+1}&&\cdots&&c_s\\
\hline
index(a)&k_1&<&k_2&<&\cdots&<&k_i&<&k_{i+1}&<&\cdots&<&k_s\\
index(b)&\ell_1&>&\ell_2&>&\cdots&>&\ell_i&>&\ell_{i+1}&>&\cdots&>&\ell_s
\end{array}
\end{equation}
There are three possible sub-cases:

\noindent
Case 2a: for all $i\in\{1,\dots,s\}, k_i >\ell_i$.

\noindent
Case 2b: for all $i\in\{1,\dots,s\}, k_i <\ell_i$.

\noindent
Case 2c: there is exactly one $i\in\{1,\dots,s\}$ such that $k_i <\ell_i$ and $k_{i+1} >\ell_{i+1}$.
\end{proposition}
\begin{proof}
The fact that there is at most one free colour in $C$ follows from the triangle inequality satisfied by $\Delta$. Assume that there are two different free colours $c_i$ and $c_{i+j}$ in $C$, then by the triangle inequality, we have $1=\Delta(c_i,c_j) \leq \Delta(c_i,c_i+1) + \cdots + \Delta(c_{j-1},c_j)$, contradicting the fact that each term in this sum is $0$.

The inequalities presented in the tables above follow from a straightforward application of Proposition \ref{prop:Delta=0}.
The last thing to check is that Cases 2a, 2b, and 2c are exhaustive. Assume for the purpose of contradiction that there are two indices $i < j$ such that $k_i <\ell_i$, $k_{i+1} >\ell_{i+1}$, $k_j <\ell_j$, and $k_{j+1} >\ell_{j+1}$. First, $j$ is bigger than $i+1$, otherwise we would have both $k_{i+1} >\ell_{i+1}$ and $k_{i+1} <\ell_{i+1}$. Now by \eqref{tab:nofreecolour}, we have
$$k_j>k_{i+1}>\ell_{i+1}>\ell_j>k_j,$$
which is a contradiction.
\end{proof}

Proposition \ref{prop:seq0000} allows us to characterise the insertions of free colours that can be performed in the colour sequences in Case $2$ without disrupting the other differences in the partition.
\begin{proposition}
\label{prop:insertion0000}
Let $C=c_1 \cdots c_s$ be a sequence of bound colours such that for all $i \in \{1 , \dots, s-1\},$ $\Delta(c_i,c_{i+1})=0.$ Then, writing for all $i$, $c_i=a_{k_i}b_{\ell_i}$, the insertions of free colours that we can perform in $C$ are exactly the following.

\begin{itemize}
\item If $C$ is in Case 2a, then we can insert the free colour $a_kb_k$ to the left of $c_1$, where $k$ is such that
$$\ell_1<k\leq k_1.$$
The sequence we obtain is in Case 1a.
\item If $C$ is in Case 2b, then we can insert the free colour $a_kb_k$ to the right of $c_s$, where $k$ is such that
$$k_s<k\leq \ell_s.$$
The sequence we obtain is in Case 1b.
\item If $C$ is in Case 2c, where $k_i <\ell_i$ and $k_{i+1} >\ell_{i+1}$, then we can insert the free colour $a_kb_k$ between $c_i$ and $c_{i+1}$, where $k$ is such that
$$k_i<k\leq k_{i+1} \quad \text{ and} \quad \ell_{i+1}<k\leq \ell_i.$$
The sequence we obtain is in Case 1c.
\end{itemize}
\end{proposition}

\begin{remark}
In Case 2c, we have 
$$\max\{k_i,\ell_{i+1}\} < k \leq \min\{\ell_i,k_{i+1}\}.$$
\end{remark}

\medskip
Now we study how we can insert a part with free colour between two parts with bound colours whose sizes differ by $1$. We distinguish three cases.

\begin{proposition}
\label{prop:prop0.1}
Let $n$ be a positive integer, and let $c_1=a_{k_1}b_{\ell_1}$ and $c_2=a_{k_2}b_{\ell_2}$  with $k_1<\ell_1$ and $k_2>\ell_2$. For every positive integer $p$, and every $i$ and $j$ such that $k_1<i\leq \ell_1$ and $\ell_2<j\leq k_2$, we can insert $(p+1)_{a_ib_i}$ and $p_{a_jb_j}$ between $(p+1)_{c_1}$ and $p_{c_2}$ without disrupting the other differences in the partition. In other words,
$$(p+1)_{c_1}+\underbrace{(p+1)_{a_ib_i}+\cdots+(p+1)_{a_ib_i}}_{\text{possibily empty}}+\underbrace{p_{a_jb_j}+\cdots+p_{a_jb_j}}_{\text{possibily empty}}+p_{c_2} \in \mathcal{P}_n.$$
\end{proposition}

\begin{proof}
We know that $(p+1)_{c_1}+p_{c_2} \in \mathcal{P}_n$ by \eqref{eq:uplow}.

We then check that $(p+1)_{c_1}+(p+1)_{a_ib_i}+p_{a_jb_j}+p_{c_2} \in \mathcal{P}_n$ by using \eqref{eq:upf},  \eqref{eq:ff}, and \eqref{eq:flow}. The empty cases follow from the triangle inequality, and we have shown before that repetition of free colours do not modify any other differences.
\end{proof}

Let us turn to the second case.
\begin{proposition}
\label{prop:prop0.2}
Let $n$ be a positive integer, and let $c_1=a_{k_1}b_{\ell_1}$ and $c_2=a_{k_2}b_{\ell_2}$ with $k_1>\ell_1$ and $k_2>\ell_2$. For every positive integers $p$ and $i$, we have
$$(p+1)_{c_1}+\underbrace{p_{a_ib_i}+\cdots+p_{a_ib_i}}_{\text{possibily empty}}+p_{c_2} \in \mathcal{P}_n$$
if and only if ($k_1<k_2$ or $\ell_1>\ell_2$) and $i\in \{\ell_2+1,\ldots,k_2\}\setminus\{\ell_1+1,\ldots,k_1\}$. 
\end{proposition}

\begin{proof} 
Note that $(p+1)_{c_1}+p_{c_2} \in \mathcal{P}_n$ if and only if  
$\Delta(c_1,c_2)\neq 2,$ i.e. if and only if $k_1<k_2$ or $\ell_1>\ell_2.$

By the triangle inequality, we also need  $k_1<k_2$ or $\ell_1>\ell_2$ in order to have $(p+1)_{c_1}+p_{a_ib_i}+p_{c_2} \in \mathcal{P}_n.$
Moreover when $k_1<k_2$ or $\ell_1>\ell_2$, we have $(p+1)_{c_1}+p_{a_ib_i}+p_{c_2} \in \mathcal{P}_n$ if and only if $i\notin \{\ell_1+1,\ldots,k_1\}$ ( by \eqref{eq:lowf}) and $i\in \{\ell_2+1,\ldots,k_2\}$ (by \eqref{eq:flow}).
\end{proof}

We finish with the last possible case.
\begin{proposition}
\label{prop:prop0.3}
Let $n$ be a positive integer, and let $c_1=a_{k_1}b_{\ell_1}$ and $c_2=a_{k_2}b_{\ell_2}$ with $k_1<\ell_1$ and $k_2<\ell_2$. For every positive integers $p$ and $i$, we have
$$(p+1)_{c_1}+\underbrace{(p+1)_{a_ib_i}+\cdots+(p+1)_{a_ib_i}}_{\text{possibily empty}}+p_{c_2} \in \mathcal{P}_n$$
if and only if ($k_1<k_2$ or $\ell_1>\ell_2$) and $i\in \{k_1+1,\ldots,\ell_1\}\setminus\{k_2+1,\ldots,\ell_2\}$. 
\end{proposition}

\begin{proof} 
Again, $(p+1)_{c_1}+p_{c_2} \in \mathcal{P}_n$ if and only if  
$\Delta(c_1,c_2)\neq 2,$ i.e. if and only if $k_1<k_2$ or $\ell_1>\ell_2,$
and we also need  $k_1<k_2$ or $\ell_1>\ell_2$ to have $(p+1)_{c_1}+(p+1)_{a_ib_i}+p_{c_2} \in \mathcal{P}_n.$
Moreover when $k_1<k_2$ or $\ell_1>\ell_2$, we have $(p+1)_{c_1}+(p+1)_{a_ib_i}+p_{c_2} \in \mathcal{P}_n$ if and only if $i\notin \{k_2+1,\ldots,\ell_2\}$ ( by \eqref{eq:fup}) and $i\in \{k_1+1,\ldots,\ell_1\}$ (by \eqref{eq:upf}).
\end{proof}

\medskip

The idea behind using forbidden patterns to define our generalised Capparelli partitions $\mathcal{C}_n(\delta,\gamma)$ (Definition \ref{def:Capp_conditions}) is the following. When a part $p$ could be inserted in several free colours (given by the propositions above) beside a bound part of the same size, our forbidden patterns give a rule to distinguish one of these possible insertions. Such a distinguished insertion will be forbidden in generalised Capparelli partitions, while it is allowed in generalised Primc partitions. This leads us to construct a relatively simple bijection in the next section.

\subsection{The bijection}
Now that we have understood how free colours can be inserted in a generalised Primc partition to still respect the difference conditions and understood where the forbidden patterns come from, we can finally give explicitly the bijection $\Phi$ from Theorem \ref{th:bij}.

In this whole section, $n$ is a fixed positive integer and $\delta$ and $\gamma$ are fixed functions satisfying Conditions 1 and 2 of Definitions \ref{def:cond1} and \ref{def:cond2}, respectively.
We now give our bijection $\Phi$ between the set $\mathcal{P}_n$ of generalised Primc partitions and the product set $\mathcal{C}_n(\delta,\gamma)\times \mathcal{P}^0$, where $\mathcal{C}_n(\delta,\gamma)$ is the set of generalised Capparelli partitions related to $\delta$ and $\gamma$, and $\mathcal{P}^0$ is the set of classical partitions where all parts are coloured $a_0b_0$.

\subsubsection{The bijection $\Phi$ from $\mathcal{P}_n$ to $\mathcal{C}_n(\delta,\gamma)\times \mathcal{P}^0$}
We start with the transformation $\Phi$ from $\mathcal{P}_n$ to $\mathcal{C}_n(\delta,\gamma)\times \mathcal{P}^0$, which can be explained very easily thanks to the forbidden patterns. 

Let us consider a partition $\lambda\in \mathcal P_n$. The steps of the bijection will be illustrated on the following example, corresponding to Meurman--Primc's $8$-coloured identity, i.e. $n=3$, $\delta=\delta_1$, $\gamma=\gamma_1$:
$$\lambda = 8_{a_1b_1}+6_{a_0b_2}+6_{a_2b_2}+5_{a_0b_1}
+5_{a_1b_0}+4_{a_0b_0}+4_{a_0b_0}
+3_{a_0b_2}+3_{a_1b_1}+3_{a_1b_1}+3_{a_1b_0}+2_{a_2b_2}
+2_{a_2b_2}+2_{a_2b_2}+2_{a_2b_0}+1_{a_0b_0}.$$

First, recall that $\Delta(c,a_0b_0) = \Delta(a_0b_0,c) = \chi(c \neq a_0b_0)$ for every color $c$. Thus if there is a part of size $p$ and colour $a_0b_0$ in $\lambda$, it can repeat still with colour $a_0b_0$, but the integer $p$ cannot appear in any other colour.

To build $\Phi(\lambda)=(\mu,\nu)\in\mathcal{C}_n(\delta,\gamma)\times \mathcal{P}^0$, we proceed as follows.

\medskip
\textbf{Step 1:} Remove all the parts coloured $a_0b_0$ from $\lambda$ and put them in a separate partition called $\nu_1$. Let us denote by $\mu_1$ what remains of the partitions $\lambda$ after doing this.

The partition $\mu_1$ is a partition from $\mathcal{P}_n$ having no part coloured $a_0b_0$, $\nu_1$ belongs to $\mathcal{P}^0$, and $\mu_1$ has no part of the same size as a part of $\nu_1$.

In our example, we obtain:
\begin{align*}
\mu_1 &= 8_{a_1b_1}+6_{a_0b_2}+6_{a_2b_2}+5_{a_0b_1}
+5_{a_1b_0}
+3_{a_0b_2}+3_{a_1b_1}+3_{a_1b_1}+3_{a_1b_0}+2_{a_2b_2}
+2_{a_2b_2}+2_{a_2b_2}+2_{a_2b_0},\\
\nu_1 &=4_{a_0b_0}+4_{a_0b_0}+1_{a_0b_0}.
\end{align*}

\textbf{Step 2:} If, in $\mu_1$, there is a part $p_{a_ib_i}$ with $i>0$ which repeats, then transform all but one of these $p_{a_ib_i}$'s into $p_{a_0b_0}$'s and insert them in the partition $\nu_1$, forming a partition $\nu_2$. Note that the new parts that were inserted were all different from the parts in $\nu_1$. We denote by $\mu_2$ the partition formed by the remaining parts of $\mu_1$ after this process.

The partition $\mu_2$ is a partition from $\mathcal{P}_n$ with no part coloured $a_0b_0$, such that free colours cannot repeat. The partition $\nu_2$ belongs to $\mathcal{P}^0$, and $\mu_2$ and $\nu_2$ can now have parts of the same size.

In our example, we obtain:
\begin{align*}
\mu_2 &= 8_{a_1b_1}+6_{a_0b_2}+6_{a_2b_2}+5_{a_0b_1}
+5_{a_1b_0}
+3_{a_0b_2}+3_{a_1b_1}+3_{a_1b_0}+2_{a_2b_2}
+2_{a_2b_0},\\
\nu_2 &= 4_{a_0b_0}+4_{a_0b_0}+3_{a_0b_0}+2_{a_0b_0}+2_{a_0b_0}+1_{a_0b_0}.
\end{align*}

\textbf{Step 3:} If in $\mu_2$, there is a pattern with central part $p_{a_ib_i}$ which is forbidden in $\mathcal{C}_n(\delta,\gamma)$, then transform this $p_{a_ib_i}$ into $p_{a_0b_0}$ and insert it in the partition $\nu_2$, forming a partition $\nu$. Note that some parts $p_{a_0b_0}$ may have already been moved to $\nu_2$ at Step $2$. Finally, denote by $\mu$ the partition formed by the remaining parts of $\mu_2$ after this process.

Now the partition $\mu$ is a partition from $\mathcal{P}_n$ with no part coloured $a_0b_0$, such that free colours cannot repeat, and avoiding all forbidden patterns. This is equivalent to saying that $\mu$ belongs to $\mathcal{C}_n(\delta,\gamma).$ Moreover, $\nu$ is simply an unrestricted partition coloured $a_0b_0$.

In our example, $\mu_2$ contains the forbidden patterns $3_{a_1b_0}+2_{a_2b_2}+2_{a_2b_0}$, $3_{a_0b_2}+3_{a_1b_1}+3_{a_1b_0}$, and $6_{a_0b_2}+6_{a_2b_2}+5_{a_0b_1}$, so we obtain:
\begin{align*}
\mu &= 8_{a_1b_1}+6_{a_0b_2}+5_{a_0b_1}
+5_{a_1b_0}+3_{a_0b_2}+3_{a_1b_0}+2_{a_2b_0},\\
\nu &=6_{a_0b_0}+4_{a_0b_0}+4_{a_0b_0}
+3_{a_0b_0}+3_{a_0b_0}+2_{a_0b_0}+2_{a_0b_0}+2_{a_0b_0}+1_{a_0b_0}.
\end{align*}

\subsubsection{The inverse bijection $\Phi^{-1}$ from  $\mathcal{C}_n(\delta,\gamma)\times \mathcal{P}^0$ to $\mathcal{P}_n$}
We now build the inverse bijection $\Phi^{-1}$. Consider a partition pair $(\mu,\nu) \in \mathcal{C}_n(\delta,\gamma)\times\mathcal{P}^0$. To obtain a partition $\lambda \in \mathcal{P}_n$, we insert each part $p_{a_0b_0}$ of $\nu$ in the partition $\mu$ as follows.

We illustrate $\Phi^{-1}$ on the same example as before, starting from
\begin{align*}
\mu &= 8_{a_1b_1}+6_{a_0b_2}+5_{a_0b_1}
+5_{a_1b_0}+3_{a_0b_2}+3_{a_1b_0}+2_{a_2b_0},\\
\nu &=6_{a_0b_0}+4_{a_0b_0}+4_{a_0b_0}
+3_{a_0b_0}+3_{a_0b_0}+2_{a_0b_0}+2_{a_0b_0}+2_{a_0b_0}+1_{a_0b_0}.
\end{align*}

\textbf{Inverse of Step 3:}
 If there are parts $p_{a_0b_0}$ in $\nu$, and if in $\mu$ there is a part $p_c$ for some bound colour $c=a_kb_\ell$ ($k \neq \ell$) but $p$ does not appear in any free colour, we proceed as follows. Note that the order in which we perform the insertions does not matter. Indeed, the only case in which two insertions interact with each other is the case of Proposition \ref{prop:prop0.1}, and then whether we insert $(p+1)_{a_ib_i}$ or $p_{a_jb_j}$ first does not change anything to the final result.
\begin{itemize}
\item If there exists a pair of bound colors $(a_{k_1}b_{\ell_1},a_{k_2}b_{\ell_2})$ such that $k_1<\ell_1$ and $k_2>\ell_2$, and the pattern $p_{a_{k_1}b_{\ell_1}}+p_{a_{k_2}b_{\ell_2}}$ appears in $\mu$, then by Proposition \ref{prop:seq0000}, we have
$$\max\{k_1,\ell_2\}<\min\{k_2,\ell_1\}.$$
In this case, transform $\nu$'s first $p_{a_0b_0}$ into $p_{a_ib_i}$ with $i=\gamma(a_{k_1}b_{\ell_1},a_{k_2}b_{\ell_2})$, and insert it between $p_{a_{k_1}b_{\ell_1}}$ and $p_{a_{k_2}b_{\ell_2}}$ in $\mu$, creating the pattern 
$$p_{a_{k_1}b_{\ell_1}}+p_{a_ib_i}+p_{a_{k_2}b_{\ell_2}}$$
which is forbidden in $\mathcal{C}_n(\delta,\gamma)$, but is allowed in $\mathcal{P}_n$ by Proposition \ref{prop:insertion0000}.

In our example, the pattern $3_{a_0b_2}+3_{a_1b_0}$ appears in $\mu.$ Thus, we transform $\nu$'s first $3_{a_0b_0}$ into $3_{a_1b_1}$ because $\gamma(a_{0}b_{2},a_{1}b_{0})=1$, and insert it into $\mu$. We obtain
\begin{align*}
\mu' &= 8_{a_1b_1}+6_{a_0b_2}+5_{a_0b_1}
+5_{a_1b_0}+3_{a_0b_2}+3_{a_1b_1}+3_{a_1b_0}+2_{a_2b_0},\\
\nu' &=6_{a_0b_0}+4_{a_0b_0}+4_{a_0b_0}
+3_{a_0b_0}+2_{a_0b_0}+2_{a_0b_0}+2_{a_0b_0}+1_{a_0b_0}.
\end{align*}

\item If all the parts of size $p$ in $\mu$ have colours $a_kb_\ell$ with $k>\ell$, let us consider the first of these parts, denoted by $p_{a_{k_2}b_{\ell_2}}$.
This is the case for parts of size $2$ in our example.

By Case 2a of Proposition \ref{prop:seq0000}, we have $k>k_2>\ell_2>\ell$ for all the other colours $a_kb_\ell$ in which $p$ appears. We distinguish several cases for our insertions.

\begin{enumerate}
\item If the part to the left of $p_{a_{k_2}b_{\ell_2}}$ has size $p+u$ with $u \geq 2$ (with the convention that $u=\infty$ when $p_{a_{k_2}b_{\ell_2}}$ is the first part of the partition), transform $\nu$'s first $p_{a_0b_0}$ into $p_{a_ib_i}$ with $i=\delta(a_{k_2}b_{\ell_2})$, and  insert it to the left of $p_{a_{k_2}b_{\ell_2}}$ in $\mu$, creating the pattern
$$(p+u)_{a_{k_1}b_{\ell_1}}+p_{a_ib_i}+p_{a_{k_2}b_{\ell_2}}$$
for some $k_1,\ell_1$. This pattern is forbidden in $\mathcal{C}_n(\delta,\gamma)$, but allowed in $\mathcal{P}_n$ by the definition of $\Delta$.

This case does not occur in our example.

\item If the part to the left of $p_{a_{k_2}b_{\ell_2}}$ is equal to $(p+1)_{a_{k_1}b_{\ell_1}}$ with $k_1\leq \ell_1$, transform $\nu$'s first $p_{a_0b_0}$ into $p_{a_ib_i}$ with $i=\delta(a_{k_2}b_{\ell_2})$, and insert it to the left of $p_{a_{k_2}b_{\ell_2}}$ in $\mu$, creating the pattern
$$(p+1)_{a_{k_1}b_{\ell_1}}+p_{a_ib_i}+p_{a_{k_2}b_{\ell_2}},$$
which is forbidden in $\mathcal{C}_n(\delta,\gamma)$, but allowed in $\mathcal{P}_n$ by Proposition \ref{prop:prop0.1}.

This case does not occur in our example either.

\item If the part to the left of $p_{a_{k_2}b_{\ell_2}}$ is equal to $(p+1)_{a_{k_1}b_{\ell_1}}$ with $k_1> \ell_1$, 
we necessarily have that $k_1<k_2$ or $\ell_1>\ell_2$, i.e. $\{\ell_2+1,\ldots,k_2\}\setminus \{\ell_1+1,\ldots,k_1\}\neq \emptyset$. In that case, transform $\nu$'s first $p_{a_0b_0}$  into $p_{a_ib_i}$ with $i=\gamma(a_{k_1}b_{\ell_1},a_{k_2}b_{\ell_2})$, and insert it to the left of $p_{a_{k_2}b_{\ell_2}}$ in $\mu$, creating the pattern
$$(p+1)_{a_{k_1}b_{\ell_1}}+p_{a_ib_i}+p_{a_{k_2}b_{\ell_2}},$$
which is forbidden in $\mathcal{C}_n(\delta,\gamma)$, but allowed in $\mathcal{P}_n$ by Proposition \ref{prop:prop0.2}.

In our example, assuming we already did the insertion leading to $\mu'$ and $\nu'$ above, we transform the first $2_{a_0b_0}$ of $\nu'$ into $2_{a_2b_2}$ because $\gamma_1(a_{1}b_{0},a_{2}b_{0})=2$, and insert it to the left of $2_{a_2b_0}$ in $\mu'$. We obtain the partitions
\begin{align*}
\mu'' &= 8_{a_1b_1}+6_{a_0b_2}+5_{a_0b_1}
+5_{a_1b_0}+3_{a_0b_2}+3_{a_1b_1}+3_{a_1b_0}+2_{a_2b_2}+2_{a_2b_0},\\
\nu'' &=6_{a_0b_0}+4_{a_0b_0}+4_{a_0b_0}
+3_{a_0b_0}+2_{a_0b_0}+2_{a_0b_0}+1_{a_0b_0}.
\end{align*}
\end{enumerate}

\item If all the parts of size $p$ in $\mu$ have colours $a_kb_\ell$ with $k<\ell$, let us consider the last of these parts, denoted by $p_{a_{k_1}b_{\ell_1}}$.
This is the case for parts of size $6$ in our example.

By Case 2b of Proposition \ref{prop:seq0000}, we have $k<k_1<\ell_1<\ell$ for all the other colours $a_kb_\ell$ in which $p$ appears. As before, we distinguish several cases for our insertions.

\begin{enumerate}
\item If the part to the right of $p_{a_{k_1}b_{\ell_1}}$ has size at $p-u$ with $u \geq 2$ (with the convention that $u=\infty$ when $p_{a_{k_1}b_{\ell_1}}$ is the last part of the partition), transform $\nu$'s first $p_{a_0b_0}$ into $p_{a_ib_i}$ with $i=\delta(a_{k_1}b_{\ell_1})$, and  insert it to the right of $p_{a_{k_1}b_{\ell_1}}$ in $\mu$, creating the pattern
$$p_{a_{k_1}b_{\ell_1}}+p_{a_ib_i}+(p-u)_{a_{k_2}b_{\ell_2}}$$
for some $k_2,\ell_2$. This pattern is forbidden in $\mathcal{C}_n(\delta,\gamma)$, but allowed in $\mathcal{P}_n$ by the definition of $\Delta$.

This case does not occur in our example.

\item If the part to the right of $p_{a_{k_1}b_{\ell_1}}$ is equal to $(p-1)_{a_{k_2}b_{\ell_2}}$ with $k_2\geq \ell_2$, transform $\nu$'s first $p_{a_0b_0}$  into $p_{a_ib_i}$ with $i=\delta(a_{k_1}b_{\ell_1})$, and insert it to the right of $p_{a_{k_1}b_{\ell_1}}$ in $\mu$, creating the pattern
$$p_{a_{k_1}b_{\ell_1}}+p_{a_ib_i}+(p-1)_{a_{k_2}b_{\ell_2}}$$
which is forbidden in $\mathcal{C}_n(\delta,\gamma)$, but allowed in $\mathcal{P}_n$ by Proposition \ref{prop:prop0.1}.

This case does not occur in our example either.

\item If the part to the right of $p_{a_{k_1}b_{\ell_1}}$ is equal to $(p-1)_{a_{k_2}b_{\ell_2}}$ with $k_2 < \ell_2$, we necessarily have $k_1<k_2$ or $\ell_1>\ell_2$, i.e. $\{k_1+1,\ldots,\ell_1\}\setminus\{k_2+1,\ldots,\ell_2\}\neq \emptyset$.
In that case, transform $\nu$'s first $p_{a_0b_0}$  into $p_{a_ib_i}$ with $i=\gamma(a_{k_1}b_{\ell_1},a_{k_2}b_{\ell_2})$, and insert it to the right of $p_{a_{k_1}b_{\ell_1}}$ in $\mu$, creating the pattern
$$(p)_{a_{k_1}b_{\ell_1}}+p_{a_ib_i}+(p-1)_{a_{k_2}b_{\ell_2}},$$
which is forbidden in $\mathcal{C}_n(\delta,\gamma)$, but allowed in $\mathcal{P}_n$ by Proposition \ref{prop:prop0.3}.

In our example, assuming we already did the insertions leading to $\mu''$ and $\nu''$ above, we transform the $6_{a_0b_0}$ of $\nu''$ into $6_{a_2b_2}$ because $\gamma_1(a_{0}b_{2},a_{0}b_{1})=2$, and insert it to the right of $6_{a_0b_2}$ in $\mu''$. We obtain the partitions
\begin{align*}
\mu_2 &= 8_{a_1b_1}+6_{a_0b_2}+6_{a_2b_2}+5_{a_0b_1}
+5_{a_1b_0}+3_{a_0b_2}+3_{a_1b_1}+3_{a_1b_0}+2_{a_2b_2}+2_{a_2b_0},\\
\nu_2 &=4_{a_0b_0}+4_{a_0b_0}
+3_{a_0b_0}+2_{a_0b_0}+2_{a_0b_0}+1_{a_0b_0}.
\end{align*}

\end{enumerate}
\end{itemize}
Denote $\mu_2$ and $\nu_2$ the resulting partitions when all the insertions described in this step have been performed.
The partition $\mu_2$ is a partition from $\mathcal{P}_n$ with no part coloured $a_0b_0$, such that free colours cannot repeat. It can now contain all the patterns that were forbidden in $\mu.$ The partition $\nu_2$ belongs to $\mathcal{P}^0$.

\medskip
\textbf{Inverse of Step 2:} 
If there are some parts $p_{a_0b_0}$ in $\nu_2$ that also appear in $\mu_2$ in a free colour $a_ib_i$ ($i>0$), then transform them all into $p_{a_ib_i}$ and insert them in $\mu_2$. Let us call $\mu_1$ and $\nu_1$ the resulting partitions. 
The partition $\mu_1$ is a partition from $\mathcal{P}_n$ having no part coloured $a_0b_0$. The partition $\nu_1$ belongs to $\mathcal{P}^0$, and $\mu_1$ has no part of the same size as a part of $\nu_1$.

In our example, we obtain again
\begin{align*}
\mu_1 &= 8_{a_1b_1}+6_{a_0b_2}+6_{a_2b_2}+5_{a_0b_1}
+5_{a_1b_0}+3_{a_0b_2}+3_{a_1b_1}+3_{a_1b_1}+3_{a_1b_0}
+2_{a_2b_2}+2_{a_2b_2}+2_{a_2b_2}+2_{a_2b_0},\\
\nu_1 &=4_{a_0b_0}+4_{a_0b_0}+1_{a_0b_0}.
\end{align*}

\medskip
\textbf{Inverse of Step 1:} 
The remaining parts in $\nu_1$ do not appear in $\mu_1$, so we simply insert them, creating a partition $\lambda$ which belongs to $\mathcal P_n$.

In our example, we recover
$$\lambda = 8_{a_1b_1}+6_{a_0b_2}+6_{a_2b_2}+5_{a_0b_1}
+5_{a_1b_0}+4_{a_0b_0}+4_{a_0b_0}
+3_{a_0b_2}+3_{a_1b_1}+3_{a_1b_1}+3_{a_1b_0}+2_{a_2b_2}
+2_{a_2b_2}+2_{a_2b_2}+2_{a_2b_0}+1_{a_0b_0},$$
as desired.

\medskip
All the steps in this bijection only consist of colour modifications on free colours and moving some parts, so the bijection preserves the weight, the number of parts, the sizes of the parts, and the number of appearances of each bound colour.

\section{Proof of Proposition \ref{prop:main'}} 
\label{sec:prop1}
In this last section, we give a proof of Proposition \ref{prop:main'}. Let $S= c_1, \dots , c_s$ be a reduced colour sequence of length $s$, having $t$ maximal primary subsequences. We use the same notation as in Section \ref{subsec:gfkernelDelta}.
In addition, we define for all $u \in \{1, \dots, t\}$, $j_{2u-1}$ (resp. $j_{2u}$) to be the index of the free colour which can be inserted to the left (resp. right) of $S_u$. Thus we have $\mathcal{T}_0^u = \{j_{2u-1},j_{2u}\} \cap \mathcal{T}_0$ and $\mathcal{T}_1^u = \{j_{2u-1},j_{2u}\} \cap \mathcal{T}_1.$

For brevity, from now on, we denote the set of all integers between $i$ and $j$ by $\llbracket i;j \rrbracket .$

Our starting point is the equality
\begin{equation}
\label{eq:6starting}
G_{S,m}(q):=\sum_{\substack{C \text{colour sequence of length } s+m\\ \text{such that } \mathrm{red}(C)=S}} q^{|\min_{\Delta}(C)|} = \sum_{\substack{n_1, \dots, n_{s+t}:\\n_1+\cdots+n_{s+t}=m}} q^{|\mathrm{min}_{\Delta}(S(n_1, \dots, n_{s+t}))|},
\end{equation}
which simply follows from the definition of reduced colour sequences.

Proposition \ref{prop:minpartition} gives us an expression for $|\mathrm{min}_{\Delta}(S(n_1, \dots, n_{s+t}))|$, which we will use to derive  Proposition \ref{prop:main'}.
Let us start with a lemma which evaluates a sum appearing in the formula for $|\mathrm{min}_{\Delta}(S(n_1, \dots, n_{s+t}))|.$ 

\begin{lemma}
\label{lem:6S1}
Let 
$$\Sigma_1:=\sum_{j \in \mathcal{S}_1} \left( P(j)+ \#\left(\llbracket j;s+t\rrbracket \cap (\mathcal{N}\sqcup\mathcal{T}_0 \sqcup \mathcal{S}_1) \right) \right),$$
where $P(j)$ is the number of colours of $S$ that are to the left of $f_j$.
We have
$$\Sigma_1 = \sum_{u=1}^t \left( \left( |\mathcal{N}|+u-1  +\sum_{v=u}^t \left(|\mathcal{T}_0^{v}| +|\mathcal{S}_1^v|\right) \right) |\mathcal{S}_1^u| + \sum_{j\in \mathcal{S}_1^u} \# \{j'<j: j' \in \overline{\mathcal{S}_1^u} \}\right),$$
where $\overline{\mathcal{S}_1^u} := \mathcal{T}_1^u \setminus \mathcal{S}_1^u$ is the set of indices $j$ of $\mathcal{T}_1^u$ such that the free colour $f_j$ is not inserted.
\end{lemma}
\begin{proof}
First, writing $\mathcal{S}_1 = \bigsqcup\limits_{u=1}^{t} \mathcal{S}_1^u,$ we have
$$\Sigma_1= \sum_{u=1}^t \sum_{j \in \mathcal{S}_1^u} \left( P(j)+ \#\left(\llbracket j;s+t\rrbracket \cap (\mathcal{N}\sqcup\mathcal{T}_0 \sqcup \mathcal{S}_1) \right) \right).$$
Now, noticing that for $j \in \mathcal{S}_1^u$, $P(j)=j-u$, we can write
\begin{equation}
\label{eq:Sigma1}
\Sigma_1= \sum_{u=1}^t \sum_{j \in \mathcal{S}_1^u} \left(j_{2u-1}-u+j-j_{2u-1}+ \#\left(\llbracket j;s+t\rrbracket \cap (\mathcal{N}\sqcup\mathcal{T}_0 \sqcup \mathcal{S}_1) \right) \right).
\end{equation}
We first note that
$$
\begin{array}{lll}
&j_{2u-1}-u = 1-u +j_{2u-1} -1& \\
&\quad = 1-u + \#(\llbracket 1;j_{2u-1} -1\rrbracket \cap \mathcal{N}) + \#(\llbracket 1;j_{2u-1} -1\rrbracket \cap (\mathcal{T}_0 \sqcup \mathcal{T}_1)) & \text{ because }  \llbracket 1;s+t\rrbracket = \mathcal{N} \sqcup \mathcal{T}_0 \sqcup \mathcal{T}_1 \\
&\quad = 1-u + \#(\llbracket 1;j_{2u-1} -1\rrbracket \cap \mathcal{N}) + 2u-2 & \text{ by definition of } j_{2u-1}\\
&\quad =\#(\llbracket 1;j_{2u-1} -1\rrbracket \cap \mathcal{N}) + u-1.
\end{array}
$$
We also rewrite $j-j_{2u-1}$ as 
$$j-j_{2u-1} = \#(\llbracket j_{2u-1};j-1\rrbracket \cap \mathcal{T}_0^u)+ \#(\llbracket j_{2u-1};j-1\rrbracket \cap \mathcal{S}_1^u)+ \#(\llbracket j_{2u-1};j-1\rrbracket \cap \overline{\mathcal{S}_1^u})+ \#(\llbracket j_{2u-1};j-1\rrbracket \cap \mathcal{N}).$$ 
Finally, we have
\begin{align*}
 \#(\llbracket j;s+t\rrbracket \cap (\mathcal{N} \sqcup \mathcal{T}_0 \sqcup \mathcal{S}_1)) &= \#(\llbracket j;s+t\rrbracket \cap \mathcal{N}) + \#(\llbracket j;j_{2u}\rrbracket \cap (\mathcal{T}_0^u \sqcup \mathcal{S}_1^u))+\#(\llbracket j_{2u}+1;s+t\rrbracket \cap (\mathcal{T}_0 \sqcup \mathcal{S}_1))\\
 &=\#(\llbracket j;s+t\rrbracket \cap \mathcal{N}) + \#(\llbracket j;j_{2u}\rrbracket \cap (\mathcal{T}_0^u \sqcup \mathcal{S}_1^u)) + \sum_{v=u+1}^t \left( |\mathcal{T}_0^v|+| \mathcal{S}_1^v|\right).
\end{align*}
Combining the three observations above, \eqref{eq:Sigma1} becomes
$$
\Sigma_1 = \sum_{u=1}^t \sum_{j \in \mathcal{S}_1^u} \left( |\mathcal{N}|+u-1+\sum_{v=u}^t \left( |\mathcal{T}_0^v|+| \mathcal{S}_1^v|\right) +\#(\llbracket j_{2u-1};j-1\rrbracket \cap \overline{\mathcal{S}_1^u}) \right).
$$
Noticing that $|\mathcal{N}|+u-1+\sum_{v=u}^t \left( |\mathcal{T}_0^v|+| \mathcal{S}_1^v|\right)$ does not depend on $j$, and that $\#(\llbracket j_{2u-1};j-1\rrbracket \cap \overline{\mathcal{S}_1^u})=\# \{j'<j: j' \in \overline{\mathcal{S}_1^u} \}$ yields the desired formula.
\end{proof}

We can now give a formula for the generating function for minimal partitions $\mathrm{min}_{\Delta}(S(n_1, \dots, n_{s+t}))$ for a fixed set $\mathcal{S}_1$. The desired generating function $G_{S,m}(q)$ of \eqref{eq:6starting} will then be obtained by summing over all possible sets $\mathcal{S}_1$.

\begin{lemma}
\label{lem:6HS1}
Let $\mathcal{S}_1$ be fixed. Define
$$H_{S,\mathcal{S}_1}(q):= \sum_{\substack{n_1, \dots, n_{s+t}:\\n_1+\cdots+n_{s+t}=m,\\ \{j \in \mathcal{T}_1 :n_j>0\}=\mathcal{S}_1}} q^{|\mathrm{min}_{\Delta}(S(n_1, \dots, n_{s+t}))|}.$$
We have
\begin{equation}
\label{eq:6HS1}
H_{S,\mathcal{S}_1}(q)= q^{|\mathrm{min}_{\Delta}(S)| +\Sigma_1 +m - |\mathcal{S}_1|} {m-1+|\mathcal{N}| + |\mathcal{T}_0| \brack m-|\mathcal{S}_1|}_q.
\end{equation}

\end{lemma}
\begin{proof}
By Proposition \ref{prop:minpartition} and Lemma \ref{lem:6S1}, we have
$$
H_{S,\mathcal{S}_1}(q)= \sum_{\substack{n_1, \dots, n_{s+t}:\\n_1+\cdots+n_{s+t}=m,\\ \{j \in \mathcal{T}_1 :n_j>0\}=\mathcal{S}_1}} q^{|\mathrm{min}_{\Delta}(S)| +\Sigma_1+ \sum_{j \in \mathcal{S}_1} (n_j-1) \#\left(\llbracket j;s+t\rrbracket \cap (\mathcal{N}\sqcup\mathcal{T}_0 \sqcup \mathcal{S}_1) \right)  + \sum_{j \in \mathcal{N}\cup\mathcal{T}_0}  n_j\#\left(\llbracket j;s+t\rrbracket \cap (\mathcal{N}\sqcup\mathcal{T}_0 \sqcup \mathcal{S}_1) \right)}.
$$
Thus by the changes of variables
$$
n'_j=
\begin{cases}
n_j \text{ if } j \in \mathcal{N}\sqcup\mathcal{T}_0\\
n_j-1 \text{ if } j \in \mathcal{S}_1
\end{cases},
$$
and noticing that $|\mathrm{min}_{\Delta}(S)|$ and $\Sigma_1$ do not depend on the $n_j$'s,
we obtain
\begin{equation}
\label{eq:6HS1'}
H_{S,\mathcal{S}_1}(q)=q^{|\mathrm{min}_{\Delta}(S)| +\Sigma_1}\sum_{\substack{(n'_j)_{j \in \mathcal{N}\sqcup\mathcal{T}_0 \sqcup \mathcal{S}_1}:\\ \sum_j n'_j=m-|\mathcal{S}_1|}} q^{ \sum_{j \in \mathcal{N}\sqcup\mathcal{T}_0 \sqcup \mathcal{S}_1}  n'_j\#\left(\llbracket j;s+t\rrbracket \cap (\mathcal{N}\sqcup\mathcal{T}_0 \sqcup \mathcal{S}_1) \right)}.
\end{equation}

Moreover, we can interpret the sum above as the generating function  for partitions into exactly $m - |\mathcal{S}_1|$ parts, each part being at most $|\mathcal{N}|+|\mathcal{T}_0|+|\mathcal{S}_1|.$ Indeed, for all  $j \in \mathcal{N}\sqcup\mathcal{T}_0 \sqcup \mathcal{S}_1,$ $n'_j$ can be interpreted as the number of parts of size $\#\left(\llbracket j;s+t\rrbracket \cap (\mathcal{N}\sqcup\mathcal{T}_0 \sqcup \mathcal{S}_1) \right)$ (see Figure \ref{fig:42} below).

\begin{figure}[H]
\includegraphics[width=0.6\textwidth]{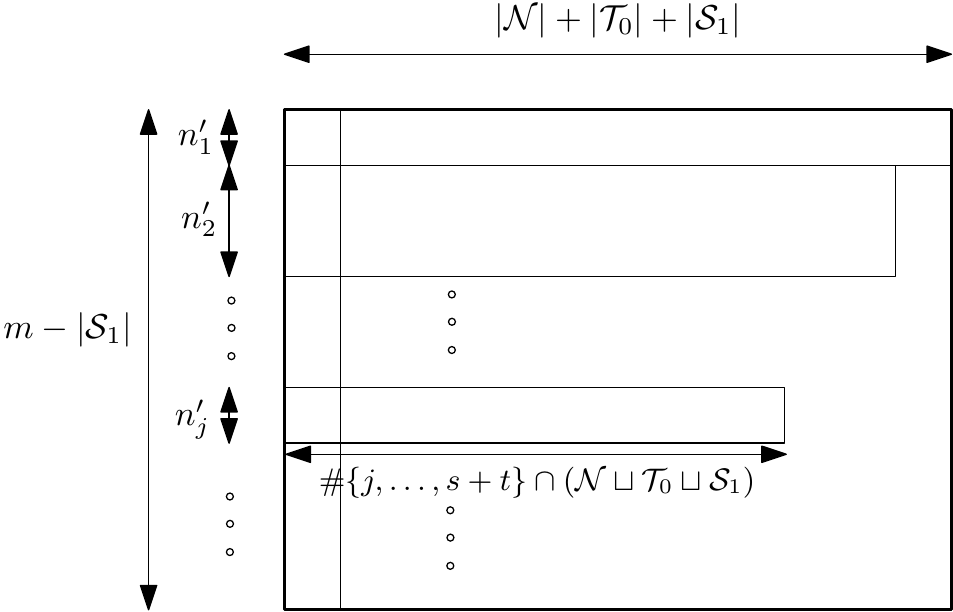}
\caption{Decomposition of the Ferrers board.}
\label{fig:42}
\end{figure}
The generating function for such partitions is given by $q^{m - |\mathcal{S}_1|} {m-1+|\mathcal{N}| + |\mathcal{T}_0| \brack m-|\mathcal{S}_1|}_q$, which yields the desired formula \eqref{eq:6HS1} for $H_{S,\mathcal{S}_1}(q).$
\end{proof}

Before we compute $G_{S,m}(q)$, we still need one more lemma about $q$-binomial coefficients.
\begin{lemma}
\label{lem:6qbinpath}
Let $a$ and $b$ be non-negative integers. We have 
$$\sum_{\substack{ A \subseteq \llbracket 1;a+b \rrbracket \\ |A|=a}} q^{\sum_{j \in A} \#\{j'<j:j'\in \llbracket 1;a+b \rrbracket \setminus A\}} = {a+b \brack a}_q.$$
\end{lemma}
\begin{proof}
Partitions whose Ferrers diagram fits inside an $a \times b$ box, generated by ${a+b \brack a}_q$, are in bijection with walks on the plane going from $(0,0)$ to $(b,a)$, having $b$ right steps and $a$ up steps. The partition can be seen on top of the path, as shown in Figure \ref{fig:41}.
\begin{figure}[H]
\includegraphics[width=0.5\textwidth]{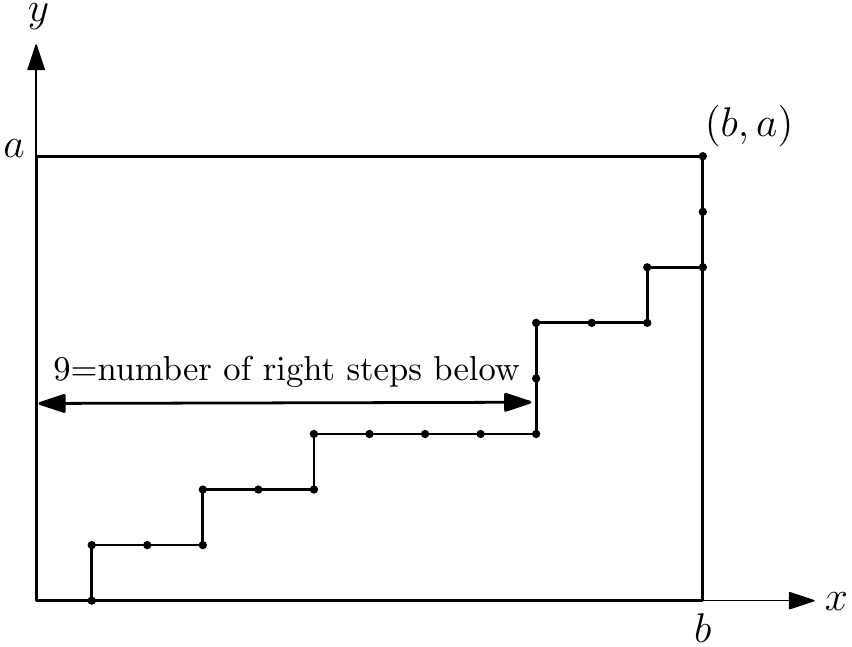}
\caption{A partition as a path.}
\label{fig:41}
\end{figure}

If $A \subseteq \llbracket 1;a+b \rrbracket ,|A|=a$ is the set of up steps, then for each position $j \in A$, the part of the partition corresponding to this up step has its size equal to the number of right steps that have been done before, i.e. $\#\{j'<j:j'\in \llbracket 1;a+b \rrbracket \setminus A\}$.
\end{proof}

We are now ready to sum $H_{S,\mathcal{S}_1}(q)$ over all possible sets $\mathcal{S}_1$ to obtain a formula for $G_{S,m}(q).$

\begin{proposition}
\label{prop:6GS}
Let $S$ be a reduced colour sequence, and $m$ a non-negative integer.
We have
\begin{align*}
G_{S,m}(q)&= \sum_{\substack{k_1, \dots, k_t: \\ k_u \leq |\mathcal{T}_1^u|}} q^{|\mathrm{min}_{\Delta}(S)| + \sum_{u=1}^t k_u(|\mathcal{N}|+u-1+\sum_{v=u}^t (|\mathcal{T}_0^v|+k_v))} q^{m- \sum_{u=1}^t k_u} {m-1+|\mathcal{N}| + |\mathcal{T}_0| \brack m-\sum_{u=1}^t k_u}_q \prod_{u=1}^t {|\mathcal{T}_1^u| \brack k_u}_q.
\end{align*}
\end{proposition}
\begin{proof}
By Lemma \ref{lem:6HS1}, we have:
$$G_{S,m}(q)= \sum_{\substack{k_1, \dots, k_t: \\ k_u \leq |\mathcal{T}_1^u|}} \sum_{\substack{\mathcal{S}_1:\\ \forall u, \mathcal{S}_1^u \subseteq \mathcal{T}_1^u \\ \text{and } |\mathcal{S}_1^u|= k_u}} H_{S,\mathcal{S}_1}(q) = \sum_{\substack{k_1, \dots, k_t: \\ k_u \leq |\mathcal{T}_1^u|}} \sum_{\substack{\mathcal{S}_1:\\ \forall u, \mathcal{S}_1^u \subseteq \mathcal{T}_1^u \\ \text{and } |\mathcal{S}_1^u|= k_u}} q^{|\mathrm{min}_{\Delta}(S)| +\Sigma_1 +m - |\mathcal{S}_1|} {m-1+|\mathcal{N}| + |\mathcal{T}_0| \brack m-|\mathcal{S}_1|}_q.
$$ 
By Lemma \ref{lem:6S1}, this becomes
\begin{align*}
G_{S,m}(q)&= \sum_{\substack{k_1, \dots, k_t: \\ k_u \leq |\mathcal{T}_1^u|}} q^{|\mathrm{min}_{\Delta}(S)| + \sum_{u=1}^t k_u(|\mathcal{N}|+u-1+\sum_{v=u}^t (|\mathcal{T}_0^v|+k_v))} q^{m- \sum_{u=1}^t k_u} {m-1+|\mathcal{N}| + |\mathcal{T}_0| \brack m-\sum_{u=1}^t k_u}_q
\\& \times \sum_{\substack{\mathcal{S}_1:\\ \forall u, \mathcal{S}_1^u \subseteq \mathcal{T}_1^u \\ \text{and } |\mathcal{S}_1^u|= k_u}} \prod_{u=1}^t q^{ \sum_{j\in \mathcal{S}_1^u} \# \{j'<j: j' \in \overline{\mathcal{S}_1^u} \}}.
\end{align*}
Exchanging the final sum and product, and then using Lemma \ref{lem:6qbinpath} for each $u \in \{1, \dots, t\}$ with $a=k_u$  and $b= |\mathcal{T}_1^u|-k_u$ gives the desired formula.
\end{proof}

What remains to do is show that the expression for $G_{S,m}(q)$ in Proposition \ref{prop:6GS} is actually the same as  \eqref{eq:prop1'}.

First, let us give yet another lemma about $q$-binomial coefficients.
\begin{lemma}
\label{lem:6lemmaY}
Let $m, \ell_1, \dots, \ell_t$ be non-negative integers. We have
$$ q^m {m+\ell_1 + \cdots + \ell_t-1 \brack m}_q= q^m \sum_{0=x_0 \leq x_1 \leq \cdots \leq x_t=m} \prod_{r=1}^t q^{\ell_r x_{r-1}} {x_r-x_{r-1}+\ell_r-1 \brack x_r-x_{r-1}}_q.$$
In the above, we use the convention that ${-1 \brack 0}=1$, corresponding to the case where a certain $\ell_r$ is equal to $0$.
\end{lemma}
\begin{proof}
The left-hand side is the generating function for partitions fitting inside a $m \times (\ell_1 + \cdots + \ell_t)$ box, such that the largest part is equal to $m$. Take the Ferrers board of such a partition, and draw it in the plane as shown on Figure \ref{fig:44} (where the partition is above the path).
\begin{figure}[H]
\includegraphics[width=0.6\textwidth]{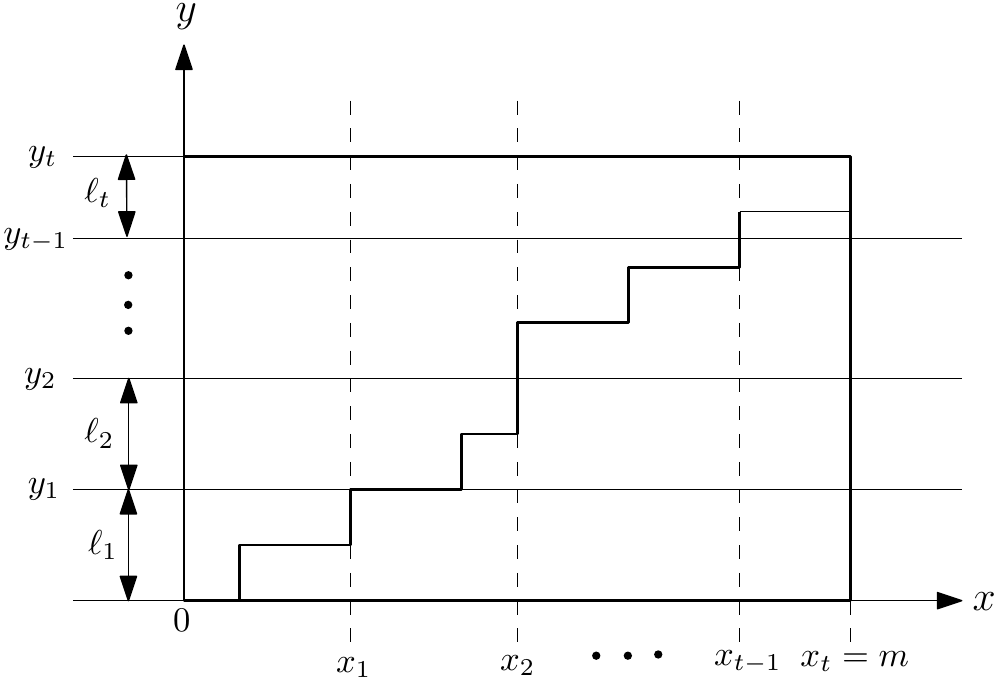}
\caption{Decomposition of the Ferrers board.}
\label{fig:44}
\end{figure}

For all $i \in \{1, \dots, t\},$ let $x_i$ be the size of the $\left(\sum_{k=i+1}^t \ell_k+1 \right)$-th part  (with $x_i=0$ if there are less than $\ell_1 + \cdots + \ell_t - y_i+1$ parts).

For all $i \in \{1, \dots, t\},$ let $y_i:= \sum_{k=1}^i \ell_k$. For fixed $0\leq x_1 \leq \cdots \leq x_t=m$, these partitions are generated by
$$\prod_{r=1}^t q^{\ell_r x_{r-1}} \times q^{x_r-x_{r-1}} {x_r-x_{r-1}+\ell_r-1 \brack x_r-x_{r-1}}_q,$$
where $q^{\ell_r x_{r-1}}$ generates the rectangle between the $y$-axis, the lines $y=y_r$ and $y=y_{r-1},$ and the line $x=x_{r-1}$, and the second term generates partitions fitting inside a $(x_r-x_{r-1}) \times \ell_r$ box, such that the largest part is equal to $x_r-x_{r-1}$.

The above is equal to
$$q^m \prod_{r=1}^t q^{\ell_r x_{r-1}} {x_r-x_{r-1}+\ell_r-1 \brack x_r-x_{r-1}}_q,$$
and summing over all possible values for $x_1, \dots, x_{t-1}$ gives the desired result.
\end{proof}

We use the lemma above to rewrite a part of the expression in Proposition \ref{prop:6GS}.
\begin{lemma}
\label{lem:6grosqbin}
We have:
\begin{align*}
&q^{m- \sum_{u=1}^t k_u} {m-1+|\mathcal{N}| + |\mathcal{T}_0| \brack m-\sum_{u=1}^t k_u}_q = q^{m- \sum_{u=1}^t k_u(1 + |\mathcal{N}|+\sum_{v=u+1}^t (k_v+|\mathcal{T}_0^v|))} 
\\&\times\sum_{0=m_0 \leq m_1 \leq \cdots \leq m_t\leq m} \left(\prod_{u=1}^t q^{(k_u +|\mathcal{T}_0^u|)m_{u-1}} {m_u-m_{u-1}+|\mathcal{T}_0^u|-1 \brack m_u-m_{u-1}-k_u}_q \right)q^{|\mathcal{N}|m_t} {m-m_t+|\mathcal{N}|-1 \brack m-m_t}_q.
\end{align*}
\end{lemma}
\begin{proof}
Let us start  by applying Lemma \ref{lem:6lemmaY} with $t=t+1$, $m= m- \sum_{u=1}^t k_u$, $\ell_u = k_u+|\mathcal{T}_0^u|$ for all $u \in \{1, \dots , t\},$ and $\ell_{t+1}= |\mathcal{N}|.$
We have
\begin{align*}
X &:= q^{m- \sum_{u=1}^t k_u} {m+|\mathcal{T}_0|+|\mathcal{N}|-1 \brack m- \sum_{u=1}^t k_u}_q
\\&= q^{m- \sum_{u=1}^t k_u} \sum_{0=x_0 \leq x_1 \leq \cdots \leq x_{t+1}=m- \sum_{u=1}^t k_u} \left(\prod_{u=1}^t q^{(k_u + |\mathcal{T}_0^u|) x_{u-1}} {x_u-x_{u-1}+k_u+|\mathcal{T}_0^u|-1 \brack x_u-x_{u-1}}_q\right) 
\\ &\qquad \qquad  \qquad \times q^{|\mathcal{N}|x_t} {m- \sum_{u=1}^t k_u-x_{t}+|\mathcal{N}|-1 \brack m- \sum_{u=1}^t k_u-x_{t}}_q.
\end{align*}
By the changes of variables $x_u = m_u - \sum_{v=1}^u k_v$, we obtain
\begin{align*}
X &= q^{m- \sum_{u=1}^t k_u} \sum_{0=m_0 \leq m_1 \leq \cdots \leq m_{t+1}=m} \left(\prod_{u=1}^t q^{(k_u + |\mathcal{T}_0^u|) (m_{u-1} - \sum_{v=1}^{u-1} k_v)} {m_u-m_{u-1}+|\mathcal{T}_0^u|-1 \brack m_u-m_{u-1}-k_u}_q\right) 
\\ &\qquad \qquad  \qquad \times q^{|\mathcal{N}|(m_t - \sum_{v=1}^t k_v)} {m- m_t+|\mathcal{N}|-1 \brack m- m_t}_q
\\&=q^{m- \sum_{u=1}^t k_u (1 +|\mathcal{N}|) - \sum_{u=1}^t (k_u+ |\mathcal{T}_0^u|) \sum_{v=1}^{u-1} k_v} 
\\&\times \sum_{0=m_0 \leq m_1 \leq \cdots \leq m_{t} \leq m} \left(\prod_{u=1}^t q^{(k_u + |\mathcal{T}_0^u|) m_{u-1}} {m_u-m_{u-1}+|\mathcal{T}_0^u|-1 \brack m_u-m_{u-1}-k_u}_q\right) 
 q^{|\mathcal{N}|m_t} {m- m_t+|\mathcal{N}|-1 \brack m- m_t}_q .
\end{align*}
We deduce the final formula by using that
$$\sum_{u=1}^t (k_u+ |\mathcal{T}_0^u|) \sum_{v=1}^{u-1} k_v = \sum_{v=1}^{t} k_v \sum_{u=v+1}^t (k_u+ |\mathcal{T}_0^u|).$$
\end{proof}

Substituting Lemma \ref{lem:6grosqbin} in Proposition \ref{prop:6GS} leads to
\begin{align*}
G_{S,m}(q)&= q^{|\mathrm{min}_{\Delta}(S)| +m} \sum_{\substack{k_1, \dots, k_t: \\ k_u \leq |\mathcal{T}_1^u|}} \prod_{u=1}^t q^{k_u(u-2+k_u+|\mathcal{T}_0^u|)} {|\mathcal{T}_1^u| \brack k_u}_q
\\&\times\sum_{0=m_0 \leq m_1 \leq \cdots \leq m_{t}\leq m} \left(\prod_{u=1}^t q^{(k_u +|\mathcal{T}_0^u|)m_{u-1}} {m_u-m_{u-1}+|\mathcal{T}_0^u|-1 \brack m_u-m_{u-1}-k_u}_q \right)q^{|\mathcal{N}|m_t} {m-m_t+|\mathcal{N}|-1 \brack m-m_t}_q.
\end{align*}
Exchanging the summations, we obtain:
\begin{align}
\label{eq:6intermediaire}
\begin{aligned}
G_{S,m}(q)&= q^{|\mathrm{min}_{\Delta}(S)| +m} \sum_{0=m_0 \leq m_1 \leq \cdots \leq m_{t} \leq m} \left( \sum_{\substack{k_1, \dots, k_t: \\ k_u \leq |\mathcal{T}_1^u|}} \prod_{u=1}^t q^{k_u(u-2+k_u+|\mathcal{T}_0^u|)+(k_u +|\mathcal{T}_0^u|)m_{u-1}} \right.
\\& \left. \vphantom{\sum_{\substack{k_1, \dots, k_t: \\ k_u \leq |\mathcal{T}_1^u|}}} \times  {|\mathcal{T}_1^u| \brack k_u}_q {m_u-m_{u-1}+|\mathcal{T}_0^u|-1 \brack m_u-m_{u-1}-k_u}_q \right) q^{|\mathcal{N}|m_t} {m-m_t+|\mathcal{N}|-1 \brack m-m_t}_q.
\end{aligned}
\end{align}

We need one last lemma to complete our proof of Proposition \ref{prop:main'}.
\begin{lemma}
\label{lem:6LemmaX}
We have
\begin{align*}
&\sum_{0=m_0 \leq m_1 \leq \cdots \leq m_{t}} \sum_{\substack{k_1, \dots, k_t: \\ k_u \leq |\mathcal{T}_1^u|}} \prod_{u=1}^t q^{k_u(u-2+k_u+|\mathcal{T}_0^u|)+(k_u +|\mathcal{T}_0^u|)m_{u-1}} {|\mathcal{T}_1^u| \brack k_u}_q {m_u-m_{u-1}+|\mathcal{T}_0^u|-1 \brack m_u-m_{u-1}-k_u}_q 
\\&= \sum_{v=0}^t g_{v,t}(q;|\mathcal{T}_0^1|,\ldots,|\mathcal{T}_0^{t}|) {m_t+t-1 \brack m_t-v}_q,
\end{align*}
where $g_{v,t}$ was defined in Proposition \ref{prop:main'}.
\end{lemma}

Indeed, once Lemma \ref{lem:6LemmaX} is proved, we can write
\begin{align*}
&G_{S,m}(q)= q^{|\mathrm{min}_{\Delta}(S)| +m}  \sum_{v=0}^t g_{v,t}(q;|\mathcal{T}_0^1|,\ldots,|\mathcal{T}_0^{t}|) \sum_{0 \leq m_t \leq m} q^{|\mathcal{N}|m_t} {m-m_t+|\mathcal{N}|-1 \brack m-m_t}_q  {m_t+t-1 \brack m_t-v}_q
\\&=  q^{|\mathrm{min}_{\Delta}(S)| +m}  \sum_{v=0}^t g_{v,t}(q;|\mathcal{T}_0^1|,\ldots,|\mathcal{T}_0^{t}|) \sum_{0 \leq m'_t \leq m-v} q^{|\mathcal{N}|(m'_t+v)} {m-m'_t-v+|\mathcal{N}|-1 \brack m-m'_t-v}_q  {m'_t+v+t-1 \brack m'_t}_q,
\end{align*}
where the second equality follows from the change of variable $m'_t=m_t-v.$
Using Lemma \ref{lem:6lemmaY} with $t=2$, $m=m-v$, $\ell_1=v+t$, and $\ell_2= |\mathcal{N}|,$ this becomes
$$G_{S,m}(q)=  q^{|\mathrm{min}_{\Delta}(S)| +m}  \sum_{v=0}^t q^{v|\mathcal{N}|} g_{v,t}(q;|\mathcal{T}_0^1|,\ldots,|\mathcal{T}_0^{t}|) {m+t+|\mathcal{N}|-1 \brack m-v}_q.$$
Observing that $|\mathcal{N}|=s-t$ concludes the proof of 
Proposition \ref{prop:main'}. \qed

We conclude this section by the proof of Lemma \ref{lem:6LemmaX}.
\begin{proof}[Proof of Lemma \ref{lem:6LemmaX}]
Let us define $G_0(q;m)= \chi(m=0),$ and for $v\geq 1$,
\begin{align*}
&G_v(q;x_1,\dots,x_v;m):= 
\\&\sum_{0=m_0 \leq m_1 \leq \cdots \leq m_{v}=m} \sum_{\substack{k_1, \dots, k_v: \\ k_u \in \llbracket 0; 2-x_u \rrbracket }} \prod_{u=1}^v q^{k_u(u-2+k_u+x_u)+(k_u +x_u)m_{u-1}} {2-x_u \brack k_u}_q {m_u-m_{u-1}+x_u-1 \brack m_u-m_{u-1}-k_u}_q,
\end{align*}
so that the function in Lemma \ref{lem:6LemmaX} is $G_t(q;|\mathcal{T}_0^1|,\dots,|\mathcal{T}_0^t|;m_t).$

We show by induction on $v$ that 
\begin{equation}
\label{eq:6rec}
G_v(q;x_1,\dots,x_v;m)= \sum_{u=0}^v g_{u,v}(q;x_1,\dots,x_v) {m+v-1 \brack m-u}_q.
\end{equation}

Recall from \cite[p. 37, (3.3.10)]{Abook} that
\begin{equation}
\label{eq:6dernierqbin}
{a+b \brack c}_q = \sum_{a' \geq 0} {a \brack a'}_q {b \brack c-a'}_q q^{a'(b-c+a')}.
\end{equation}

By \eqref{eq:6dernierqbin} with $a=2-x_1$, $b=m+x_1-1$, and $c=m$, we have
\begin{align*}
G_1(q;x_1;m)&= {m+1 \brack m}_q
\\&= {m \brack m}_q + q{m \brack m-1}_q
\\&= g_{0,1}(q;x_1) {m \brack m}_q +g_{1,1}(q;x_1) {m \brack m-1}_q.
\end{align*}
So \eqref{eq:6rec} is true for $v=1$.

Now assume that it is true for $v-1 \geq 1$ and prove it for $v$.
We have
\begin{align*}
&G_v(q;x_1,\dots,x_v;m)= 
\\&\sum_{0=m_0 \leq m_1 \leq \cdots \leq m_{v}=m}  \prod_{u=1}^v \left( \sum_{ k_u =0}^{2-x_u} q^{k_u(u-2+k_u+x_u)+(k_u +x_u)m_{u-1}} {2-x_u \brack k_u}_q {m_u-m_{u-1}+x_u-1 \brack m_u-m_{u-1}-k_u}_q \right)
\\&= \sum_{m_{v-1}=0}^m \left( \sum_{0=m_0 \leq m_1 \leq \cdots \leq m_{v-1}}  \prod_{u=1}^{v-1} \left( \sum_{ k_u =0}^{2-x_u} q^{k_u(u-2+k_u+x_u)+(k_u +x_u)m_{u-1}} {2-x_u \brack k_u}_q {m_u-m_{u-1}+x_u-1 \brack m_u-m_{u-1}-k_u}_q \right) \right)
\\ & \qquad \times \sum_{ k_v =0}^{2-x_v} q^{k_v(v-2+k_v+x_v)+(k_v +x_v)m_{v-1}} {2-x_v \brack k_v}_q {m-m_{v-1}+x_v-1 \brack m-m_{v-1}-k_v}_q 
\\&= \sum_{m_{v-1}=0}^m G_{v-1}(q;x_1,\dots,x_{v-1};m_{v-1})
\sum_{ k_v =0}^{2-x_v} q^{k_v(v-2+k_v+x_v)+(k_v +x_v)m_{v-1}} {2-x_v \brack k_v}_q {m-m_{v-1}+x_v-1 \brack m-m_{v-1}-k_v}_q 
\\&= \sum_{m_{v-1}=0}^m \sum_{u=0}^{v-1} g_{u,v-1}(q;x_1,\dots,x_{v-1}) {m_{v-1}+v-2 \brack m_{v-1}-u}_q
\\&\qquad \times
\sum_{ k_v =0}^{2-x_v} q^{k_v(v-2+k_v+x_v)+(k_v +x_v)m_{v-1}} {2-x_v \brack k_v}_q {m-m_{v-1}+x_v-1 \brack m-m_{v-1}-k_v}_q,
\end{align*}
where we used the induction hypothesis in the last equality.

\noindent
Rearranging the order of summation leads to
\begin{align*}
G_v(q;x_1,\dots,x_v;m) &= \sum_{u=0}^{v-1} q^{ux_v} g_{u,v-1}(q;x_1,\dots,x_{v-1}) \sum_{ k_v =0}^{2-x_v} q^{k_v(v-2+u+k_v+x_v)} {2-x_v \brack k_v}_q
\\& \times\sum_{m_{v-1}=0}^m  q^{(k_v +x_v)(m_{v-1}-u)} {m_{v-1}+v-2 \brack m_{v-1}-u}_q
 {m-m_{v-1}+x_v-1 \brack m-m_{v-1}-k_v}_q.
\end{align*}

Using Lemma \ref{lem:6lemmaY} with $t=2$, $m=m-u-k_v$, $\ell_1= v-1+u$, and $\ell_2=k_v+x_v$, and the change of variable $x_1=m_{v-1}-u$, this yields:
\begin{align*}
G_v(q;x_1,\dots,x_v;m) &= \sum_{u=0}^{v-1} q^{ux_v} g_{u,v-1}(q;x_1,\dots,x_{v-1}) \sum_{ k_v =0}^{2-x_v} q^{k_v(v-2+u+k_v+x_v)} {2-x_v \brack k_v}_q
\\& \times 
 {m+v+x_v-2 \brack m-u-k_v}_q.
\end{align*}
Using \eqref{eq:6dernierqbin} again with $a=2-x_v$, $b=m+v+x_v-2$, $c=m-u$, and $a'=k_v$, we obtain
$$G_v(q;x_1,\dots,x_v;m) = \sum_{u=0}^{v-1} q^{ux_v} g_{u,v-1}(q;x_1,\dots,x_{v-1}) {m+v \brack m-u}_q.$$
By the $q$-analogue of Pascal's triangle, this becomes
\begin{align}
&G_v(q;x_1,\dots,x_v;m) \nonumber
\\&= \sum_{u=0}^{v-1} q^{ux_v} g_{u,v-1}(q;x_1,\dots,x_{v-1}) {m+v-1 \brack m-u}_q
+ \sum_{u=0}^{v-1} q^{ux_v+u+v} g_{u,v-1}(q;x_1,\dots,x_{v-1}) {m+v-1 \brack m-u-1}_q \nonumber\\
&= \sum_{u=0}^{v-1} \left( q^{ux_v} g_{u,v-1}(q;x_1,\dots,x_{v-1}) +q^{(u-1)x_v+u+v-1} g_{u-1,v-1}(q;x_1,\dots,x_{v-1}) \right) {m+v-1 \brack m-u}_q . \label{eq:finale}
\end{align}
Recall that
$$g_{u,v}(q;x_1, \dots, x_v) = \sum_{\substack{\epsilon_1,\dots , \epsilon_v \in \{0,1\}:\\ \epsilon_1 + \cdots + \epsilon_v = u}} q^{uv+ {u \choose 2}} \prod_{k=1}^v q^{(x_k -1) \sum_{i=1}^{k-1} \epsilon_i}.$$
So, separating the case where $\epsilon_v=0$ from the case where $\epsilon_v=1$, we have
\begin{align*}
g_{u,v}(q;x_1, \dots, x_v) &= \sum_{\substack{\epsilon_1,\dots , \epsilon_{v-1} \in \{0,1\}:\\ \epsilon_1 + \cdots + \epsilon_{v-1} = u}} q^{uv+ {u \choose 2}} \left(\prod_{k=1}^{v-1} q^{(x_k -1) \sum_{i=1}^{k-1}\sum_{i=1}^{k-1} \epsilon_i} \right) q^{(x_v-1)u}
\\&+\sum_{\substack{\epsilon_1,\dots , \epsilon_{v-1} \in \{0,1\}:\\ \epsilon_1 + \cdots + \epsilon_{v-1} = u-1}} q^{uv+ {u \choose 2}} \left(\prod_{k=1}^{v-1} q^{(x_k -1) \sum_{i=1}^{k-1}\sum_{i=1}^{k-1} \epsilon_i} \right) q^{(x_v-1)(u-1)}.
\end{align*}
After simplification, this is exactly \eqref{eq:finale}. The lemma is proved.

\end{proof}

\section*{Acknowledgements}
The research of the first author was supported by the project IMPULSION of IdexLyon. Part of this research was conducted while the second author was visiting Lyon, funded by that project.

We thank Leonard Hardiman and Jeremy Lovejoy for their helpful comments on earlier versions of this paper.

\bibliographystyle{alpha}
\bibliography{biblio}

\end{document}